\documentclass[12pt]{amsart}
\pdfoutput=1

\providecommand{\MR}{\relax\ifhmode\unskip\space\fi MR }

\providecommand{\href}[2]{#2}

\usepackage[top=30truemm,bottom=28truemm,left=28truemm,right=28truemm]{geometry}

\usepackage{amsmath}
\usepackage{mathrsfs}
\usepackage{amssymb}
\usepackage{amsfonts}
\usepackage{amsthm}
\usepackage{amscd}
\usepackage{graphicx}
\usepackage{bbold}
 \usepackage{ascmac}
 
  \graphicspath{{./graphics/}}
 
\usepackage{tikz}
\usetikzlibrary{matrix,arrows,decorations.pathmorphing}
\usepackage{tikz-cd}
\usetikzlibrary{arrows}
\tikzset{commutative diagrams/.cd,arrow style=tikz,diagrams={>=latex'}}

\mathchardef\-="2D 
\newcommand{\id}{\mathrm{id}}
\newcommand{\Z}{\mathbb{Z}} 
\newcommand{\R}{\mathbb{R}}

\newcommand{\F}{\mathcal{F}}
\newcommand{\Co}{\mathrm{Co}}

\newcommand{\srg}{\mathrm{Srg}}

\newcommand{\leftb}{\partial_{L}}
\newcommand{\rightb}{\partial_{R}}
\newcommand{\bottomb}{\partial_{B}}
\newcommand{\topb}{\partial_{T}}
\newcommand{\V}{\mathcal{V}}
\newcommand{\Hom}{\mathrm{Hom}}
\newcommand{\1}{\mathbb{1}}
\newcommand{\D}{\mathcal{D}}
\newcommand{\col}{\mathrm{col}}

\newcommand{\T}{\mathcal{T}}
\newcommand{\mir}{\mathrm{mir}}
\newcommand{\refl}{\mathrm{ref}}
\newcommand{\bob}{\partial_{-}}
\newcommand{\tob}{\partial_{+}}

\newcommand{\Int}{\mathrm{Int}}
\newcommand{\Vect}{\mathrm{Vect}}

\newcommand{\Bimod}{\mathrm{Bimod}}
\newcommand{\Cat}{\mathrm{Cat}}

\newcommand{\obj}{\mathrm{Obj}}
\newcommand{\Ob}{\mathrm{Ob}}

\newcommand{\catC}{\mathscr{C}}
\newcommand{\catD}{\mathscr{D}}

\newcommand{\nstand}[1]{S^{\sqcup {#1} }}
\newcommand{\stand}{S^{\sqcup n}}
\newcommand{\lempty}{{_* \emptyset}}
\newcommand{\rempty}{\emptyset_*}
\newcommand{\KV}{2\-\mathrm{Vect}}
\newcommand{\X}{\mathcal{X}}
\newcommand{\bfk}{\mathbf{k}}
\newcommand{\M}{\mathcal{M}}
\newcommand{\Rib}{\mathrm{Rib}_{\V}}
\newcommand{\End}{\mathrm{End}}
\newcommand{\Fill}{\mathrm{Fill}}
\newcommand{\cFill}{\mathrm{Fill}_{\mathrm{c}}}

\newcommand{\cdotv}{\cdot_{\text{v}}}
\newcommand{\circh}{\circ_{\text{h}}}

\numberwithin{equation}{section}
\numberwithin{figure}{section}
\theoremstyle{plain}
\newtheorem{thm}{Theorem}[section]

\theoremstyle{definition}
\newtheorem{Definition}[thm]{Definition}

\newtheorem{lemma}[thm]{Lemma}
\newtheorem{prop}[thm]{Proposition}
\begin{document}
\title[A TQFT extending the RT Theory]{A 2-categorical extension of the Reshetikhin--Turaev theory}
\author{Yu Tsumura}

\email{ytsumura@math.purdue.edu}
\address{Department of Mathematics, Purdue University, West Lafayette IN 47907, USA}

\begin{abstract}
We concretely construct a 2-categorically extended TQFT that extends the Reshetikhin-Turaev TQFT to cobordisms with corners.
The source category will be a well chosen 2-category of decorated cobordisms with corners and the target bicategory will be the Kapranov-Voevodsky 2-vector spaces.
\end{abstract}

\maketitle

\section{Introduction}
One of the great breakthroughs in the understanding of physical theories was the construction of Reshetikhin--Turaev
(2+1)-dimensional topological quantum field theories \cite{Turaev10, MR1091619}. 
Prior to this, Atiyah axiomatized a TQFT in \cite{MR1001453}.
The simpler (1+1)-dimensional theory was
nicely formulated by  \cite{Dijkgraafthesis}. 
The latter construction was lifted to a conformal field theory by Segal \cite{MR981378}.
In all these cases, one has a functor from a cobordism category to an algebraic category. 

Going back to Freed and Quinn \cite{MR1240583}, Cardy and Lewellen \cite{MR1107480} ,
there has been an interest in including boundary conditions/information. 
Besides the physical challenges, this poses a mathematical problem as both the geometric and the algebraic category need to be moved into higher categories.
Naively, a cobordism with corners is a 2-category, by viewing the corners as objects, the boundaries as cobordisms between them and the cobordism as a cobordism of cobordisms. 
The devil is of course in the details. There have been several
approaches to the problem such as Lurie's approach \cite{Lurie2009}, \cite{MR2648901}, \cite{MR2713992}
and the project \cite{DSS}.

Taking a step back to the (1+1)-dimensional situation, the TQFTs with boundaries have been nicely characterized and
given rise to new axioms such as the Cardy axiom, see e.g.\ \cite{MR2395583} for a nice introduction or \cite{MR2242677} for a model free approach. 
Here the objects are not quite cobordisms with corners, but more simply surfaces with boundaries
and points on the boundary.

In this paper we will give a constructive solution to the problem by
 augmenting the setup of the Reshetikhin--Turaev theory.
 We give very careful treatment and check all the details.
In order to do this we use an algebraic and a geometric 2-category and define a 2-functor (with anomaly) between them.
The source geometric 2-category $\Co$ is constructed extending the category of decorated cobordisms  of the RT TQFT. 
The target algebraic (weak) 2-category is the Kapranov-Voevodsky 2-vector spaces $\KV$ defined in \cite{KV1994}.
%

Great caution has to be used in the definition of gluing of decorated cobordisms with corners.
As objects and  1-morphisms, we fix the \textit{standard circles} and the \textit{standard surfaces}.
Thus, the 2-category $\Co$ on the level of objects and 1-morphisms is combinatoric.
The topological nature of $\Co$ lies in the 2-morphism level.
A 2-morphisms of $\Co$ is a decorated 3-manifold with corners.
One of the property of a decorated 3-manifold is that the boundary of the manifold is parametrized.
Namely, there are several embeddings from the standard surfaces to the boundary of the manifold.
This is where we need to be careful.
The gluing of standard surfaces is not a standard surface but homeomorphic to it.
Therefore to define horizontal composition of 2-morphisms, we need to choose and fix a homeomorphism between these spaces with caution.

In order to construct a 2-functor from $\Co$ to $\KV$, we "cap off" the corners of a cobordism to reduce it to a cobordism without corners.
Then we apply the original RT TQFT.
To prove the functoriality of this 2-functor, we develop a technique of representing cobordisms by special ribbon graphs and reduce calculations on manifolds to calculations on special ribbon graphs.

Our work is intended as a bridge between several subjects. The presentation of the categories was chosen to match with the  considerations of string topology \cite{MR2597733} in the formalism of \cite{MR2314100,MR2411420}, which will hopefully lead
to some new connections between the two subjects. 

Once the classification of \cite{DSPVB} is available, it will be interesting to see how our concrete realization of 3-2-1 TQFT fits.
Their results and those of \cite{DSS} are for a different target category.
We will provide a link in Section \ref{sec:comments}.
Further studies of 3-2-1 extensions but for Turaev--Viro are in  \cite{Turaev2010} and in \cite{BalsamI, BalsamII, BalsamIII}.
It would be interesting to relate these to our construction using the relationship between the RT and the TV invariants \cite{Turaev2010}.
 
The paper is organized as follows.
In Section \ref{sec:A 2-category of cobordisms with corners}, we will introduce the 2-category $\Co$.
The 2-category $\Co$ is our choice of a 2-category of decorated cobordisms with corners.
In Section \ref{sec:A 2-category of the Kapranov-Voevodsky 2-vector spaces}, we will recall the Kapranov-Voevodsky 2-vector spaces 2-$\Vect$ as the target 2-category of the extended TQFT.
Before we will start the discussion of an extension of the RT TQFT to the cobordisms with corners, we will review some of the original construction of the Reshetikhin-Turaev TQFT in Section \ref{sec:Review and Modification of the Reshetikhin-Turaev TQFT} since we will extensively use the original theory.
This will also serve as a quick reference of the notations and definitions of \cite{Turaev10}.
In Section \ref{sec:An extended TQFT}, we construct an extended TQFT $\X$ from $\Co$ to $\KV$ and all the details of several compatibility of gluings, which is the main part of the paper, will be proved in Section \ref{sec:Main Theorem}.
We also show that this is indeed an extension by showing when it is restricted to regular cobordisms it produce the RT TQFT.
In Appendix, we review B\'{e}nabou's definition of a bicategory and a pseudo 2-functor \cite{MR0220789} and extend it to the definition of \textit{projective pseudo 2-functor}.

\section*{Acknowledgments}
I would like to show my greatest appreciation to Professor Ralph M.\ Kaufmann whose advice, guidance, suggestions were of inevitable value for my study.
I also owe a very important debt to Professor Alexander A.\ Voronov and Professor Christopher Schommer-Pries for their interest in the current paper and valuable feedback.
I would like to thank the referee for carefully reading my manuscript and for giving such constructive comments which substantially helped improving the quality of the current paper.

\section{A 2-category of cobordisms with corners}\label{sec:A 2-category of cobordisms with corners}
In this section we define a 2-category of decorated cobordisms with corners $\Co$ which is the source 2-category of our extended TQFT as a projective pseudo 2-functor.
Let us explain the outline of our construction of a 2-category of decorated cobordisms with corners $\Co$.
We will give the precise definitions later.
The objects are standard circles.
The 1-morphisms are standard surfaces with boundaries.
The 2-morphisms are decorated 3-manifolds with corners with parametrized boundaries.
In the literature, there are many kind of definitions of a 2-category of cobordism with corners but as far as the author knows there has not been a definition of 2-category of decorated cobordisms with corners.
(Remark: We found that the definition given by Kerler and Lyubashenko \cite{MR1862634}  is close to our definition. 
They use a double category instead of a 2-category.)
One of the difference between our 2-category of cobordisms with corners from others is that 2-morphism cobordisms are parametrized.
This means that we fix standard surfaces and each cobordism is equipped with homeomorphisms from standard surfaces to its boundary components.
The difficulty with standard surfaces is that the composite of two standard surfaces is not a standard surface.
Thus we need to deal with compositions carefully.

Even though we choose to use standard circles and standard surfaces as our objects and 1-morphisms of $\Co$, the essence is combinatorial.
Namely, we only need the number of components of circles and \textit{decorated type} of a surface.
 Topological information lives in the level of 2-morphisms.
 Thus our definition of $\Co$ can be regarded as a geometric realization of combinatorial data on objects and 1-morphisms. See the table below.
 
 \begin{center}
    \begin{tabular}{ | l | l | l | }
    \hline
           $\Co$ & Geometric realization & Combinatorial data   \\ \hline
   Objects  & Standar circles & Integers
    \\ \hline
   1-morphisms & Standard surfaces & Decorated types   \\ \hline
    2-morphisms & Classes of decorated cobordisms with corners &    \\
    \hline
    \end{tabular}
\end{center}

In our setting, surfaces are restricted to connected ones. 
This restriction makes the theory simpler and fit rigorously into a 2-categorical setting.
Also, to avoid  non-connected surfaces, we introduce two formal objects $_*\emptyset$ and $\rempty$, which we call the left and the right empty sets.
As a just 2-category of cobordisms with corners, including non-connected surfaces is natural and simpler. 
However including general non-connected surfaces makes  it complicated to construct an extended TQFT.
In fact, our technique of representing cobordisms by special ribbon graphs does not generalize to the case of non-connected surfaces.
For this reason, the non-connected surfaces will be dealt with in a future paper.

Now we are going to explain the rigorous definitions.
Along with doing so, we need to modify and extend several definitions used in the Reshetikhin-Turaev TQFT.
We define data $\Co$ consisting of objects, 1-morphisms, and 2-morphisms.
It will be shown that the data $\Co$ is indeed a 2-category.
\subsection{Objects of $\Co$} 
Let us consider the 1-dimensional circle $S^1=\{(x, y) \in \R^2 \mid x^2+y^2=1\}$ in $\R^2$. 
We call the pair $(S^1, (0,-1))$ the \textit{ 1-dimensional standard circle}.
We often omit the point $(0, -1)$ in the notation and just write $S^1$.
For each natural number $n$, the ordered disjoint union of $n$  1-dimensional standard circles is called the \textit{ $n$-standard circles}. 
We denote $n$-standard circle by $\stand:=(S^1, i_1=(0,-1))\sqcup \cdots \sqcup (S^1, i_n=(0,-1))$, where $i_k$  are called the \textit{$k$-th base point} for $k=1, \dots, n$.
In general, a pair of a connected manifold and a point of the manifold is called a \textit{pointed} manifold and the specified point is called the \textit{base point}.
The disjoint union of several pointed manifolds is called a multi-pointed manifold.
Thus the $n$-standard circles is a multi-pointed manifold.
\begin{Definition}
We define an \textit{object} of the data $\Co$ to be the $n$-standard surface $\stand$ for each natural number $n$. We also include two formal symbols $\rempty$ and $\lempty$. We call them the left and the right empty set, respectively.
These two formal symbols are needed to confine ourselves to connected surfaces.
\end{Definition}

\subsection{1-Morphisms of $\Co$}
In the previous section we defined the objects of $\Co$.
Now we define a 1-morphism between two objects of $\Co$. 
A 1-morphism will be defined to be a standard surface, which we define below.
First, we define decorated types and decorated surface needed to define the standard surfaces.
%

\subsubsection{Decorated types and decorated surfaces}
Already Reshetikhin-Turaev's construction uses surfaces that are decorated by objects of a modular tensor category.
We extend their definition to include surfaces with boundaries.
Fix a modular  category $\V$.
Let
\begin{equation}\label{equ:type}
t=(m,n; a_1, a_2, \dots, a_p)
\end{equation}
be a tuple consisting of non-negative integers $m$, $n$, and for $i=1, \dots, p$, $a_i$ is either a non-negative integer or a pair $(W, \nu)$, where $W$ is an object of the modular  category $\V$ and $\nu$ is either $1$ or $-1$.
Such a pair is called a \textit{mark} or a \textit{signed object} of $\V$.
The tuple $t$ is called a  \textit{decorated type} or when confusion is unlikely we simply call it \textit{type}.
Let $L(t)=m$ and $R(t)=n$ denote the first and the second integer of the type $t$, respectively.

By an \textit{arc} on a surface $\Sigma$, we mean a simple oriented arc lying $\Sigma \setminus \partial \Sigma$.
An arc on $\Sigma$ endowed with an object $W$ of $\V$ and a sign $\nu=\pm 1$ is said to be \textit{marked}.
A connected compact orientable surface $\Sigma$ is said to be \textit{decorated} by a decorated type $t=(m,n; a_1, a_2, \dots, a_p)$ if the following conditions are satisfied.
\begin{enumerate}
\item There are $m+n$ boundary components of $\Sigma$ and the boundary components are totally ordered. 
The first $m$ components are called \textit{inboundary} or \textit{left boundary} and the last $n$ components are called \textit{outboundary} or \textit{right boundary}.
\item The boundary $\partial \Sigma$ is a multi-pointed manifold.
\item For each singed object entry of $a_i$, the surface $\Sigma$ is equipped with a marked arc with the mark $a_i$.
\item The genus of $\Sigma$ is the sum of the all integer components $a_i$ except for the first and the second integers.
\end{enumerate}

A \textit{$d$-homeomorphism} of decorated surfaces is a degree 1 homeomorphism of the underlying surfaces preserving the order of boundary components, base points, orientation, the distinguished arcs together with their orientations, marks, and order.

There is a natural \textit{negation} of the structure on a decorated surface.
First, for a type $t=(m,n; a_1, a_2, \dots, a_p)$ we define its opposite type $-t=(m,n; b_1, b_2,\dots, b_p)$ as follows.
If $a_i$ is an integer entry, then let $b_i=a_i$.
If $a_i=(W,\nu)$ is a mark, then let $b_i=(W, -\nu)$.
For a decorated surface $\Sigma$, its opposite decorated surface $-\Sigma$ is obtained from $\Sigma$ by reversing the orientation of $\Sigma$, reversing the orientation of its distinguished arcs, and multiplying the signs of all distinguished arcs by $-1$ while keeping the labels and the order of these arcs.
Note that the decorated type of $-\Sigma$ is the opposite type of $\Sigma$.

\subsubsection{Remark}\label{subsec:Remark}
In the Reshetikhin-Turaev theory for the cobordisms without corners, a decorated type is denoted by
\[t_{\text{RT}}=(g; (W_1, \nu_1), \dots, (W_m, \nu_m)),\]
where $g$ is an integer indicating a genus and $W_i$ is an object of a modular category and $\nu_i$ is either $1$ or $-1$ for $i=1, \dots, m$.
In our notation, this decorated type is expressed by the type
\[t=(0, 0; (W_1, \nu_1), \dots, (W_m, \nu_m), 1,1,\dots, 1),\]
where the number of $1$'s is $g$.
Thus our theory includes the RT theory.
\subsubsection{Standard surfaces}

For each type $t$ we define the standard surface of type $t$. 
%
Let $t=(m,n; a_1, a_2, \dots, a_p)$ be a decorated type.
To construct the standard surface, we first define a ``block" of a ribbon graph, which can be thought of an elementary core of the standard surface.
For a mark $a=(W, \nu)$, the block for $(W, \nu)$ is defined to be a $1\times 1$ square coupon in $\R^2 $ and a length 1 band attached to the top of the coupon and the band is colored by $W$ if $\nu=1$ and $W^*$ if $\nu=-1$.
See Figure \ref{fig:blockmark}.

\begin{figure}[h]
\center
\includegraphics[width=1.2in]{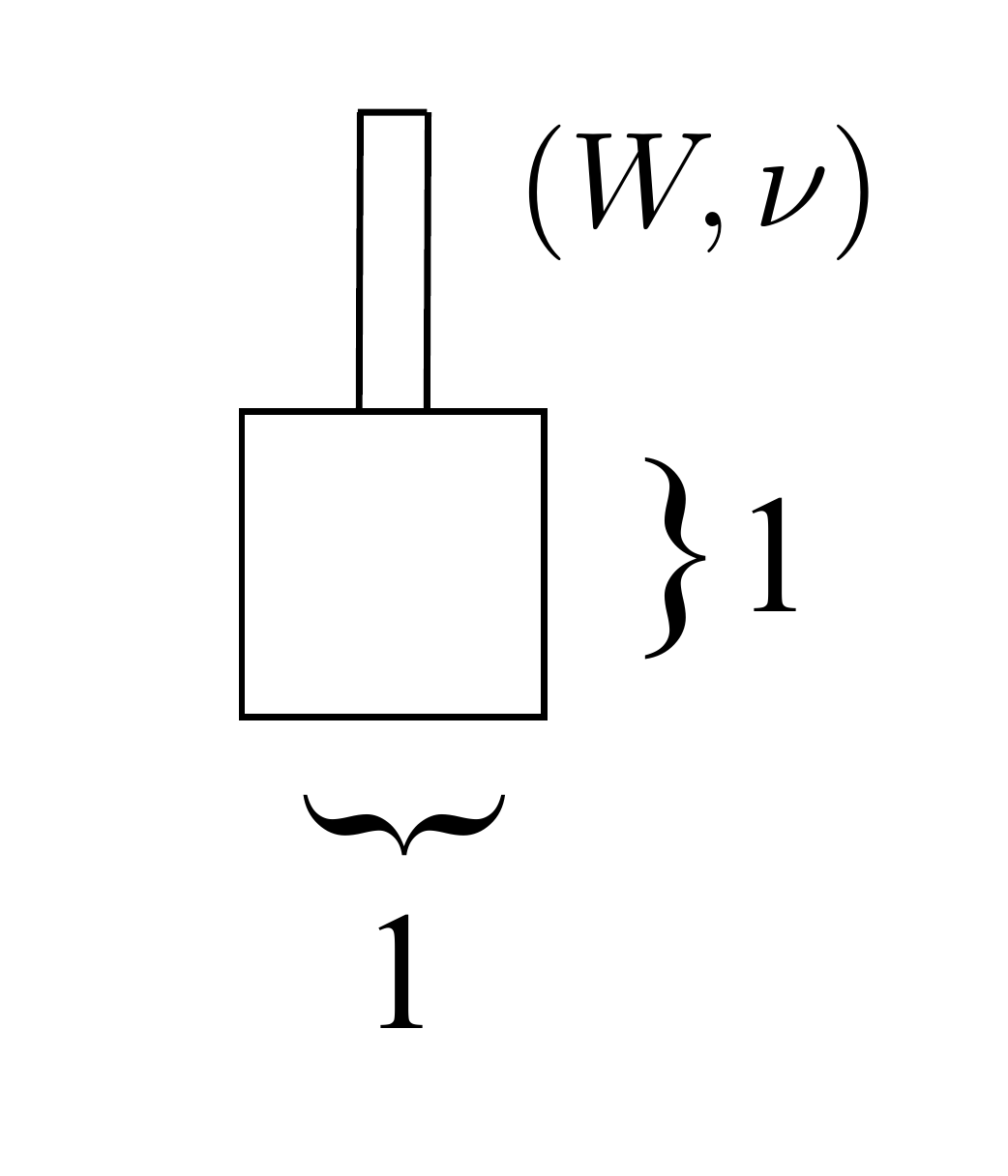}
\caption{The block for $(W,\nu)$}
\label{fig:blockmark}
\end{figure}
The block for a positive integer $a$ consists of a $1 \times 1$ square coupon in $\R^2$ and rainbow like bands with $a$ bands on the top of the square.
These bands are not colored and their cores are oriented from right to left. 
See Figure \ref{fig:blockrainbow}.

\begin{figure}[h]
\center
\includegraphics[width=2.2in]{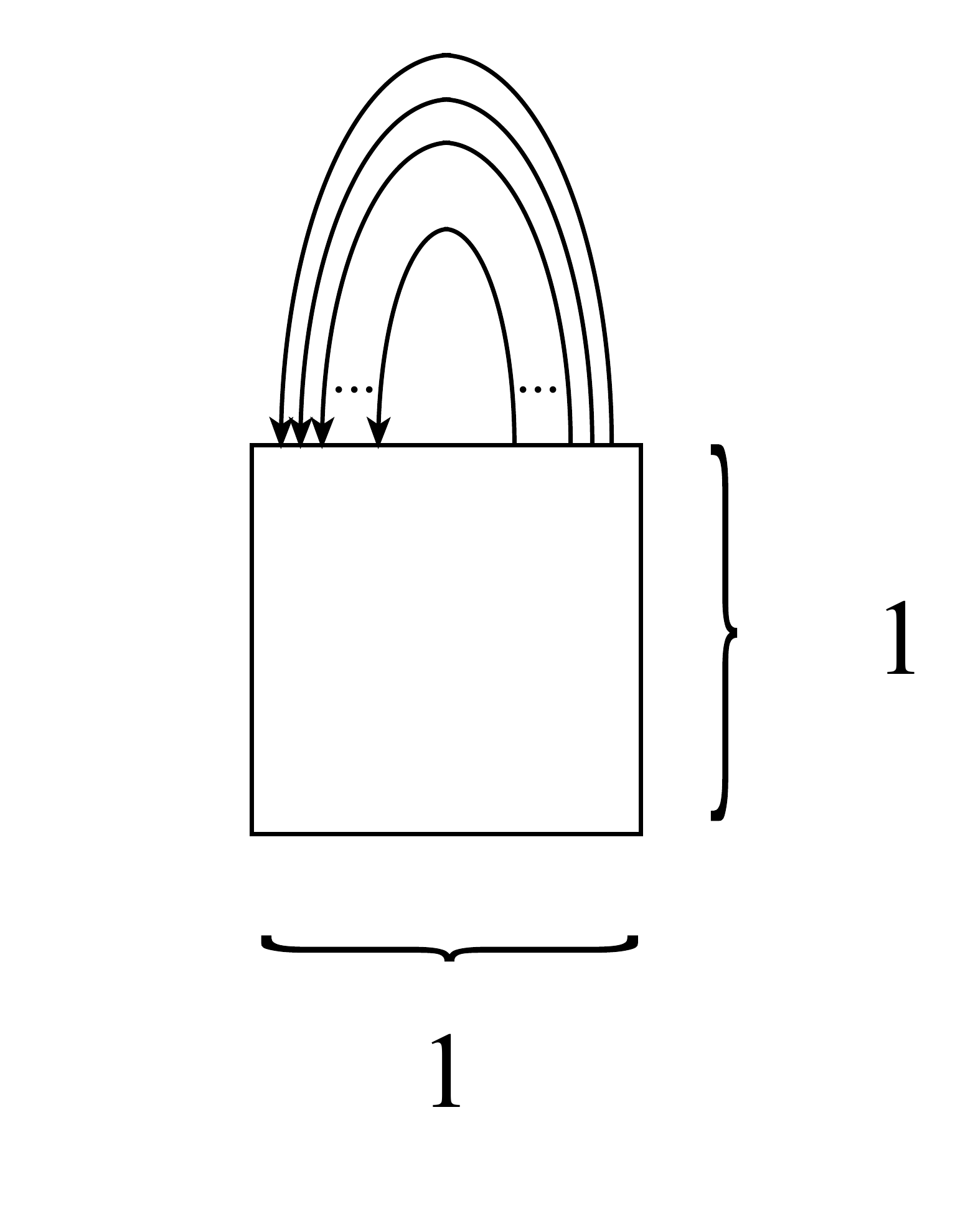}
\caption{The block for an integer}
\label{fig:blockrainbow}
\end{figure}

For the first entry integer $m$ of the type $t$, the block for $m$ is defined as in the left figure of  Figure \ref{fig:side ribbons}.
There are $m$ bands attached to the top of the square and the bands are bent so that the ends of bands have the same $x$-coordinates as in the figure.
Similarly, for the second entry integer $n$ of the type $t$, the block for $n$ is defined as in the right figure of Figure \ref{fig:side ribbons}.
For each integer, the left and the right ribbon graphs in Figure \ref{fig:side ribbons} are mirror reflection with respect to $y$-axis.

\begin{figure}[h]
\center
\includegraphics[width=2.2in]{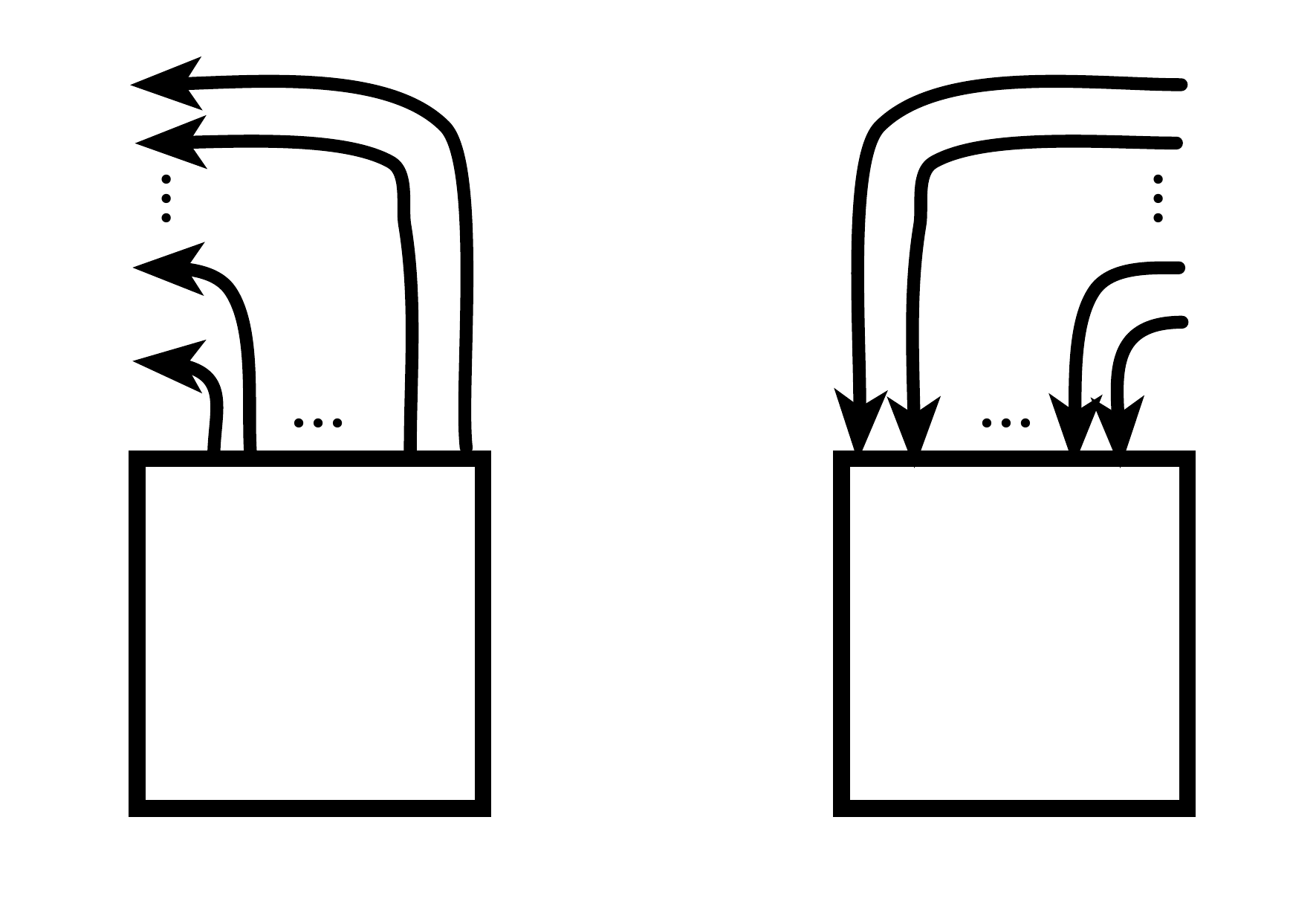}
\caption{The block for the fist and the second integers}
\label{fig:side ribbons}
\end{figure}

Now let $R_t$ be a ribbon graph in $\R^3$ constructed by arranging, in the strip $\R \times 0 \times [-1, 1] \subset \R^3$, 
the block for $m$ so that the top left corner is at $(0,0,0)$ and
the block for $a_i$  so that the top left corner of the square of the block is located at $(i,0,0)$ for $i=1,\dots, p$ and the block for $n$ so that the top left corner is at $(p+1,0,0)$.
We delete the joint segments of the coupons and make it a single coupon with length $p+2$.
Let $R_t$ denote the resulting ribbon graph.
See Figure \ref{fig:Rtnew} for an example of $R_t$ with the type 
\[t=(2,3; (W_1, \nu_1), 1, (W_2, \nu_2),3, (W_3, \nu_3),2).\]

\begin{figure}[h]
\center
\includegraphics[width=4in]{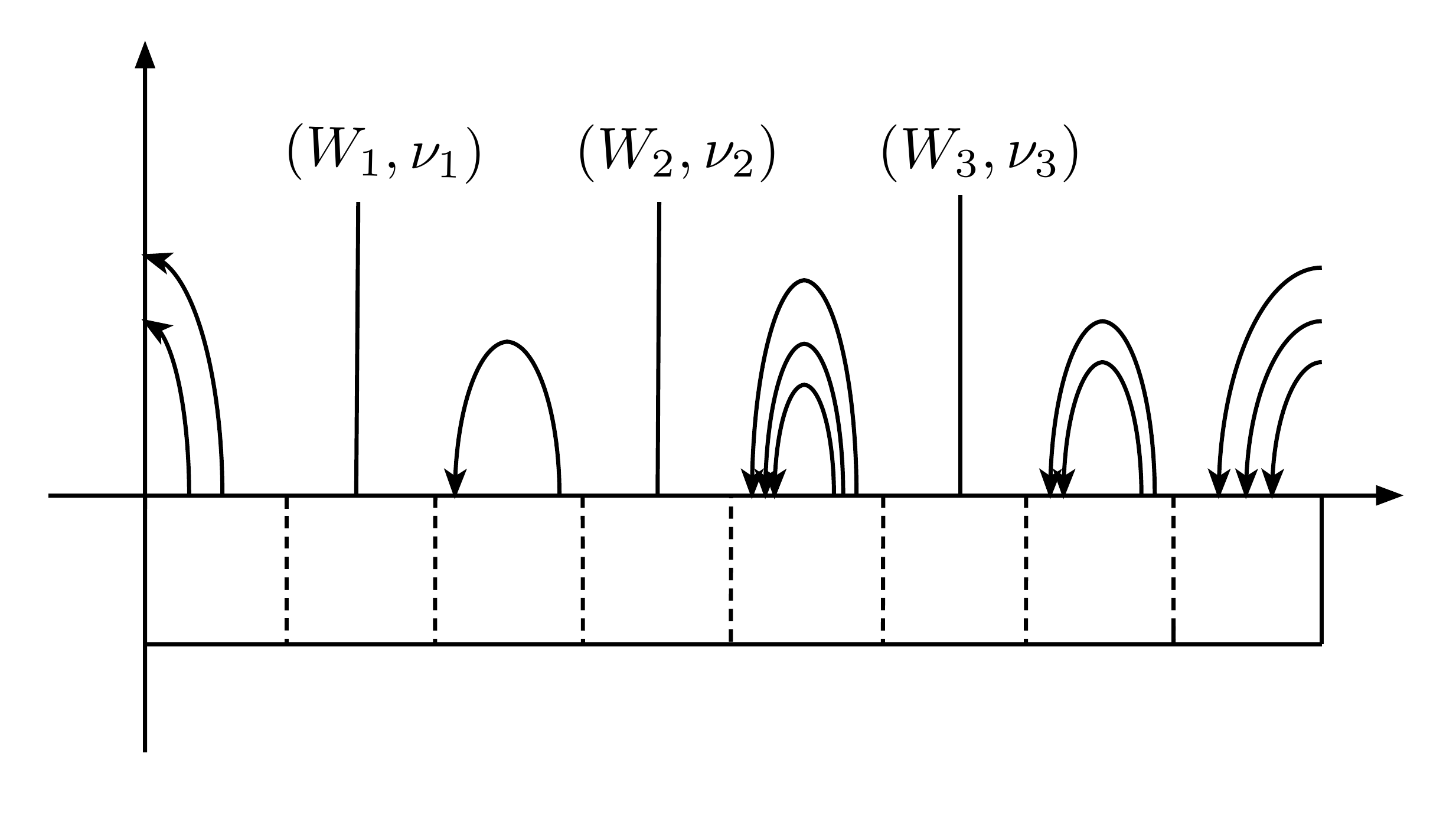}
\caption{The ribbon graph $R_t$}
\label{fig:Rtnew}
\end{figure}

Let $l$ be the number of entries in the type $t$, which is the width of the coupon in $R_t$.
Fix a close regular neighborhood $U_t$ of the ribbon graph $R_t$ in the stripe $[0,l]\times \R \times [-2,1] \subset \R^3$.
We provide $U_t$ with right-handed orientation and provide the boundary surface $\partial U_t$ with the induced orientation.
We assume by shrinking the coupon slightly so that the graph $R_t$ intersect with $\partial U_t$ only at the ends of short bands.
If a band has a mark $(W_i, \nu_i)$, provide the intersection arc with this mark.
The surface $\partial U_t$ with these intersection arcs with marks and $m+n$ non-marked arcs is called the \textit{capped standard surface} for the type $t$ and denoted by $\hat{\Sigma}_t$.
Fix an embedding of the disjoint union of $m+n$ 2-dimensional  disks $D^2$'s into $\hat \Sigma_t$ so that each boundary circle enclose exactly one non-marked arc of $\hat \Sigma_t$.
Each image of $[-1/2,1/2]\subset D^2$ is one of the arcs.
Cutting out the image of the interior of these disks we obtain a surface with marked arcs and boundary. 
Each boundary component has a base point which is an image of $(0, -1)$.
We assume that the intersection of planes $\{0\}\times \R^2$ and $\{l\}\times \R^2$ with $U_t$ are those embedded disks.
The resulting surface is called the \textit{standard surface} of type $t$ and denoted by $\Sigma_t$.
The boundary components of $\Sigma_t$ corresponding to the uncolored left $m$ bands are called the left boundary and denoted by $\leftb \Sigma_t$ and the boundary components corresponding to the uncolored right $n$ bands are called the right boundary and denoted by $\rightb \Sigma_t$.
The left boundary circles are ordered according to the order of the left bands ordered from left to right.
The right boundary circles are ordered according to the order of the right bands ordered from right to left.
The 3-manifold $U_t$ with the ribbon graph $R_t$ sitting inside $U_t$ is called the \textit{standard handle body} for type $t$.
\begin{figure}[h]
\center
\includegraphics[width=3in]{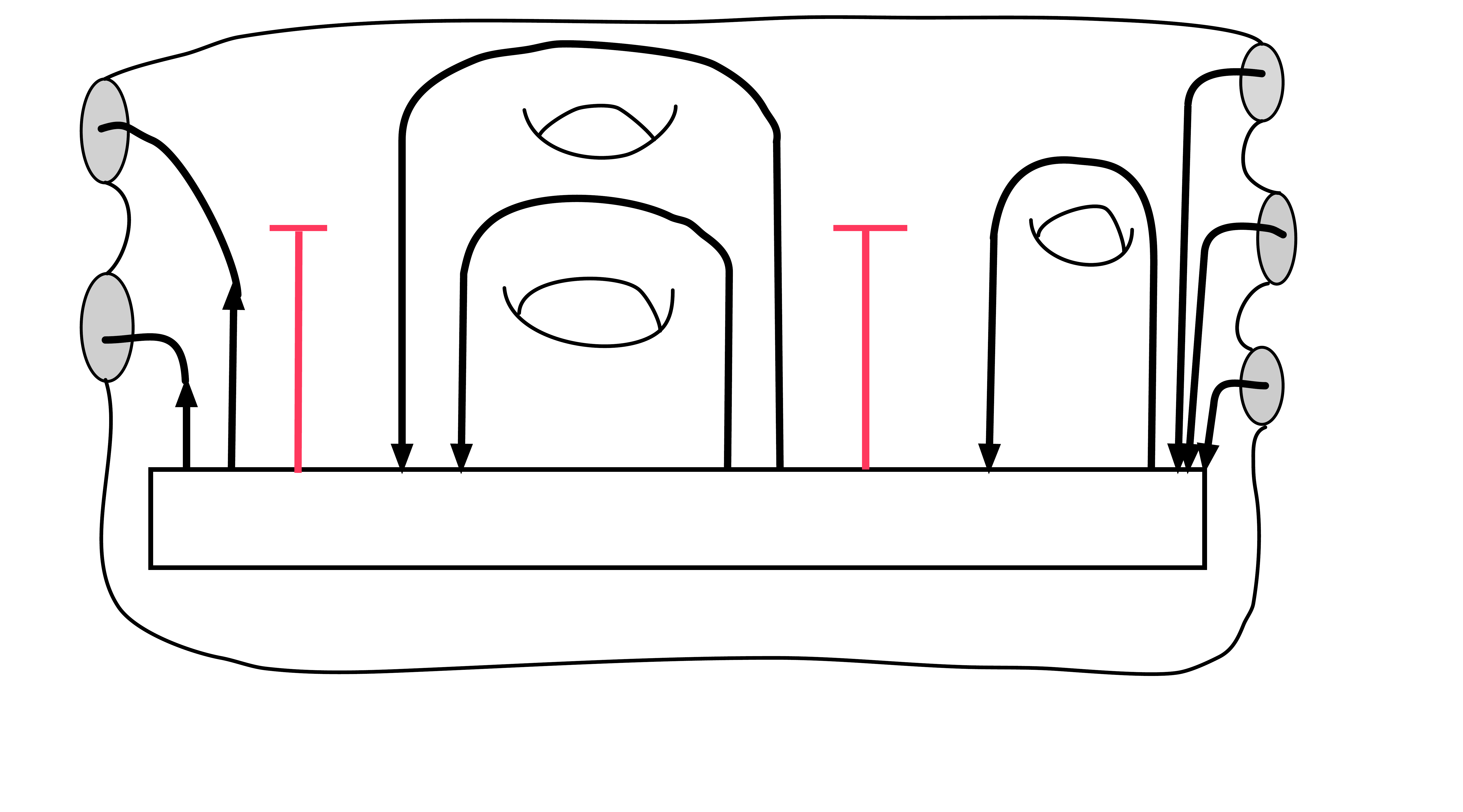}
\caption{The standard handle body and embedded disks}
\label{fig:capped standard handlebody}
\end{figure}

Analogously, consider the mirror reflection $-R_t:=\mir(R_t)$, where $\mir:\R^3 \to \R^3$ is a reflection with respect to a plane $\R^2\times \{1/2\} \subset \R^3$.
Set $U_t^-=\mir(U_t)$. 
We provide $U_t^-$ with right-handed orientation and provide $\partial(U_t^-)$ with the induced orientation.
For the $i$-th arc of the intersection $-R_t \cap \partial (U_t^-)$, we assign marks $(W_i, -\nu_i)$. 
Set $\Sigma_t^-:=\partial U_t^-$.

If we confine ourselves to closed surfaces, the definition of the standard surfaces are minor modification of that of Turaev's.
For our purpose, we need to consider gluings of surfaces along boundaries.
Thus we need to deal with the composition of these data we defined.
Two types $t=(l,m; a_1, a_2, \dots, a_p)$ and $s=(m',n; b_1, b_2, \dots, b_q)$ are said to be \textit{composable} if $m=m'$.
If they are composable, the composition of $t$ and $s$ is defined to be
\begin{equation}\label{equ:composition of types}
t\circ s =(l,n; a_1, a_2, \dots, a_p, m-1, b_1, b_2,\dots, b_q).
\end{equation}


As we need it later, we also define $D_n$ (Figure \ref{fig:Dn}) to be the disjoint union of $n$ cylinder $D^2\times [0, 1]$, where $D^2=\{ (x,y)\in \R^2 \mid x^2+y^2 \leq 1 \}$, with an uncolored untwisted band $[-1/2, 1/2] \times [0, 1]$ in each cylinder  that only intersects with the boundary of the cylinder at the bottom disk $D^2 \times \{0\}$ and the top disk $D^2 \times \{1\}$ transversally. 
Let $C(n)=\sqcup_n \partial (D^2) \times [0,1]$. 
The space $C(n)$ is the boundary of $D_n$ minus the interior of the union of the top boundary $\sqcup_n \partial(D^2) \times \{1\}$ and the bottom boundary $\sqcup_n \partial(D^2) \times \{0\}$.
The points in the boundary of $C(n)$ corresponding to the point $(0,-1) \times \{0\}$ and $(0,-1)\times \{1\}$ in $D^2 \times [0,1]$ are base points of $C(n)$.
We provide $D_n$ with right-handed orientation and provide the boundary surface $C(n)$ with the induced orientation.
Let $\mathrm{ref}:D_n \to D_n$  be an orientation reversing homeomorphism  that is induced by the map sending $(x, y)\times \{t\}$ to $(-x, y) \times \{t\}$ in $D^2 \times [0, 1]$.
Thus the map $\mathrm{ref}$ is a reflection map with respect to $y$-$z$ plane in $\R^3$.
Restricting on $C(n)$, the map $\mathrm{ref}$ induces an orientation reversing map on $C(n)$, which is also denoted by $\mathrm{ref}$.
\begin{figure}[h]
\center
\includegraphics[width=4.4in]{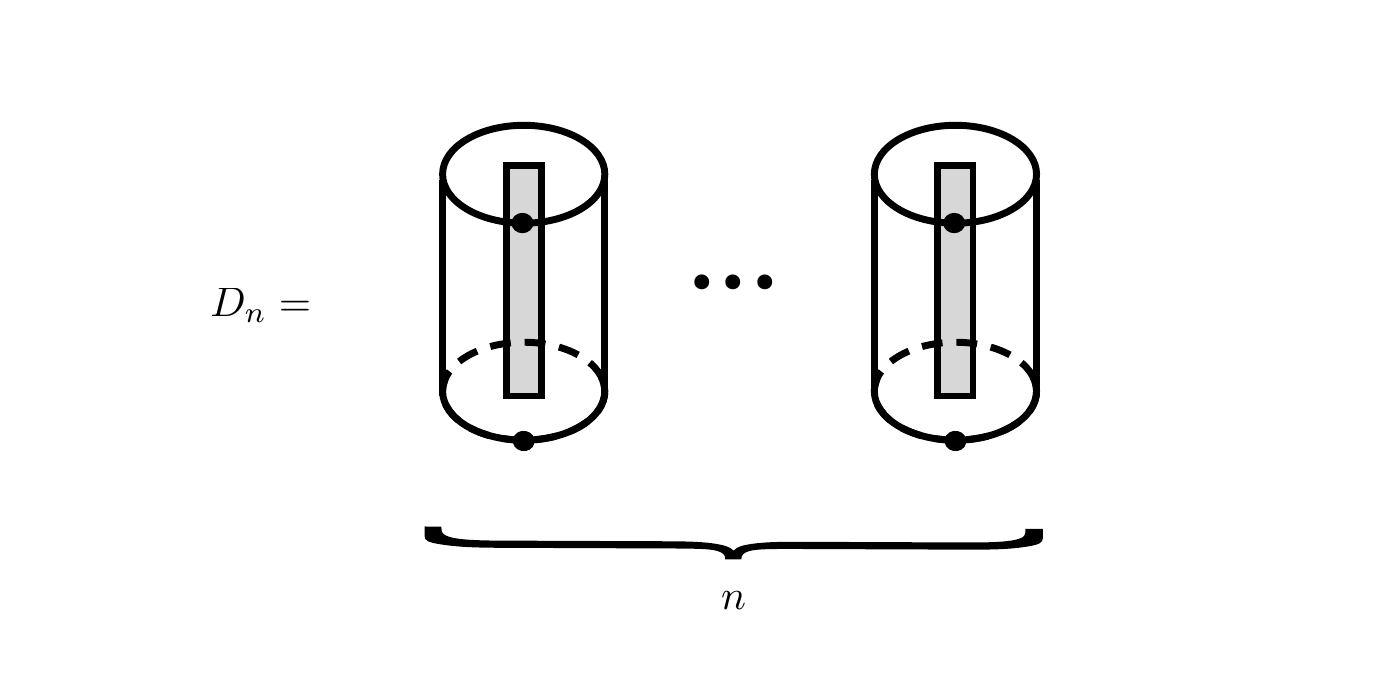}
\caption{The cylinder $D_n$}
\label{fig:Dn}
\end{figure}

\begin{Definition}[1-morphisms of $\Co$]\label{def:1-morphism of Co}
Let $X$ and $Y$ be objects of $\Co$.
A 1-morphism from $X$ to $Y$ is defined to be the standard surface $\Sigma_t$ for a decorated type $t$ depending on $X$ and $Y$ as follows.
\begin{enumerate}
\item If $X=\nstand{m}$ and $Y=\nstand{n}$, then $t=(m, n;a_1, a_2, \dots, a_p)$.
\item If $X=\lempty$ and $Y=\nstand{n}$, then $t=(0, n;a_1, a_2, \dots, a_p)$.
\item If $X=\nstand{m}$ and $Y=\rempty$, then $t=(m, 0;a_1, a_2, \dots, a_p)$.
\item If $X=\lempty$ and $Y=\rempty$, then $t=(0, 0;a_1, a_2, \dots, a_p)$. 
\end{enumerate}
We add formal identity symbols $\id_n:\nstand{n} \to \nstand{n}$ in the set of 1-morphism for each natural number $n$.
If we agree with the convention that  the source  object $X=\nstand{0}$ denotes $\lempty$ and  the target object $Y=\nstand{0}$ denotes $\rempty$, then the definition (2)-(4) are special cases of (1).

\end{Definition}

We will explain the role of the formal symbol $\id_n$ later when we discuss compositions of $\Co$.
\subsection{2-morphisms of $\Co$}
A 2-morphism of $\Co$ will be an equivalence class of a \textit{decorated} cobordism, which we are going to define.
Let
 \[ t=(m,n; a_1, a_2, \dots, a_p) \mbox{ and } s=(m,n; b_1, b_2, \dots, b_q)\] be types.
Let $\Sigma_{t}$ and $\Sigma_{s}$ be 1-morphisms from $\nstand{m}$ to $\nstand{m}$.
We define a \textit{decorated cobordism with corner} from $\Sigma_{t}$ to $\Sigma_{s}$ as follows.
Consider a compact oriented 3-manifold $M$ whose boundary decomposes into four pieces as 
\[\partial M= \bottomb M \cup \topb M \cup \leftb M \cup \rightb M,\]
such that 
\begin{enumerate}
\item $\bottomb M \cap \topb M=\emptyset$, $\leftb M \cap \rightb M=\emptyset$
\item The intersections $\bottomb M \cap \leftb M$ and $\topb M \cap \leftb M$ consist of $m$ circles, respectively.
\item The intersections $\bottomb M \cap \rightb M$ and $\topb M \cap \rightb M$ consist of $n$ circles, respectively.
\item The surfaces $\bottomb M$ and $\topb M$ are decorated surface of type $-t$ and $-s$, respectively.
\item The surfaces $\leftb M$ and $\rightb M$ are multi-pointed surface which are homeomorphic to $m$ cylinder over a circle and $n$ cylinder over a circle, respectively.
\item The base points of these four surfaces agree on their intersections.
\end{enumerate}
A ribbon graph $\Omega$ in $M$ meets $\partial M$ transversely along the distinguished arcs in $\bottomb M \cup \topb M \subset \partial M$ which are bases of certain bands of $\Omega$.
Such a manifold $M$ together with a $v$-colored ribbon graph $\Omega$ is said to be \textit{decorated} if the surfaces $\bottomb M$, $\topb M$, $\leftb M$, and $\rightb M$ are \textit{parametrized}.
This means that there are $d$-homeomorphism (for bottom and top) and base point preserving homeomorphisms (for left and right) 
\[\phi_B:\Sigma_{t} \to -\bottomb M,\]
\[\phi_T: \Sigma_{s}^{-} \to \topb M,\]
\[\phi_L: C(m) \to \leftb M, \]
\[\phi_R: C(n) \to \rightb M. \]
We call $\phi=(\phi_B,\phi_T,\phi_L,\phi_R)$ a \textit{parametrization} of $\partial M$ (or  $M$).

A $d$-homeomorphism of decorated 3-manifolds is a homeomorphism of the underlying 3-manifold preserving all additional structures in question.
In the sequel, we often call $d$-homeomorphism simply homeomorphism when the domain and the range are decorated cobordisms with corners.

We say that such pairs $(M, \phi)$ and $(M', \phi')$ are equivalent if there exist a ($d$-)homeomorphism $f$ from $M$ to $M'$ such that it commutes with parametrizations: $f\circ\phi_*= \phi'_*$ for $*=B$, $T$, $L$, $R$.
This is clearly an equivalence relation.

\begin{Definition}[2-morphisms of $\Co$]
Let $\Sigma_{t}$ and $\Sigma_{s}$ be 1-morphisms from $\nstand{m}$ to $\nstand{m}$.
A \textit{2-morphism} from $\Sigma_{t}$ to $\Sigma_{s}$ is the class $[(M,\phi)]$ of a pair of a decorated cobordism with corners from $\Sigma_t$ to $\Sigma_s$ and its parametrization $\phi$.
For each 1-morphism $X$, we add the formal identity symbol $\id_{X}$.
If one of the 1-morphisms is a formal identity 1-morphism $\id_n$, then there is no 2-morphism unless both are formal identity 1-morphisms and for this case there is only one formal identity 2-morphism $\id_{\id_n}$.
\end{Definition}

Let $(M,\phi)$ be a representative of the 2-morphisms $[(M,\phi)]$ from $\Sigma_t$ to $\Sigma_s$.
We define the \textit{standard boundary} $\Sigma(\phi)$ for the parametrization $\phi$ to be the surface obtained from $\Sigma_t$, $\Sigma_s^-$, $C(m)$, and $C(n)$ by identifying the boundaries via homeomorphisms of boundaries 
\[g_{BL}:=\phi_L^{-1}\circ \phi_B |_{\leftb \Sigma_t},\]
\[g_{BR}:=\phi_R^{-1}\circ \phi_B |_{\rightb \Sigma_t},\]
\[g_{TL}:=\phi_L^{-1}\circ \phi_T |_{\leftb \Sigma^{-}_s},\]
\[g_{TR}:=\phi_R^{-1}\circ \phi_T |_{\rightb \Sigma^{-}_s}.\]
Hence 
\[\Sigma(\phi)=(\Sigma_t \sqcup \Sigma^{-}_s) \cup_{\mbox{glue}}( C(m) \sqcup C(n)),\] 
where ``glue'' means the identification of boundaries by the homeomorphisms $g_{BL}$, $g_{BR}$, $g_{TL}$, $g_{TR}$.
Then the parametrization $\phi$ of $M$ induces the homeomorphism, also denoted by $\phi$, from $\Sigma(\phi)$ to $\partial M$.

In addition to decorated 3-manifolds with specific parametrizations, we will add formal identities in the set of 2-morphism.
The details are explained below when we deal with compositions.

We now introduce the notion of \textit{isotopy} in decorated cobordisms with corners.
Recall that if $\Sigma$ is a parametrized $d$-surface, then the cylinder $\Sigma \times [0, 1]$ has a natural structure of a decorated cobordism.
\begin{Definition}
Let $\phi$ and $\phi'$ be two parametrizations of $M$. 
We say that $\phi=(\phi_B, \phi_T,\phi_L,\phi_R)$ and $\phi'=(\phi'_B, \phi'_T,\phi'_L,\phi'_R)$ are \textit{isotopic} if the following conditions are satisfied.
Let $S_*$ be a standard boundary corresponding to each $*=B$, $T$, $L$, $R$.
\begin{enumerate}
\item $\phi_*$ is equal to $\phi'_*$ on the boundary circles $\partial S_*$ for each $*$.
\item  
There is a homeomorphism $F_*:S_* \times [0, 1]\to \partial_* M \times [0,1]$  for each $*$ satisfying the following conditions.
\begin{enumerate}
\item  $F_*(x,0)=\phi_*(x) \times \{0\}$ and $F_*(x, 1)=\phi'_*(x)\times \{1\}$.
\item Its restriction on $\partial S_* \times [0,1]$ agrees with $\phi_* \times \id_{[0,1]}$
\end{enumerate}
\end{enumerate}

\end{Definition}

If two parametrizations $\phi$ and $\phi'$ are isotopic, then we have $\Sigma(\phi)=\Sigma(\phi')$ since the gluing maps are the same by the condition (1).
The condition (b) guarantees that we can combine four homeomorphisms $F_*$ with $*=B$, $T$, $L$, $R$ into a homeomorphism 
\[F:\Sigma(\phi) \times [0,1] \to \partial M \times [0,1]\]
such that $F(x, 0)=\phi(x) \times \{0\}$ and $F(x,1)=\phi'(x) \times \{1\}$.

\begin{lemma}\label{lem:isotopy equivalence}
	Let $\phi$ and $\phi'$ be  two parametrizations on a decorated cobordism $M$. 
	Assume that $ \phi $ and $ \phi' $ are isotopic. 
	Then $(M, \phi)$ is equivalent to $(M, \phi')$.
\end{lemma}

\begin{proof}
		Let us just write $\Sigma$ for $\Sigma(\phi)=\Sigma(\phi')$.
	Let $F:\Sigma\times [0,1] \to \partial M \times [0, 1]$ be a $d$-homeomorphism that gives an isotopy between $\phi$ and $\phi'$ so that we have $F(x, 0)=\phi(x)\times \{0\}$ and $F(x, 1)=\phi'(x)\times \{1\}$.
	Consider a collar neighborhood $U=\partial M \times [-1,0]$ of $\partial M$ in $M$.
	Also consider the space $M \cup_{\partial M \times \{0\}} (\partial M \times [0,1])$ obtained by attaching $\partial M \times [0,1]$ to $M$ along $\partial M\times \{0\}=\partial M$.
	Let $f:M \to M \cup_{\partial M \times \{0\}} (\partial M \times [0,1])$ be a map that is identity outside $U$ and sends each point $x \times t \in U$ with $x\in \partial M$ and $t\in[-1, 0]$ to the point
	\[ x \times (2t+1)\in U \cup_{\partial M \times \{0\}} (\partial M \times [0,1]) \subset  M \cup_{\partial M \times \{0\}} (\partial M \times [0,1]).\]
	See Figure \ref{fig:callar f}.
	It is easy to see that the map $f$ is a $d$-homeomorphism.
	\begin{figure}[h]
\center
\includegraphics[width=4.2in]{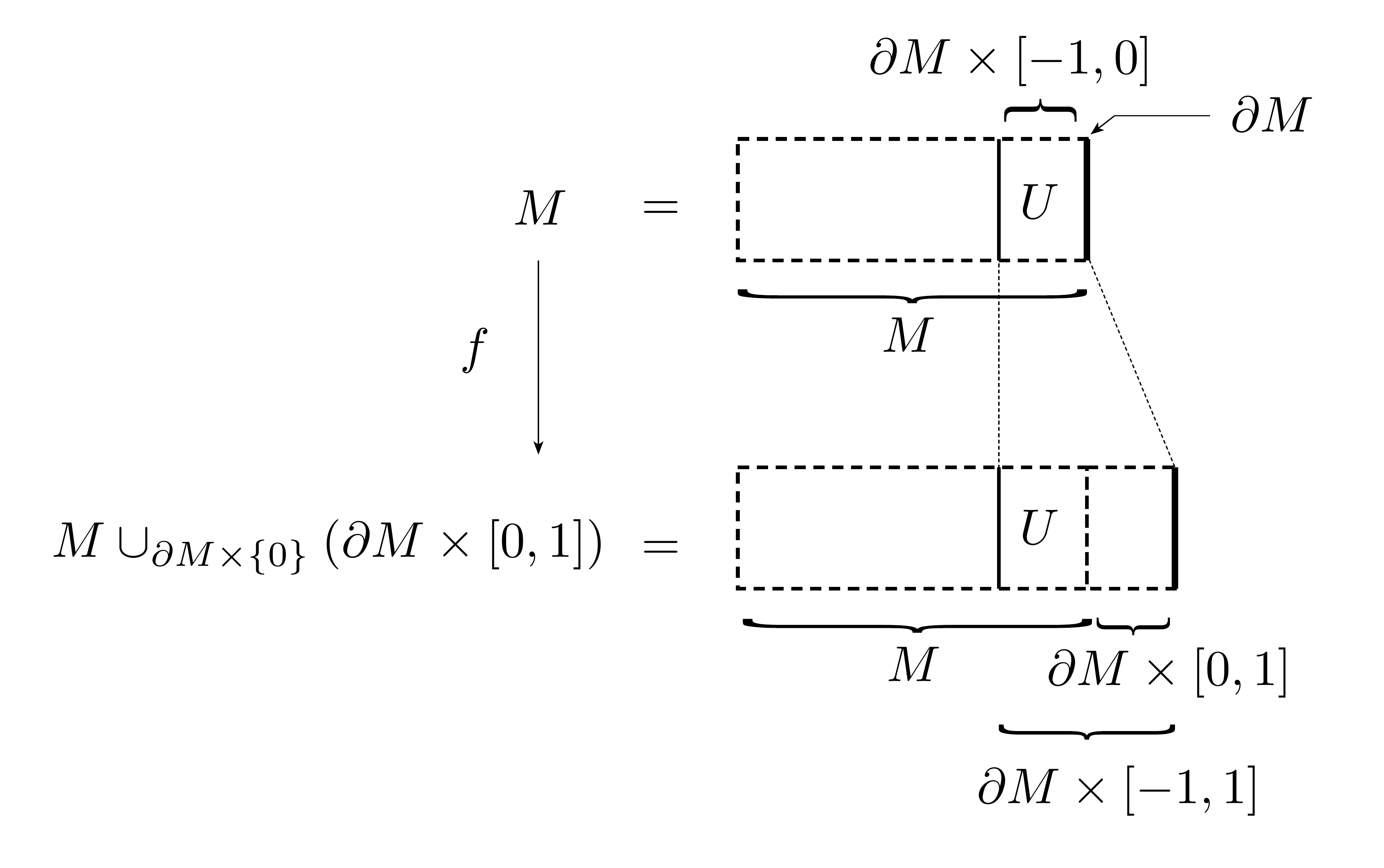}
\caption{The $d$-homeomorphism $f$}
\label{fig:callar f}
\end{figure}
	 
	Now the proof of the lemma is summarized into the following commutative diagram. 
	\begin{center}
	\begin{tikzpicture}

	\matrix (m) [matrix of nodes, column sep=3em, row sep=1em]
	{  & $M$ & $M \cup_{\partial M \times \{0\}} (\partial M \times [0,1])$ &\\
		$\Sigma$ &  &  &  $M\cup_{\phi^{-1}} (\Sigma \times [0,1])$ \\
		& $M$ & $M \cup_{\partial M \times \{0\}} (\partial M \times [0,1])$ & \\};
	
	\path[->, font=\scriptsize]
	(m-2-1) edge node[above] {$\phi$} (m-1-2);
	
	\path[->, font=\scriptsize]
	(m-1-2) edge node[above] {$f$} node[below]{$\sim$} (m-1-3);
	
	\path[->, font=\scriptsize]
	(m-1-3.east) edge node[auto] {$\id_M\cup (\phi^{-1} \times \id_{[0,1]})$}
	node[below right, sloped] {$\sim$} (m-2-4.north west);
	
	\path[->, font=\scriptsize]
	(m-3-3) edge node[above]{$f^{-1}$} node[auto] {$\sim$} (m-3-2);
	
	\path[->, font=\scriptsize]
	(m-2-4.south west) edge node[auto] {$\id_M\cup F$} 
	node[above left, sloped] {$\sim$} (m-3-3.east);
	
	\path[->, font=\scriptsize]
	(m-2-1) edge node[below] {$\phi'$} (m-3-2);
	
	\path[->, font=\scriptsize, dashed]
	(m-1-2) edge node[above] {} (m-3-2);
	
	\end{tikzpicture}
	\end{center}
	The collar homeomorphism  gives the homeomorphisms at the top and the bottom of the diagram above.
	The  space $ M \cup (\partial M_B \times [0,1])$ is further homeomorphic to $M\cup_{\phi^{-1}} (\Sigma \times [0,1])$, where we identify $\partial M$ with $\Sigma\times \{0\}$ by the inverse of the parametrization $\phi^{-1}$.
	The homeomorphism is given by the identity on $M$ and $[0, 1]$, and $\phi^{-1}$ from $\partial M$ to $\Sigma$.
	The next homeomorphism from $M\cup_{\phi^{-1}} (\Sigma \times [0,1])$ to $M \cup (\partial M \times [0,1])$ is given by the identity on $M$ and $F$ on the rest.
	The map $\id\cup F$ is compatible with the unions:
	every element $x\in \partial M$ is identified with $(\phi^{-1}(x), 0) \in\Sigma \times \{0\}$.
	This is, in turn, mapped to $F(\phi^{-1}(x), 0)=(\phi\circ\phi^{-1}(x), 0)=(x, 0) \in \partial M \times \{0\}$ and this is identified with $x=\id(x)\in \partial M$.
	Thus the map $\id \cup F$ is well-defined.
	Composing these homeomorphisms, we obtain a homeomorphism from $M$ to $M$ (the dashed arrow in the diagram).
	Now we show that this homeomorphism commutes with parametrizations.
	For each element $x\in \Sigma$, we have the following commutative diagram and it shows that the homeomorphism commutes with parametrizations.
	\begin{tikzpicture}

	\matrix (m) [matrix of nodes, column sep=3em, row sep=3em]
	{  & $\phi(x)\in M$ & $\phi(x) \times \{1\} \in M \cup (\partial M \times [0,1])$ \\
		$x \in \Sigma$ &  &    $x\times \{1\} \in M\cup_{\phi^{-1}} (\Sigma \times [0,1])$ \\
		& $\phi'(x)\in M$ & $F(x,1)=\phi'(x)\times \{1\} \in M \cup (\partial M \times [0,1])$  \\};
	
	\path[|->, font=\scriptsize]
	(m-2-1) edge node[above] {$\phi$}  (m-1-2);
	
	\path[|->, font=\scriptsize]
	(m-1-2) edge  (m-1-3);
	
	\path[|->, font=\scriptsize]
	(m-1-3.south) edge node[auto] {$\id\cup \phi^{-1} \cup \id$}
	(m-2-3);
	
	\path[|->, font=\scriptsize]
	(m-3-3) edge  (m-3-2);
	
	\path[|->, font=\scriptsize]
	(m-2-3) edge node[auto] {$\id\cup F$} 
	(m-3-3.north);
	
	\path[|->, font=\scriptsize]
	(m-2-1) edge node[below] {$\phi'$} (m-3-2);
	
	\end{tikzpicture}

\end{proof}
%
%
%

\subsection{Proving that $\Co$ is a 2-category}
Now that we defined the data $\Co$, in this section we will show the following proposition:
\begin{prop}
The data $\Co$ is a 2-category.
\end{prop}
Our convention about 2-categories is summarized in Section \ref{sec:appendix:bicategory}.
To claim that $\Co$ is a 2-category, we need to define several composition rules for 1-morphisms and 2-morphisms among other things.

Let $\Sigma_t :\nstand{l} \to \nstand{m}$ and $\Sigma_s :\nstand{m} \to \nstand{n}$ be two 1-morphism so that the target object of $\Sigma_t$ is the source object of $\Sigma_s$ (including the cases when $l=0$ or $n=0$). 
We define the composition of $\Sigma_t$ and $\Sigma_s$ to be $\Sigma_{t\circ s}$, where $t\circ s$ is the composition of types defined in (\ref{equ:composition of types}).
This composition is associative since the composition of types is associative. 
For each object $\nstand{n}$ with an integer $n$, we let the formal symbol $\id_n$ act as an identity.

Remark: note that the composition of 1-morphism is not a topological gluing of surfaces along boundaries.
As 1-morphism surfaces, we fixed the standard surfaces and we need that the composite of 1-morphisms is also a standard surface.
Also note that on the 1-morphism level, the boundary circles are not parametrized and hence there is no canonical homeomorphism of boundary circles.


\subsubsection{Vertical Gluing}
We define the vertical composition.
Let $[(M_1, \phi_1)]:\Sigma_{t_1}\Rightarrow \Sigma_{t_2}: \nstand{m} \to \nstand{n}$ and $[(M_2, \phi_2)]:\Sigma_{t_2}\Rightarrow \Sigma_{t_3}: \nstand{m} \to \nstand{n}$ be 2-morphisms of $\Co$ so that the target 1-morphism of $[M_1]$ is equal to the source 1-morphism of $[M_2]$.
We define the vertical composite $[M_1]\cdot [M_2]$ of $[M_1]$ and $[M_2]$.
The vertical composite will be a 2-morphism  $[M_1]\cdot [M_2]: \Sigma_{t_1}\Rightarrow \Sigma_{t_3}: \nstand{l} \to \nstand{n}$.
Let us fix representative $(M_1, \phi_1)$ and $(M_2, \phi_2)$ of these 2-morphisms.
We first glue $M_1$ and $M_2$ along the top boundary $\phi_1(\Sigma_{t_2}^-)$ of $M_1$ and the bottom boundary $\phi_2(\Sigma_{t_2})$ via the homeomorphism obtained from the composition of the following homeomorphisms
\[\partial M_1 \supset \phi_1(\Sigma_{t_2}^-) \xrightarrow{\phi_1^{-1}} \Sigma_{t_2}^- \xrightarrow{(\mir)^{-1}} \Sigma_{t_2} \xrightarrow{\phi_2} \phi_2(\Sigma_{t_2}) \subset \partial M_2.\]
Denote the resulting manifold by $M_1 \cdot M_2$.
Now we need to construct a parametrization from $\Sigma_{t_1} \sqcup \Sigma_{t_3}^{-} \sqcup C(m) \sqcup C(n)$ to the boundary of $M_1\cdot M_2$.
There is a natural parametrization, which we denote by $\phi_1\cdot_{\text{v}} \phi_2$, obtained as follows.
The map $\phi_1\cdot_{\text{v}} \phi_2$ restricts to $\phi_1$ and $\phi_2$ on $\Sigma_{t_1}$ and $\Sigma_{t_3}^-$.
This means that $(\phi_1 \cdotv \phi_2)_B=(\phi_1)_B$ and $(\phi_1 \cdotv \phi_2)_T=(\phi_2)_T$.
Next we define $(\phi_1 \cdotv \phi_2)_L$ as follows.
Let $C_1(m)$ and $C_2(m)$ be copies of $C(m)$ with $(\phi_i)_L: C_i(m) \to \rightb M_i$ for $i=1,2$.
We identify the top boundary of the cylinder $C_1(m)$ with the bottom boundary of $C_2(m)$ via the homeomorphism
\begin{equation}\label{equ:zeta vertical gluing map}
\zeta:=(\phi_2)^{-1}_L (\phi_2)_B (\mir)^{-1} (\phi_1)_T^{-1} (\phi_1)_L |_{\topb C_1(m)}.
\end{equation}
The following diagram summarizes the definition of $\zeta$.

\begin{center}
	\begin{tikzpicture}
\matrix (m) [matrix of nodes, column sep=3em, row sep=1em]
{  
$C_1(m)$ &  & & $C_2(m)$ \\
 &  $\Sigma_{t_2}^{-}$ &  $\Sigma_{t_2}$   \\
 $M_1$ &  & & $M_2$  \\
};

\path[->, font=\scriptsize]
(m-1-1) edge  node[auto] {$\zeta$} (m-1-4);

\path[->, font=\scriptsize]
(m-1-1) edge node[auto] {$(\phi_1)_L$}
 (m-3-1);

\path[->, font=\scriptsize]
(m-2-2) edge node[auto] {$(\phi_1)_T$}
 (m-3-1);

\path[->, font=\scriptsize]
(m-1-4) edge node[auto] {$(\phi_2)_L$}
 (m-3-4);

\path[->, font=\scriptsize]
(m-2-3) edge node[auto] {$(\phi_2)_B$}
 (m-3-4);

\path[->, font=\scriptsize]
(m-2-2) edge node[auto] {$(\mir)^{-1}$} (m-2-3);

\end{tikzpicture}
\end{center}
Then there is a natural homeomorphism extending $(\phi_1)_L$ and $(\phi_2)_L$ from $C_1(m) \cup_{\zeta} C_2(m)$ to $\rightb (M_1\cdot M_2)$.
We denote this map by $\phi_1 \cup_{\zeta} \phi_2$.
A problem is that $C_1(m) \cup_{\zeta} C_2(m)$ is not a standard surface.
However, this can be easily remedied thanks to the cylindrical structure of $C(m)=\nstand{m}\times [0,1]$.
Let us define the stretching map $s$ from $C(m)$ to $C_1(m) \cup_{\zeta} C_2(m)$ by sending $(x, t) \in C(m)$ to $(x, 2t) \in C_1(m)$ if $t \leq 1/2$ and to $(\zeta(x), 2t)$ if $t >1/2$.
We define the parametrization $(\phi_1 \cdotv \phi_2)_L: C(m) \to \rightb(M_1\cdot M_2)$ to be the composite $(\phi_1 \cup_{\zeta} \phi_2) \circ s$.
Similarly we define the right parametrization.


Next we need to show that a different choice of representative gives rise to an equivalent parametrized manifold.
Let $(N_1, \psi_1)$ and $(N_2, \psi_2)$ be another choice of representatives for $[(M_1, \phi_1)]$ and $[(M_2, \phi_2)]$, respectively.
By the definition of the equivalence, we have homeomorphisms $\alpha:N_1 \to M_1$ and $\beta:N_2\to M_2$ such that the parametrizations commute: $\phi_1= \alpha \circ \psi_1$ and $\phi_2= \beta \circ \psi_2$.
These homeomorphisms induce a homeomorphism $\alpha \cup \beta: N_1\cdot N_2 \to M_1\cdot M_2$ such that $(\alpha \cup \beta)|_{N_1}=\alpha$ and $(\alpha \cup \beta)|_{N_2}=\beta$.
This is well-defined since on the glued components, we have the following commutative diagram.
\[
\begin{CD}
\topb N_1 @> \psi_2\circ (\mir)^{-1}\circ \psi_1^{-1} >> \bottomb N_2\\
@VV \alpha V @VV \beta V\\
\topb M_1 @> \phi_2 \circ (\mir)^{-1} \circ \phi_1^{-1} >> \bottomb M_2\\
\end{CD}
\]
Then we claim that the homeomorphism $\alpha \cup \beta: N_1 \cdot N_2 \to M_1\cdot M_2$ commutes with parametrizations: $\alpha \circ(\psi_1\cdotv \psi_2)=\phi_1\cdotv \phi_2$.
On the bottom and top boundaries, this is clear.
Let us check this equation on the left boundary.
Recall that the left parametrization is defined to be $(\phi_1 \cdotv \phi_2)|_{C(m)}=(\phi_1 \cup_{\zeta} \phi_2)\circ s$, where $\zeta$ is the gluing map of two copies of cylinders $C_1(m)$ and $C_2(m)$ defined in (\ref{equ:zeta vertical gluing map}) and $s$ is the stretching map.
For the second pair $(N_1, \psi_1)$ and $(N_2, \psi_2)$, we also have $(\psi_1 \cdotv \psi_2)|_{C(m)}=(\psi_1 \cup_{\eta} \psi_2)\circ s$, where $\eta$ is the gluing map of cylinders defined by
\[\eta:=(\psi_2)^{-1}_L (\psi_2)_B (\mir)^{-1} (\psi_1)_T^{-1} (\psi_1)_L |_{\topb C_1(m)}.\]
Since we have the following commutative diagram, we have in fact $\zeta=\eta$.
The commutativity of the left rectangle in the diagram is the definition of $\eta$ and the commutativity of the right rectangle follows since $\alpha$ and $\beta$ commute with parametrizations.
Note that the big rectangle is the definition of $\zeta$ since $\alpha \circ (\psi_1)_L=(\phi_2)_L$ and $\beta \circ (\psi_2)_L=(\phi_2)_L$.
\begin{center}
	\begin{tikzpicture}
\matrix (m) [matrix of nodes, column sep=7em, row sep=1em]
{  
$\bottomb C_2(m)$ & $\leftb N_2$ & $\leftb M_2$ \\
$\topb C_1(m)$ & $\leftb N_1$ & $\leftb M_1$ \\
 };

\path[->, font=\scriptsize]
(m-2-1) edge node[auto] {$\eta$}
 (m-1-1);

\path[->, font=\scriptsize]
(m-2-1) edge node[below] {$(\psi_1)_L$}
 (m-2-2);

\path[->, font=\scriptsize]
(m-1-1) edge node[auto] {$(\psi_2)_L$}
 (m-1-2);

\path[->, font=\scriptsize]
(m-1-2) edge node[auto] {$\beta$}
 (m-1-3);

\path[->, font=\scriptsize]
(m-2-2) edge node[below] {$\alpha$}
 (m-2-3);
 
 \path[->, font=\scriptsize]
(m-2-3) edge node[auto] {$\phi_2 (\mir)^{-1} \phi_1^{-1}$}
 (m-1-3);
 
 \path[->, font=\scriptsize]
(m-2-2) edge node[auto] {$\psi_2 (\mir)^{-1} \psi_1^{-1}$}
 (m-1-2);
 
\end{tikzpicture}
\end{center}
The commutativity of this diagram also shows that we have 
\[(\alpha \cup \beta) \circ (\psi_1 \cup_{\eta} \psi_2)\circ s=(\phi_1 \cup_{\zeta} \phi_2)\circ s\]
and hence we have 
\[(\alpha \cup \beta) \circ (\psi_1 \cdotv \psi_2)|_{C(m)}=(\phi_1 \cdotv \phi_2)|_{C(m)}\]
Similarly for the right boundaries.
Thus the definition of the vertical composite 
\[ [(M_1, \phi_1)]\cdot [(M_2, \phi_2)]:=[(M_1\cdot M_2, \phi_1 \cdotv \phi_2)] \]
 is independent of the choice of representatives.


\subsubsection{Associativity for the vertical composition}
Let $[(M_1, \phi_1)]:\Sigma_{t_1}\Rightarrow \Sigma_{t_2}:\nstand{m}\to \nstand{n}$, $[(M_2, \phi_2)]:\Sigma_{t_2}\Rightarrow \Sigma_{t_3}:\nstand{m}\to \nstand{n}$, and $[(M_3, \phi_3)]:\Sigma_{t_3}\Rightarrow \Sigma_{t_4}:\nstand{m}\to \nstand{n}$ be 2-morphisms.
The pairs $[M_1]$ and $[M_2]$, $[M_2]$ and $[M_3]$ are vertically composable.
We show that the associativity holds: We show that 
\[\Bigl((([M_1]\cdot  [M_2])\cdot  [M_3]), \quad (\phi_1 \cdotv \phi_2)\cdotv \phi_3 \Bigr)=
\Bigl(([M_1]\cdot  ([M_2]\cdot  [M_3])), \quad \phi_1 \cdotv (\phi_2 \cdotv \phi_3) \Bigl) \]
The both sides equal to $M_1 \cdot M_2 \cdot M_3$ as a manifold.
Thus it suffices to show that the parametrization $(\phi_1 \cdotv \phi_2)\cdotv \phi_3$ is isotopic to the parametrization $\phi_1 \cdotv (\phi_2 \cdotv \phi_3)$ by Lemma \ref{lem:isotopy equivalence}.
Checking this on the top and bottom boundaries are again trivial.
Let us check the isotopy on the left boundary.
From the definition of vertical composition, we need to consider three copies of the cylinder $C(m)$.
Name them $C_i(m)$ for $i=1,2,3$ for each $M_i$.
Let $\zeta_1$ be the gluing map for the cylinders $C_1(m)$ and $C_2(m)$ and let $\zeta_2$ be the gluing map for the cylinders $C_2(m)$ and $C_3(m)$ as in (\ref{equ:zeta vertical gluing map}).
(Technically, $\zeta_i$ depends an order of gluing of three cylinders but it is straightforward to see that $\zeta_i$ is defined independently of the order.)
Now by definition $(\phi_1 \cdotv \phi_2)\cdotv \phi_3$ on the left cylinder $C(m)$ is equal to $[ (\phi_1 \cup_{\zeta_1} \phi_2)\circ s_1 \cup_{\zeta_2} \phi_3]\circ s_2 $, where $s_i$ is the stretching map corresponding to $\zeta_i$.
This is equal to $[\phi_1 \cup_{\zeta_1} \phi_2 \cup_{\zeta_2} \phi_3 ] \circ (s_1 \cup \id )\circ s_2$.
Now we see that the map $(s_1 \cup \id )\circ s_2$ is isotopic to the map $(\id \cup s_2)\circ s_1$.
(This is like the proof of associativity of fundamental groups.)
Hence we see that $(\phi_1 \cdotv \phi_2)\cdotv \phi_3$ is isotopic to $\phi_1 \cdotv (\phi_2 \cdotv \phi_3)$ on the left boundary.
Similarly for the right boundary.
Therefore the associativity follows.



\subsubsection{Units for vertical  composition}\label{subsubsec:unit for the vertical composition of Co}
From the above arguments, for each pair of objects $\nstand{m}$, $\nstand{n}$ of $\Co$, we have the semigroupoid (category without identity) $\Co(\nstand{m}, \nstand{n})$ whose objects are 1-morphisms from $\nstand{n}$ to $\nstand{m}$ of $\Co$ and whose morphisms are 2-morphisms between such 1-morphisms in $\Co$.
To make the semigroupoid $\Co(\nstand{m}, \nstand{n})$ a category, we need to specify the identity morphism for each object of $\Co(\nstand{m}, \nstand{n})$.
For each object $X$ of $\Co(\nstand{m}, \nstand{n})$ (a standard surface), we just use a formal unit $\id_{X}$, rather than construct the identity cobordism.
Thus the formal unit $\id_{X}$ should act as the identity morphism in the category $\Co(\nstand{m}, \nstand{n})$.

Remark: one reason to use formal units is that otherwise we need to construct a concrete cobordism with a parametrization.
The obvious candidate for the identity cobordism is the cylinder over the standard surface $\Sigma_{t}\times [0,1]$, where $t$ is the type with $L(t)=m$ and $R(t)=n$.
The bottom boundary $\Sigma_t\times \{0\}$ can be identified with the standard surface $\Sigma_t$ and the identity map can be used as a parametrization.
The top boundary $\Sigma_t \times \{1\}$ also can be identified with $\Sigma_t$ and $\mir:\Sigma_t^- \to \Sigma_t$ can be used as a parametrization.
The problem is to define parametrizations for the left and the right boundaries.
We need to construct parametrization homeomorphism from $C(m)$ and $C(n)$ to the side boundaries of $\Sigma_t \times [0,1]$, which are cylinders over the boundary circles of $\Sigma_t$.
However there is no canonical homeomorphism at hand.
Since it does not seem that this gives more insights in our theory we just avoid the burden by introducing the formal units.
On the other hand, there is no obstruction in the case when $m=n=0$ since there is no side boundaries.
Using the cylindrical neighborhood, we can prove the cylinder over a standard surface is in fact the identity.

\subsubsection{Horizontal composition}

Next, we define the horizontal composition of 2-morphisms.
Let $X=\nstand{l}$, $Y=\nstand{m}$ and $Z=\nstand{n}$ be objects of $\Co$.
Let $\Sigma_{t_i}: X \to Y$ and $\Sigma_{s_i}: Y \to Z$ be 1-morphisms for $i=1,2$.
Let $[M]: \Sigma_{t_1} \Rightarrow \Sigma_{t_2}: X \to Y$ and $[M']:\Sigma_{s_1} \Rightarrow \Sigma_{s_2}: Y \to Z$ be 2-morphisms.
We define the horizontal composite $[M]\circ [M']$ of $[M]$ and $[M']$ as follows.
The composite $[M]\circ [M']$ will be a 2-morphism from $\Sigma_{t_1\circ s_1}: X\to Z$ to $\Sigma_{t_2 \circ s_2}: X \to Z$.
Pick representatives $(M,\phi)$ and $(M', \phi') $ for $[M]$ and $[M']$, respectively.
Here $\phi: \Sigma(\phi) \to \partial M$ and $\phi': \Sigma(\phi') \to \partial M'$ are parametrizations of boundaries of decorated cobordisms $M$ and $M'$, respectively.
We glue $M$ and $M'$ by identifying $\rightb M$ and $\leftb M'$ via the homeomorphism $\phi'_L \circ \refl \circ \phi_R^{-1}: \rightb M \to \leftb M'$.
Since this homeomorphisms is the composite of three orientation reversing maps, this map is orientation reversing.
In the sequel, we omit writing the map $\refl$ to simplify expressions.
Denote the resulting manifold by
\[ M\circ M'=M\cup_{\phi'_L \circ \phi_R^{-1}}M'.\]

The next task it to construct a parametrization of $M\circ M'$ from $C_{l}\sqcup C_{n} \sqcup\Sigma_{t_1\circ s_1} \sqcup\Sigma_{t_2\circ s_2}^-$ and then the equivalence class of this pair will be the horizontal composite.
On the left and right boundaries, the parametrizations are just $\phi$ and $\phi'$, respectively.
We now define a parametrization homeomorphism from the standard surface $\Sigma_{t_1\circ t_2}$ to the bottom boundary of $M\circ M'$.
This is not a straightforward task because the topological gluing of standard surfaces are not a standard surface.
First let us write 
\[ g:=\phi'^{-1}_{B}  \circ \phi'_L \circ \phi_{R}^{-1}\circ \phi_B: \rightb \Sigma_{t_1} \to \leftb \Sigma_{s_1}.\]
Gluing via this homeomorphism we obtain the surface $\Sigma_{t_1}\cup_g \Sigma_{s_1}$.
Define $\Phi=\Phi(\phi_B,\phi'_B): \Sigma_{t_1} \cup_g \Sigma_{s_1} \to \bottomb (M_1 \circ M_2)$ by
\[\Phi(x)=\Phi(\phi_B,\phi'_B)(x)=
\begin{cases}
\phi_B(x) & \mbox{ if } x \in \Sigma_{t_1} \\
\phi'_B(x) & \mbox{ if } x\in \Sigma_{s_1}
\end{cases}
\]
This is well-defined since $\bottomb (M_1 \circ M_2)=\bottomb M_1 \cup_{\phi'_L \circ \phi_R^{-1}} \bottomb M'$.
Next, because the surface $\Sigma_{t_1}\cup_g \Sigma_{s_1}$ is not a standard surface, we define a homeomorphism $\Sigma_{t_1 \circ s_1} \to \Sigma_{t_1}\cup_g \Sigma_{s_1}$.
This homeomorphism will depend on several choices. 
However, two different choices give homeomorphisms that differ only by an isotopy.
The standard surfaces $\Sigma_{t_1}$ and $\Sigma_{s_1}$ are by definition sitting in $\R^3$.
There is a translation  map $\tau$ of $\R^3$ that maps $\rightb \Sigma_{t_1}$ to  $\leftb \Sigma_{s_1}$.
Now both $\Sigma_{t_1\circ s_1}$ and $\tau(\Sigma_{t_1})\cup \Sigma_{s_1}$ are in $\R^3$ and they are homeomorphic.
We fix a homeomorphism $h(t_1, s_1):\Sigma_{t_1\circ s_1 }\to \tau(\Sigma_{t_1})\cup \Sigma_{s_1} $ as follows.
Recall that the standard surface $\Sigma_t$ is obtained from the boundary of the standard handlebody $U_t$ in $\R^3$ with several disks removed.
There is an ambient isotopy $F: \R^3 \times [0,1] \to \R^3$ of handlebodies $U_{t_1 \circ s_1}$ and $\tau(U_{t_1})\cup U_{s_1}$ which maps the ribbon graph in one to the other and when $F$ is restricted on $\leftb \Sigma_{t_1 \circ s_1}\subset U_{t_1 \circ s_1}$ it is just a translation in the $x$-coordinate in $\R^3$
 and when $F$ is restricted to $\rightb \Sigma_{t_1 \circ s_1}\subset U_{t_1 \circ s_1}$, it is also just a translation but it might move different amount in $x$-direction.
Then we define $h(t_1, s_1)$ to be the restriction of $F$ to $\Sigma_{t_1 \circ s_1}$.
There is a canonical homeomorphism $f_{\tau}:\tau(\Sigma_{t_1}) \cup \Sigma_{s_1}  \to \Sigma_{t}\cup_{\tau} \Sigma_{s_1}  $.
Here the latter space is obtained by regarding $\tau$ as a homeomorphism from $\rightb \Sigma_{t_1}$ to $\leftb \Sigma_{s_1}$ and gluing $\Sigma_{t_1}$ and $\Sigma_{s_1}$ along $\tau$. 
Finally we need to choose a homeomorphism  $\Sigma_{t}\cup_{\tau} \Sigma_{s_1} \to \Sigma_{t}\cup_{g} \Sigma_{s_1} $, which seems to be the most arbitrary.
First we have the following result which follows from Lemma in Appendix III of \cite{Turaev10}.

\begin{lemma}\label{lem:Turaev Appendix III}
	Let $X=\rightb \Sigma_t$ and let $f:X \to X$ be a homeomorphism preserving the orientation and the base points $\{x_i\}$ in each component of $X$.
	Then there exists a homeomorphism $\Psi: \Sigma_t \to \Sigma_t$ satisfying the following.
	\begin{enumerate}
		\item The restriction $\Psi|_{X}=f$.
		\item The homeomorphism $\Psi$ is the identity on $(\Sigma_t \setminus \mathrm{int}(U))\cup (\{x_i\}\times [0,1])$, where $U=X \times [0,1]$ is a cylindrical collar neighborhood of $X=X\times \{0\}$ in $\Sigma_t$.
		\item The homeomorphism $\Psi$ carries $X \times \{t\} \subset U$ into $X \times \{t\}$ for all $t \in [0,1]$.
		
	\end{enumerate}   
	Any two such homeomorphisms $\Psi:\Sigma \to \Sigma$ are isotopic via an isotopy constant on $\partial \Sigma$.
\end{lemma}
Now from two homeomorphisms $\tau$ and $g$ from $\rightb \Sigma_{t_1}$ to $\leftb \Sigma_{s_1}$ we obtain the self homeomorphism $f=\tau^{-1}\circ g$ of $X=\rightb \Sigma_{t_1}$.
Lemma \ref{lem:Turaev Appendix III} yields a homeomorphism $\Psi(f):\Sigma_{t_1} \to \Sigma_{t_1}$ that extends $f=\tau^{-1}\circ g$ and any two such homeomorphisms are isotopic.
We fix one such $\Psi(f)$.
Thus $\Psi(f)$ induces an homeomorphism $\Psi(g, \tau):\Sigma_{t}\cup_{\tau} \Sigma_{s_1} \to \Sigma_{t}\cup_{g} \Sigma_{s_1} $ and different choice of $\Phi(f)$ induces an isotopic homeomorphism $\Psi(g, \tau)$.
Then we define the parametrization homeomorphism $\phi_B\circ_h \phi'_B$  of $M \circ M'$ to be
\begin{equation}\label{equ:horizontal parametrizations}
\phi_B\circ_h \phi'_B:=\Phi(\phi_B, \phi'_B)\circ \Psi(g, \tau) \circ f_{\tau} \circ h(t_1, s_1): \Sigma_{t_1 \circ s_1} \to \bottomb (M \circ M').
\end{equation}
Similarly we obtain the top parametrization $\phi_T\circ_h \phi'_T:\Sigma_{t_2\circ s_2}^{-} \to \topb (M\circ M')$.

Now we obtained the pair $(M\circ M',\phi\circ_h\phi')$.
After choosing the representatives $M$ and $M'$, there are several choices we made to define the parametrization.
Namely, $h(t_1, s_1)$ and $\Psi(g, \tau)$ are defined up to isotopy fixing boundaries.
But this ambiguity does not affect the equivalence class by virtue of Lemma \ref{lem:isotopy equivalence}.
The next lemma shows that a different choice of representatives of a 2-morphism gives an equivalent pair.
\begin{lemma}
	Suppose that $(M,\phi)$ is equivalent to $(N, \psi)$ and $(M', \phi')$ is equivalent to $(N',\psi')$ then $(M\circ M, \phi\circ_h \phi')$ is equivalent to $(N\circ N', \psi\circ_h \psi')$.
\end{lemma}
\begin{proof}

	By the definition of the equivalence, there are homeomorphisms $\alpha:N \to M$ and $\beta:N'\to M'$ such that $\phi= \alpha \circ \psi$ and $\phi'=\beta \circ \psi'$.
	They induce a homeomorphism $\alpha \cup \beta : N\circ N' \to M \circ M'$.
	This homeomorphism is well-defined since on the common boundary we have the following commutative diagram:

	\begin{center}

		\begin{tikzpicture}
		\matrix (m) [matrix of nodes, column sep=3em, row sep=1.5em]
		{$C(m)$ &  &  & $C(m)$ \\
			& $N$ & $N'$ &\\
			& $M$ & $M'$ &\\
			$C(m)$ &  &  & $C(m)$\\};
		
		\path[->, font=\scriptsize]
		(m-1-1) edge node[auto] {$=$} (m-1-4);
		
		\path[->, font=\scriptsize]
		(m-1-1) edge node[auto] {$=$} (m-4-1);
		
		\path[->, font=\scriptsize]
		(m-1-1) edge node[auto] {$\psi_R$} (m-2-2);
		
		\path[->, font=\scriptsize]
		(m-2-2) edge node[auto] {$\psi'_L \circ \psi^{-1}_R$} (m-2-3);
		
		\path[->, font=\scriptsize]
		(m-2-3) edge node[auto] {$\beta$} (m-3-3);
		
		\path[->, font=\scriptsize]
		(m-2-2) edge node[auto] {$\alpha$} (m-3-2);
		
		\path[->, font=\scriptsize]
		(m-3-2) edge node[auto] {$\phi'_L \circ \phi^{-1}_R$} (m-3-3);
		
		\path[->, font=\scriptsize]
		(m-1-4) edge node[auto] {$=$} (m-4-4);
		
		\path[->, font=\scriptsize]
		(m-4-1) edge node[auto] {$=$} (m-4-4);
		
		\path[->, font=\scriptsize]
		(m-4-1) edge node[auto] {$\phi_R$} (m-3-2);
		
		\path[->, font=\scriptsize]
		(m-4-4) edge node[auto] {$\phi_L'$} (m-3-3);
		
		\path[->, font=\scriptsize]
		(m-1-4) edge node[auto] {$\psi'_L$} (m-2-3);
		
		\end{tikzpicture}
	\end{center}

	We need to show that  
	\begin{equation}\label{equ:horizontal parametrization commutes}
	\phi \circ_{\text{h}} \phi'=(\alpha \cup \beta)\circ (\psi \circ_{\text{h}} \psi').
	\end{equation}
	The right and the left boundary parts just follow from the definition of $\alpha$ and $\beta$.. 
	Let us check this equality on the bottom boundary.
	By the remark before the lemma, we can choose the parametrizations as follows.
\[\phi_B\circ_h \phi'_B:=\Phi(\phi_B, \phi'_B)\circ \Psi(g, \tau) \circ f_{\tau} \circ h(t_1, s_1): \Sigma_{t_1 \circ s_1} \to \bottomb (M \circ M')\]
and
\[\psi_B\circ_h \psi'_B:=\Phi(\psi_B, \psi'_B)\circ \Psi(g', \tau) \circ f_{\tau} \circ h(t_1, s_1): \Sigma_{t_1 \circ s_1} \to \bottomb (N \circ N'),\]
	where
	\[ g:=\phi'^{-1}_{B}  \circ \phi'_L \circ \phi_{R}^{-1}\circ \phi_B: \rightb \Sigma_{t_1} \to \leftb \Sigma_{s_1}\] 
	and
	\[ g':=\psi'^{-1}_{B}  \circ \psi'_L \circ \psi_{R}^{-1}\circ \psi_B: \rightb \Sigma_{t_1} \to \leftb \Sigma_{s_1}.\]
Because we have homeomorphism commuting with parametrizations, we have the following commutative diagram and we have in fact $g=g'$

	\begin{center}

	\begin{tikzpicture}

\matrix (m) [matrix of nodes, column sep=3em, row sep=2em]
{  

 &  $M$ &    & $M'$ & \\
 $\Sigma_{t}$ &  & $C(m)$& & $ \Sigma_{s_1}$ \\
 &  $N$ & & $N'$ & \\};

\path[->, font=\scriptsize]
(m-2-1) edge  node[auto] {$\phi_B$} (m-1-2);

\path[->, font=\scriptsize]
(m-2-1) edge node[left]{$\psi_B$} (m-3-2);

\path[->, font=\scriptsize]
(m-2-3) edge node[auto] {$\phi_R$}
 (m-1-2);

\path[->, font=\scriptsize]
(m-2-3) edge node[below] {$\psi_R$} (m-3-2);

\path[->, font=\scriptsize]
(m-2-3) edge node[auto] {$\phi'_L$}
 (m-1-4);

\path[->, font=\scriptsize]
(m-2-3) edge node[auto] {$\psi'_L$}
 (m-3-4);

\path[->, font=\scriptsize]
(m-2-5) edge node[auto] {$\phi'_B$}
 (m-1-4);

\path[->, font=\scriptsize]
(m-2-5) edge node[auto] {$\psi'_B$}
 (m-3-4);

\path[->, font=\scriptsize]
(m-3-2) edge node[auto] {$\alpha$}
 (m-1-2);

\path[->, font=\scriptsize]
(m-3-4) edge node[auto] {$\beta$}
 (m-1-4);

\end{tikzpicture}
\end{center}

We also have the following commutative diagram and the equality (\ref{equ:horizontal parametrization commutes}) holds on the bottom boundary.
	
\begin{center}	
\begin{tikzpicture}

\matrix (m) [matrix of nodes, column sep=6em, row sep=2em]
{  &  & $\bottomb(M\circ M')$ \\
 $\Sigma_{t\circ t'}$ &  $\Sigma_{t} \cup_{g} \Sigma_{t'}$    \\
 &  & $\bottomb (N \circ N')$ \\};

\path[->, font=\scriptsize]
(m-2-1) edge node[auto] {$\Psi(g, \tau) f_{\tau}h(t_1, s_1)$}
 (m-2-2);
 
 \path[->, font=\scriptsize]
(m-2-2) edge node[auto] {$\Phi(\phi_B, \phi_{B'})$}
 (m-1-3);
 
 \path[->, font=\scriptsize]
(m-2-2) edge node[auto] {$\Phi(\psi_B, \psi_{B'})$}
 (m-3-3);
 
 \path[->, font=\scriptsize]
(m-1-3) edge node[auto] {$\alpha \cup \beta $}
 (m-3-3);

\end{tikzpicture}
\end{center}
Similarly for the top boundary.
Hence we have $\phi \circ_{\text{h}} \phi'=(\alpha \cup \beta )\circ (\psi \circ_{\text{h}} \psi')$ and conclude that 
	$(M\circ M, \phi\circ_h \phi')$ is equivalent to $(N\circ N', \psi\circ_h \psi')$.

%
%
%
%
%
%
%
%
%
%
%
%
	
	\end{proof}

%
%

\subsubsection{Associativity for horizontal composition}

Let $(M, \phi)$, $(M', \phi')$, and $(M'', \phi'')$ be representative of 2-morphisms of $\Co$ such that $(M, \phi)$ and  $(M', \phi')$, $(M', \phi')$ and $(M'', \phi'')$ are horizontally composable.
We show that horizontal composition of $Co$ is associative.
It suffices to show that the map $(\phi\circh \phi') \circh \phi''$ is isotopic to $\phi \circh (\phi' \circh \phi'')$.
We check this on the bottom part.

Recall the definition of horizontal composition of parametrizations from (\ref{equ:horizontal parametrizations}).
For the sake of simplicity, we use the letter $\tau$ for translations in $\R^3$ and we 
denote by $\bar h$ the composite of homeomorphism $f_{\tau}\circ h(t, s): \Sigma_{t\circ s} \to \Sigma_{t} \cup_{\tau} \Sigma_{s}$. 
Thus maps $\tau$ and $\bar h$ should be understood from the context.
Let $g_1$ and $g_2$ be the homeomorphism defined by the following commutative diagram.

	\begin{center}

	\begin{tikzpicture}

\matrix (m) [matrix of nodes, column sep=3em, row sep=1em]
{  
$\rightb \Sigma_{t}$ &  & $\leftb \Sigma_{t'}, \rightb \Sigma_{t'}$& & $ \leftb \Sigma_{t''}$ \\
 &  $C(m)$ &    & $C(n)$ & \\
 $\bottomb M$ &  & $\bottomb M'$& & $ \bottomb M''$ \\
};

\path[->, font=\scriptsize]
(m-1-1) edge  node[auto] {$g_1$} (m-1-3);

\path[->, font=\scriptsize]
(m-1-3) edge node[auto]{$g_2$} (m-1-5);

\path[->, font=\scriptsize]
(m-1-1) edge node[auto] {$\phi_B$}
 (m-3-1);

\path[->, font=\scriptsize]
(m-2-2) edge node[auto] {$\phi_R$}
 (m-3-1);

\path[->, font=\scriptsize]
(m-2-2) edge node[auto] {$\phi'_L$}
 (m-3-3);

\path[->, font=\scriptsize]
(m-2-4) edge node[auto] {$\phi'_R$}
 (m-3-3);

\path[->, font=\scriptsize]
(m-2-4) edge node[auto] {$\phi''_L$}
 (m-3-5);

\path[->, font=\scriptsize]
(m-1-5) edge node[auto] {$\phi''_B$}
 (m-3-5);

\path[->, font=\scriptsize]
(m-1-3) edge node[auto] {$\phi'_B$}
 (m-3-3);

\end{tikzpicture}
\end{center}
Thus $g_1$ and $g_2$ are gluing homeomorphism of standard surfaces induced by parametrizations.
To calculate $(\phi\circh \phi') \circh \phi''$ and $\phi \circh (\phi' \circh \phi'')$ , we also need  the following gluing homeomorphisms $g_3$ and $g_4$, respectively.
\begin{center}
	\begin{tikzpicture}
\matrix (m) [matrix of nodes, column sep=7em, row sep=1em]
{  
$\rightb \Sigma_{t}$ &  & $\leftb \Sigma_{t'\circ t''}$\\
 &  $C(m)$ &\\
 $\bottomb M$ &  & $\bottomb (M'\circ M'') $ \\
};

\path[->, font=\scriptsize]
(m-1-1) edge  node[auto] {$g_3$} (m-1-3);

\path[->, font=\scriptsize]
(m-1-1) edge node[auto] {$\phi_B$}
 (m-3-1);

\path[->, font=\scriptsize]
(m-2-2) edge node[auto] {$\phi_R$}
 (m-3-1);

\path[->, font=\scriptsize]
(m-2-2) edge node[auto] {$(\phi' \circh \phi'')_L=\phi'_L$}
 (m-3-3);

\path[->, font=\scriptsize]
(m-1-3) edge node[auto] {$(\phi' \circh \phi'')_B$}
 (m-3-3);

\end{tikzpicture}
\end{center}

\begin{center}

	\begin{tikzpicture}

\matrix (m) [matrix of nodes, column sep=7em, row sep=1em]
{  
$\rightb (\Sigma_{t\circ t'})$ &  & $\leftb \Sigma_{ t''}$\\
 &  $C(n)$ &     \\
 $\bottomb (M\circ M') $ &  & $\bottomb  M'' $ \\
};

\path[->, font=\scriptsize]
(m-1-1) edge  node[auto] {$g_4$} (m-1-3);

\path[->, font=\scriptsize]
(m-1-1) edge node[auto] {$(\phi\circh \phi')_B$}
 (m-3-1);

\path[->, font=\scriptsize]
(m-2-2) edge node[auto] {$(\phi\circh \phi')_R=\phi_R$}
 (m-3-1);

\path[->, font=\scriptsize]
(m-2-2) edge node[auto] {$ \phi''_L$}
 (m-3-3);

\path[->, font=\scriptsize]
(m-1-3) edge node[auto] {$ \phi''_B$}
 (m-3-3);

\end{tikzpicture}
\end{center}
Let $\Psi_i:=\Psi(g_i, \tau)$.
Note that since $(\phi' \circh \phi'')_B |_{\leftb (\Sigma_{t' \circ t''})}=\phi'_B \circ \Psi_2\circ \bar h |_{\leftb(\Sigma_{t'\circ t''})}$, we have $\Psi_2 \circ \bar h \circ g_3=g_1$.
Moreover $\Psi_2 \circ \bar h$ is identity on the boundary $\leftb \Sigma_{t' \circ t''}$, thus in fact we have $g_1=g_3$, and hence $\Psi_1=\Psi_3$.
Also since $(\phi \circh \phi')_B|_{\rightb (\Sigma_{t \circ t'})}=\phi'_B \circ \Psi_1\circ \bar h |_{\rightb(\Sigma_{t \circ t'})}$, we have $g_2\circ \Psi_1 \circ \bar h =g_4$.
Deciphering the definition of maps we have the following diagram, where the left path is a parametrization $\phi_B \circh (\phi'_B \circh \phi''_B)$ and the right path is the parametrization $(\phi_B \circh \phi'_B)\circh \phi''_B$.

	\begin{center}

	\begin{tikzpicture}

\matrix (m) [matrix of nodes, column sep=2em, row sep=1em]
{  
   & $\Sigma_{t\circ t' \circ t''}$&   \\
   $\Sigma_t \cup_{\tau} \Sigma_{t' \circ t''}$ &    & $\Sigma_{t \circ t'} \cup_{\tau} \Sigma_{t''}$  \\
 $\Sigma_t \cup_{g_3} \Sigma_{t' \circ t''}$ &  & $\Sigma_{t \circ t'} \cup_{g_4} \Sigma_{t''} $ \\
  $\Sigma_t \cup_{\bar h g_3} (\Sigma_{t'} \cup_{\tau} \Sigma_{t''})$ &  &$(\Sigma_{t} \cup_{\tau} \Sigma_{t'}) \cup_{g_4 \bar{h}^{-1}} \Sigma_{t''} $ \\
 $\Sigma_t \cup_{\Psi_2 \bar h g_3} (\Sigma_{t'} \cup_{g_2} \Sigma_{t''})$ & & $(\Sigma_{t} \cup_{g_1} \Sigma_{t'}) \cup_{g_4 \bar{h}^{-1} \Psi^{-1}_1} \Sigma_{t''} $  \\
 &$\bottomb (M\circ M' \circ M'')$ &  \\
};

\path[->, font=\scriptsize]
(m-1-2) edge  node[above] {$\bar h$} (m-2-1);

\path[->, font=\scriptsize]
(m-1-2) edge node[auto]{$\bar h$} (m-2-3);

\path[->, font=\scriptsize]
(m-2-1) edge node[auto] {$\Psi_3$} (m-3-1);

\path[->, font=\scriptsize]
(m-3-1) edge node[auto]{$\bar h$} (m-4-1);

\path[->, font=\scriptsize]
(m-4-1) edge node[auto]{$\Psi_2$} (m-5-1);

\path[->, font=\scriptsize]
(m-2-3) edge node[auto] {$\Psi_4$} (m-3-3);

\path[->, font=\scriptsize]
(m-3-3) edge node[auto]{$\bar h$} (m-4-3);

\path[->, font=\scriptsize]
(m-4-3) edge node[auto]{$\Psi_1$} (m-5-3);

\path[->, font=\scriptsize]
(m-5-1) edge  node[auto] {$\phi_B  \cup \phi'_B \cup \phi''_B$} (m-6-2);

\path[->, font=\scriptsize]
(m-5-3) edge node[auto]{$\phi_B  \cup \phi'_B \cup \phi''_B$} (m-6-2);

\end{tikzpicture}
\end{center}
On the left path, we have $\bar h \circ \Psi_3=\Psi_3\circ  \bar h$ since $\Psi_3$ changes only $\Sigma_t$ part and $\bar h$ changes only the $\Sigma_{t'\circ t''}$ part.
On the right path, $\Psi_4$ and $\bar h$ does not commute.
This diagram can be also written as follows.
Let $\tilde \Psi_4:=\bar h \Psi_4 \bar{h}^{-1}$. 
Then we have $\bar h \circ \Psi_4=\tilde \Psi_4 \circ \bar h$ and thus we have the following diagram.

\begin{center}

	\begin{tikzpicture}

\matrix (m) [matrix of nodes, column sep=2em, row sep=1em]
{  
   & $\Sigma_{t\circ t' \circ t''}$&   \\
   $\Sigma_t \cup_{\tau} \Sigma_{t' \circ t''}$ &    & $\Sigma_{t \circ t'} \cup_{\tau} \Sigma_{t''}$  \\
 $\Sigma_t \cup_{\bar h \tau} (\Sigma_{t'} \cup_{\tau} \Sigma_{t''})$ &  & $(\Sigma_{t} \cup_{\tau} \Sigma_{t'}) \cup_{\tau} \Sigma_{t''} $ \\
  $\Sigma_t \cup_{\bar h g_3} (\Sigma_{t'} \cup_{\tau} \Sigma_{t''})$ &  &$(\Sigma_{t} \cup_{\tau} \Sigma_{t'}) \cup_{g_4 \bar{h}^{-1}} \Sigma_{t''} $ \\
 $\Sigma_t \cup_{\Psi_2 \bar h g_3} (\Sigma_{t'} \cup_{g_2} \Sigma_{t''})$ & & $(\Sigma_{t} \cup_{g_1} \Sigma_{t'}) \cup_{g_4 \bar{h}^{-1} \Psi^{-1}_1} \Sigma_{t''} $  \\
 &$\bottomb (M\circ M' \circ M'')$ & \\
};

\path[->, font=\scriptsize]
(m-1-2) edge  node[above] {$\bar h$} (m-2-1);

\path[->, font=\scriptsize]
(m-1-2) edge node[auto]{$\bar h$} (m-2-3);

\path[->, font=\scriptsize]
(m-2-1) edge node[auto] {$\bar h$} (m-3-1);

\path[->, font=\scriptsize]
(m-3-1) edge node[auto]{$\Psi_3$} (m-4-1);

\path[->, font=\scriptsize]
(m-4-1) edge node[auto]{$\Psi_2$} (m-5-1);

\path[->, font=\scriptsize]
(m-2-3) edge node[auto] {$\bar h$} (m-3-3);

\path[->, font=\scriptsize]
(m-3-3) edge node[auto]{$\tilde \Psi_4$} (m-4-3);

\path[->, font=\scriptsize]
(m-4-3) edge node[auto]{$\Psi_1$} (m-5-3);

\path[->, font=\scriptsize]
(m-5-1) edge  node[auto] {$\phi_B \cup \phi'_B \cup \phi''_B$} (m-6-2);

\path[->, font=\scriptsize]
(m-5-3) edge node[auto]{$\phi_B  \cup \phi'_B \cup \phi''_B$} (m-6-2);

\end{tikzpicture}
\end{center}
Investigating the gluing maps, we obtain the following diagram.

\begin{center}

	\begin{tikzpicture}

\matrix (m) [matrix of nodes, column sep=2em, row sep=1em]
{  
   & $\Sigma_{t\circ t' \circ t''}$&   \\
   $\Sigma_t \cup_{\tau} \Sigma_{t' \circ t''}$ &    & $\Sigma_{t \circ t'} \cup_{\tau} \Sigma_{t''}$  \\
 $\Sigma_t \cup_{ \tau} (\Sigma_{t'} \cup_{\tau} \Sigma_{t''})$ &  & $(\Sigma_{t} \cup_{\tau} \Sigma_{t'}) \cup_{\tau} \Sigma_{t''} $ \\
  $\Sigma_t \cup_{\bar h g_3} (\Sigma_{t'} \cup_{\tau} \Sigma_{t''})$ &  &$(\Sigma_{t} \cup_{\tau} \Sigma_{t'}) \cup_{g_4 \bar{h}^{-1}} \Sigma_{t''} $ \\
 $\Sigma_t \cup_{g_1} (\Sigma_{t'} \cup_{g_2} \Sigma_{t''})$ & & $(\Sigma_{t} \cup_{g_1} \Sigma_{t'}) \cup_{g_2} \Sigma_{t''} $  \\
 &$\bottomb (M\circ M' \circ M'')$ &  \\
};

\path[->, font=\scriptsize]
(m-1-2) edge  node[above] {$\bar h$} (m-2-1);

\path[->, font=\scriptsize]
(m-1-2) edge node[auto]{$\bar h$} (m-2-3);

\path[->, font=\scriptsize]
(m-3-1) edge  node[above] {$=$} (m-3-3);

\path[->, font=\scriptsize]
(m-2-1) edge node[auto] {$\bar h$} (m-3-1);

\path[->, font=\scriptsize]
(m-3-1) edge node[auto]{$\Psi_3$} (m-4-1);

\path[->, font=\scriptsize]
(m-4-1) edge node[auto]{$\Psi_2$} (m-5-1);

\path[->, font=\scriptsize]
(m-2-3) edge node[auto] {$\bar h$} (m-3-3);

\path[->, font=\scriptsize]
(m-3-3) edge node[auto]{$\tilde \Psi_4$} (m-4-3);

\path[->, font=\scriptsize]
(m-4-3) edge node[auto]{$\Psi_1$} (m-5-3);

\path[->, font=\scriptsize]
(m-5-1) edge  node[auto] {$\phi_B \cup \phi'_B \cup \phi''_B$} (m-6-2);

\path[->, font=\scriptsize]
(m-5-3) edge node[auto]{$\phi_B \cup \phi'_B \cup \phi''_B$} (m-6-2);

\end{tikzpicture}
\end{center}
Here the top pentagon is commutative up to isotopy: this is again similar to the proof of associativity of the fundamental group.
The bottom heptagon is also commutative up to isotopy. 
To see this, first note that $\tilde \Psi_4$ is identity on $\Sigma_t \subset \Sigma_t \cup_{\tau} \Sigma_{t'}$ since $\Psi_4$ is identity outside of a collar neighborhood of $\rightb \Sigma_{t \circ t'} \subset \bar {h}^{-1} (\Sigma_{t'})$.
Thus the restriction of $\bar \Psi_4$  on $\Sigma_{t'}$ part is isotopic to $\Psi_2$ by Lemma \ref{lem:Turaev Appendix III} since on the boundary $\rightb \Sigma_{t'}$ they agree.
Since $\Psi_1$ and $\Psi_2$ commute, the heptagon is commutative up to isotopy.
By Lemma \ref{lem:isotopy equivalence}  the associativity on the level of classses holds.

\subsubsection{Units for horizontal composition}

As in the case of vertical composition, we use formal units for horizontal composition.
This means that for each object $\nstand{n}$ of $\Co$, we add the formal identity object $\id_n$ to $\Co(\nstand{n}, \nstand{n})$ and we also add the formal identity 2-morphism $\id_{\id_n}$ between the object $\id_n$.
These formal identities act as identity for the horizontal composition.

\subsubsection{Interchange law}
We check the interchange law.
For four 2-morphisms 
\begin{align*}
&[(M_1, \phi_1)]: \Sigma_{t_1} \Rightarrow \Sigma_{t_2}: \nstand{l} \to \nstand{m},  \qquad &[(M_2, \phi_2)]: \Sigma_{t_2} \Rightarrow \Sigma_{t_3}: \nstand{l} \to \nstand{m},\\
&[(M_1', \psi_1)]: \Sigma_{s_1} \Rightarrow \Sigma_{s_2}: \nstand{m} \to \nstand{n}, \qquad &[(M_2', \psi_2)]: \Sigma_{s_2} \Rightarrow \Sigma_{s_3}: \nstand{m} \to \nstand{n},
\end{align*}

\begin{center}
\begin{tikzcd}
  \nstand{l} \arrow[bend left=50]{r}[name=U,below]{}{\Sigma_{t_1}}
     \arrow{r}[name=M,below]{}{\Sigma_{t_2}}
    \arrow[bend right=50]{r}[name=D]{}{\Sigma_{t_3}}
&\nstand{m}
      \arrow[bend left=50]{r}[name=U',below]{}{\Sigma_{s_1}}
     \arrow{r}[name=M',below]{}{\Sigma_{s_2}}
    \arrow[bend right=50]{r}[name=D']{}{\Sigma_{s_3}}
& \nstand{n}     
\end{tikzcd}
\end{center}
the interchange law says that the following equality holds;
\[ ([M_1]\circ[M_1'])\cdot ([M_2] \circ [M_2'])=( [M_1]\cdot [M_2]) \circ ([M_1'] \cdot [M_2']) \]
as a 2-morphism $\Sigma_{t_1\circ s_1}\Rightarrow \Sigma_{t_3\circ s_3}: \nstand{l}\to \nstand{n}$.

As a manifold both sides are the same.
Hence we only need to check whether
\[(\phi_1\circ_{\text{h}} \psi_1)\cdot_{\text{v}} (\phi_2 \circ_{\text{h}}\psi_2)  =(\phi_1 \cdot_{\text{v}} \phi_2)\circ_{\text{h}} (\psi_1\cdot_{\text{v}} \psi_2). \]
This equality is true because that horizontal composition does not change the parametrizations on the side boundaries and vertical composition does not change the parametrizations on the top and the bottom boundaries.

\section{A bicategory of the Kapranov-Voevodsky 2-vector spaces}\label{sec:A 2-category of the Kapranov-Voevodsky 2-vector spaces}
The target bicategory of our extended TQFT will be the Kapranov-Voevodsky (KV) 2-vector spaces.
The reason that we chose the KV 2-vector spaces as a target algebraic bicategory is that it is a natural extension of the usual category of vector spaces and the calculations are very explicit.

We recall the relevant definitions.
\begin{Definition}
\begin{enumerate}
Let $K$ be a commutative ring.
\item
A \textit{2-matrix} is an $m\times n$ matrix such that the $(i,j)$-component is a projective module over $K$.

\item
A \textit{2-homomorphism} from an $m\times n$ 2-matrix $V$ to an $m \times n$ 2-matrix $W$ is an $m\times n$ matrix of $K$-homomorphisms.
In other words, the $(i, j)$-components of a matrix of homomorphisms is a $K$-homomorphism from $V_{ij}$ to $W_{ij}$.
\item Two 2-matrices $V$ and $W$ are said to be \textit{isomorphic} if there is a 2-homomorphism $T$ from $V$ to $W$ such that each entry of $T$ is an isomorphism.
\end{enumerate}
\end{Definition}

\begin{Definition}
The \textit{Kapranov-Voevodsky 2-vector spaces}, $\KV$, consists of the followings.
\begin{enumerate}
\item The \textit{objects} of $\KV$ are symbols $\{n\}$ for non-negative integers $n$.
\item A \textit{1-morphism} from $\{m\}$ to $\{n\}$ is an $(m\times n)$ 2-matrix $V$. We denote a 1-morphism from $\{m\}$ to $\{n\}$ by  $V: \{m\}\to \{n\}$.
\item A \textit{2-morphism} from a 1-morphism $V:\{m\}\to \{n\}$ to a 1-morphism $W:\{m\}\to \{n\}$ is a 2-homomorphism from $V$ to $W$.
\end{enumerate}

%

Usual matrix calculations extend to this setting if we replace a multiplication by $\otimes$ and an addition by $\oplus$. 
Horizontal composition is given by matrix multiplication and vertical composition is given by the composition of each entries.
With these composition operations, the Kapranov-Voevodsky 2-vector spaces $\KV$ is indeed a bicategory.
(For details, see \cite{KV1994} on page 226.)

\end{Definition}

\section{Review and Modification of the Reshetikhin-Turaev TQFT}\label{sec:Review and Modification of the Reshetikhin-Turaev TQFT}
Our construction of a projective pseudo 2-functor from the 2-category $\Co$ to the Kapranov-Voevodsky 2-vector spaces $\KV$ requires the original Reshetikhin-Turaev theory.
In this section we review some of the relevant part of the RT theory.

\subsection{Operator Invariant}\label{subsec:operator invariant}
One of the key ingredient to construct the RT TQFT is so called the ``operator invariant''.
Let $\V$ be a modular category. (More generally, the operator invariant exists for a strict ribbon category $\V$.)
Define the category $\Rib$ of the ribbon graphs over $\V$ as follows.
The objects of $\Rib$ are finite sequences of the form $((V_1, \epsilon_1), \dots, (V_m, \epsilon_m))$, where $V_i$ is an object of the modular category $\V$ and $\epsilon_i$ is either $\pm1$ for $i=1, \dots, m$.
A morphism $\eta \to \eta'$ in $\Rib$ is an isotopy type of a $v$-colored ribbon graph over $\V$ such that $\eta$ (resp. $\eta'$) is the sequence of colors and directions of those bands which hit the bottom (resp. top) boundary intervals.
The downward direction near the corresponding boundary corresponds $\epsilon=1$, and $\epsilon=-1$ corresponds to the band directed up.
For example, the ribbon graph drawn in Figure \ref{fig:ribbon graph} represents a morphism from $((V_1, -1), (V_2,1), (V_3, 1), (U, 1))$ to $((V_2, -1), (V_1, -1), (V_3, 1), (V, 1))$.
\begin{figure}[h]
\center
\includegraphics[width=2.6in]{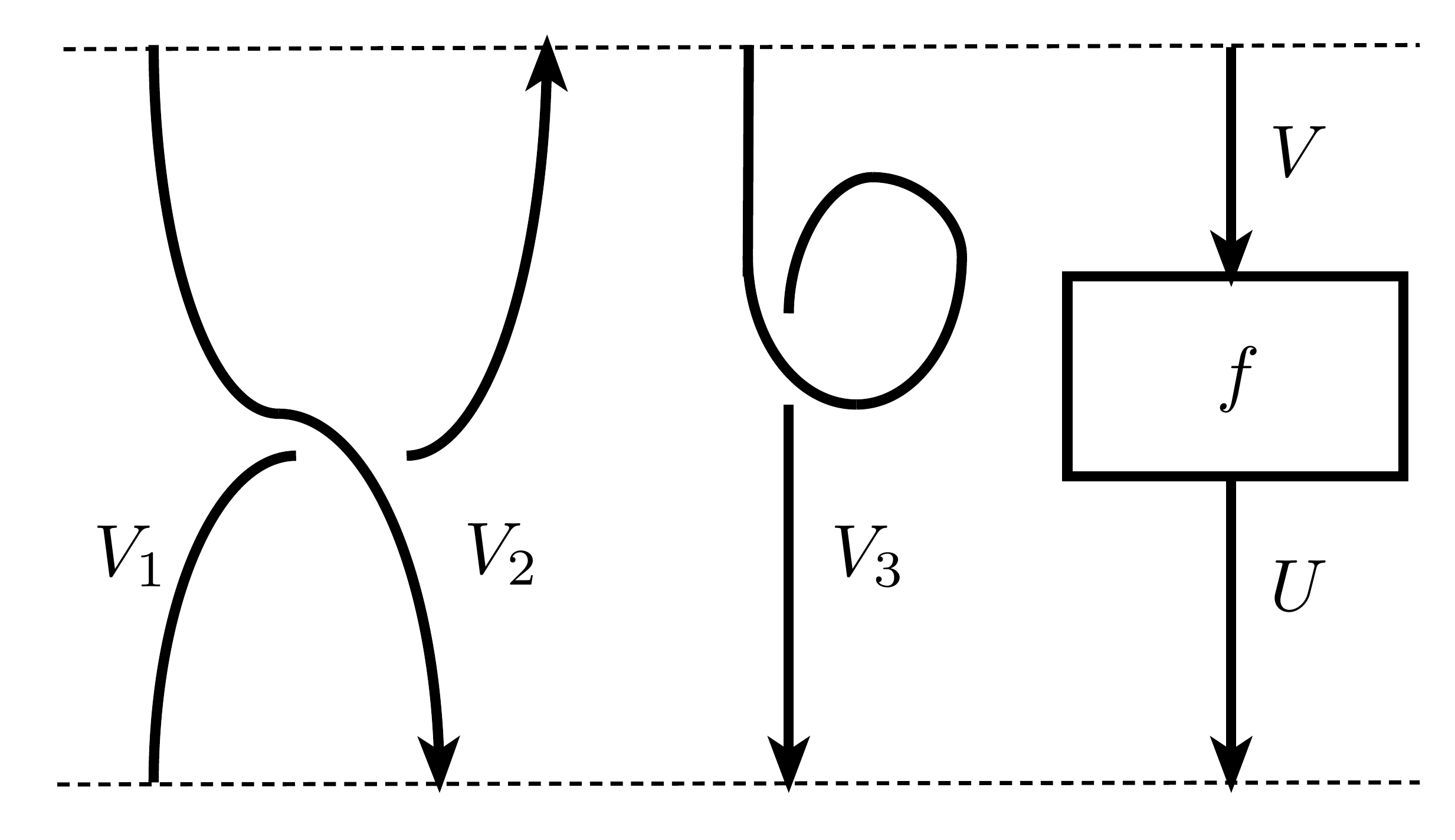}
\caption{$v$-colored Ribbon graph}
\label{fig:ribbon graph}
\end{figure}
The composition of morphisms of $\Rib$ is given by concatenation of ribbon graphs.
The juxtaposition of ribbon graphs provides $\Rib$ with the structure of a monoidal category.
Then it is a fact that there is a unique monoidal functor $F=F_{\V}: \Rib \to \V$ satisfying the following conditions:
\begin{enumerate}
\item $F$ transforms any object $(V, +1)$ into $V$ and any object $(V, -1)$ into $V^*$
\item $F$ maps a crossing ribbon graph to a braiding, a twist ribbon graph to a twist, a cup like band and a cap like  band to a corresponding duality map in $V$.
(The X-shape ribbon in Figure \ref{fig:ribbon graph} is one of crossing ribbons and the once-curled ribbon in the middle is one of twist ribbons. The others obtained by changing directions and crossings.)
\item For each elementary $v$-colored ribbon graph $\Gamma$, we have $F(\Gamma)=f$, where $f$ is the color of the only coupon of $\Gamma$.
(An example of elementary $v$-colored ribbon graph is the ribbon with one coupon colored by $f$ on the right in Figure \ref{fig:ribbon graph}. In general there may be multiple vertical bands attached to one coupon.)
\end{enumerate}
(For details, see Theorem I.2.5 in \cite{Turaev10}.)
The morphism $F(\Omega)$ associated to a $v$-colored ribbon graph $\Omega$ is called the \textit{operator invariant} of $\Omega$.

\subsection{Modular categories}
A modular category is an input for the RT TQFT.
The definition of a modular category is a ribbon category with a finite many simple objects satisfying several axioms.
(See \cite{Turaev10} for a full definition and refer to \cite{MR1321145} for a ribbon category.)
The objects of a modular category are used as decorations of surfaces as we saw in Section \ref{sec:A 2-category of cobordisms with corners} .
Here we set several notations and state an important lemma.

Let $\V$ be a modular category with a finite set $\{V_i\}_{i\in I}$ of simple objects, where $I$ is a finite index set $I=\{1,2, \dots, \bfk\}$.
We assume that $V_1$ is the unit object $\1$ of $\V$.
The ring $K:=\Hom(\1, \1)$ is called the \textit{ground ring}.
The ground ring $K$ is known to be commutative.
We assume that $\V$ has an element $\D$ called a \textit{rank} of $\V$ given by the formula
\begin{equation*}\label{equ:rank}
\D^2= \sum_{i\in I} \left( \dim(V_i) \right)^2.
\end{equation*}
(This assumption is not essential.)
 Besides the rank $\D$, we need another element $\Delta$ defined as follows.
The modular category has a twist morphism $\theta_{V}: V \to V$ for each object $V$ of $\V$.
Since $V_i$ is a simple object, the twist $\theta_{V_i}$ acts in $V_i$ as multiplication by a certain $v_i \in K$.
Since the twist acts via isomorphism, the element $v_i$ is invertible in $K$.
We set 
\begin{equation*}\label{equ:Delta}
\Delta=\sum_{i \in I} v_i^{-1} \left( \dim(V_i) \right)^2 \in K.
\end{equation*}
The elements $\D$ and $\Delta$ are known to be invertible in $K$.
The following lemma is very important.

\begin{lemma}\label{lem:sum over simple 2}
For any objects $V, W$ of the modular category $\V$, there is a canonical $K$-linear splitting
\begin{equation*}
\Hom(\1, V \otimes W)= \bigoplus_{i \in I} \left( \Hom(\1, V\otimes V_{i}^* )\otimes_K \Hom(\1, V_{i} \otimes W) \right)
\end{equation*}
The isomorphism $u$ transforming the right-hand side into the left-hand side is given by the formula
\begin{equation}\label{equ: cap isom} 
u_i: x \otimes y \mapsto (\id_V\otimes d_{V_i} \otimes \id_W) (x \otimes y),
\end{equation}
where $x \in \Hom(\1, V\otimes V^*_{i})$, $y\in \Hom(\1, V_{i} \otimes W)$.
The map (\ref{equ: cap isom}) is given graphically as in Figure \ref{fig:the map u_i}.
\end{lemma}
For a proof, see Lemma I\hspace{-.1em}V.2.2.2 in \cite{Turaev10}.
\begin{figure}[h]
\center
\includegraphics[width=3.5in]{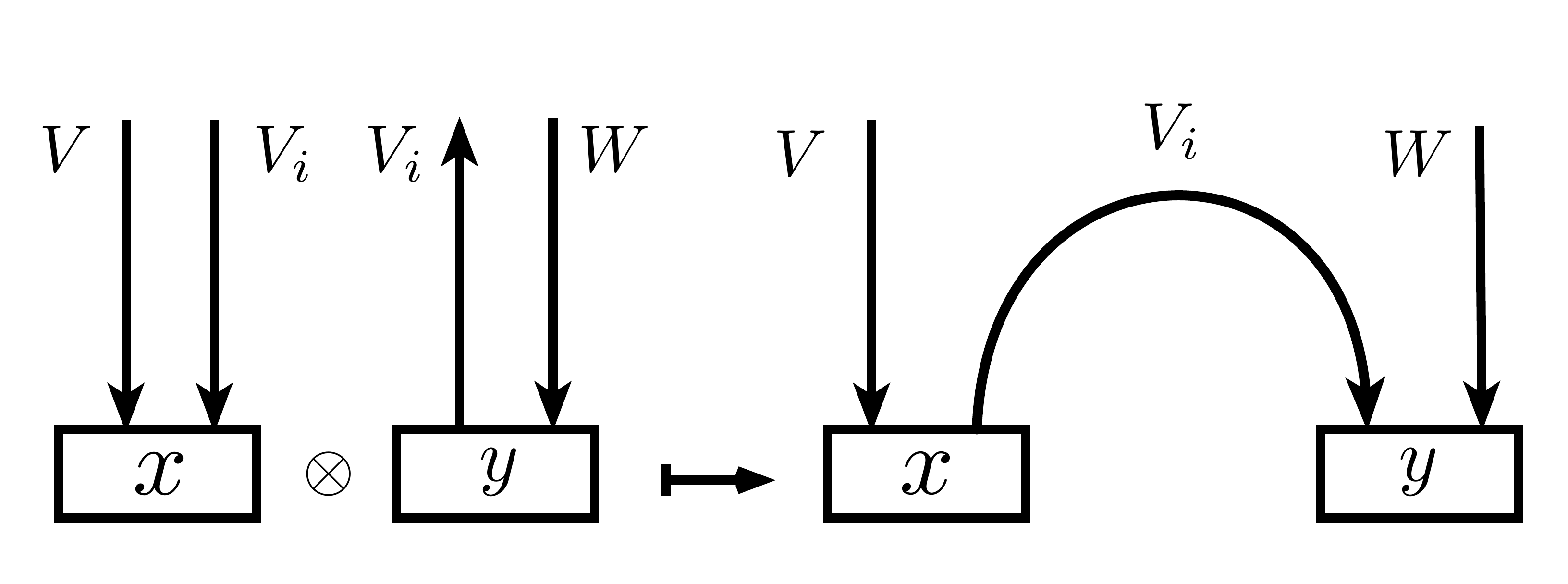}
\caption{The map $u_i$}
\label{fig:the map u_i}
\end{figure}

\subsection{Invariants of 3-manifolds with ribbon graphs}
We review an invariant of closed 3-manifolds with ribbon graphs sitting inside manifolds.
When we will construct the RT TQFT below, a cobordism will be turned into a closed 3-manifold and we apply this invariant.
Let $M$ be a closed connected oriented 3-manifold.
Let $\Omega$ be a $v$-colored ribbon graph over $\V$ in $M$.
Present $M$ as the result of surgery on $S^3$ along a framed link $L$ with components $L_1, \dots, L_m$. Fix an arbitrary orientation of $L$.
This choice can be shown to be irrelevant.
We may assume that  $\Omega \subset S^3 \setminus U$, where $U$ is a closed regular neighborhood of $L$ in $S^3$ by applying an isotopy to $\Omega$ if necessary. 
Denote by $\col(L)$ the set of all mappings from the set of components of $L$ into the index set $I$.
For each $\lambda\in \col(L)$, the pair $(L, \lambda)$ determines a colored ribbon graph $\Gamma(L, \lambda)$ formed by $m$ annuli.
The cores of these annuli are the oriented circles $L_1, \dots, L_m$, the normal vector field on the cores transversal to the annuli represents the given framing.
The color of the $i$-th annuli is $V_{\lambda(L_i)}$.
Since the union $\Gamma(L, \lambda)\cup \Omega$ is a $v$-colored ribbon graph and it has no free ends, the operator invariant $F(\Gamma(L, \lambda) \cup \Omega) \in K=\End(\1)$.
Set
\[ \{L, \Omega\}=\sum_{ \lambda \in \col(L)}\dim(\lambda)F(\Gamma(L, \lambda) \cup \Omega) \in K=\End(\1), \]
where
 \[\dim(\lambda)=\prod_{n=1}^m \dim\left( \lambda(L_n) \right) \mbox{ with } \dim(i):=\dim(V_i).\]
Set
\begin{equation}\label{equ:invariant tau}
\tau(M, \Omega)=\Delta^{\sigma(L)} \D^{-\sigma(L)-m-1} \{L, \Omega\}.
\end{equation}
Here $\sigma(L)$ is the signature of the surgery link $L$.
It is an important fact that $\tau(M, \Omega)$ is a topological invariant of the pair $(M, \Omega)$. 
The invariants $\tau$ extends to $v$-colored ribbon graphs in any non-connected closed oriented 3-manifold $M$ by the formula
\[\tau(M, \Omega)=\prod_{r} \tau(M_r, \Omega_r),\]
where $M_r$ runs over the connected components of $M$ and $\Omega_r$ denotes the part of $\Omega$ lying in $M_r$. 

The invariant $\tau(M, \Omega)$ satisfies the following multiplicativity law:
\begin{equation}\label{equ:tau multiplicative}
\tau(M_1 \# M_2, \Omega_1 \sqcup \Omega_2)=\D \tau(M_1, \Omega_1)\tau(M_2, \Omega_2),
\end{equation}
where $\Omega_1$ and $\Omega_2$ are $v$-colored ribbon graphs in closed connected oriented 3-manifolds $M_1$ and $M_2$, respectively.
This can be seen as follows.
Let $L$ and $L'$ be surgery links for $M_1$ and $M_2$, respectively.
Then the ribbon $L\cup \Omega_1$ and $L'\cup \Omega_2$ sitting inside the same $S^3$ separately is a pair of a surgery link for $M_1 \# M_2$ and $\Omega_1\sqcup \Omega_2$.
Then the formula (\ref{equ:tau multiplicative}) follows from the direct calculation using the defining formula (\ref{equ:invariant tau}).
 
\subsection{Construction of the (non-extended) Reshetikhin-Turaev TQFT}
Now we review the construction of the original RT TQFT for closed surfaces and cobordisms without corners.
Instead of going over the original construction, we modify it so that it adapts in our setting.
Let us fix a modular category $\V$.
As we noted in Section \ref{subsec:Remark}, the original theory concerns only decorated types of the form
\[t=(0, 0; (W_1, \nu_1), \dots, (W_m, \nu_m), 1,1,\dots, 1),\]
where $W_i$ is an object of a modular category and $\nu_i$ is either $1$ or $-1$ for $i=1, \dots, m$.
We can modify the theory using more general decorated types
\begin{equation}\label{equ:type 0 0}
t=(0,0; a_1, \dots, a_p),
\end{equation}
where $a_i$ is non-negative integer or the pair of an object of the modular category $\V$ and a sign $\pm 1$.
Note that the first two entries should be zero since these numbers encode the number of boundary components of surfaces, which is zero for the original RT theory.
In the following, we review the RT theory replacing the original decorated types with types  of the form in (\ref{equ:type 0 0}).

Let $t=(0,0; a_1, \dots, a_p)$ and $s=(0,0; b_1, \dots, b_q)$ be decorated types whose first two entries are zero.
Let $[(M, \phi)]: \Sigma_t \Rightarrow \Sigma_s:\lempty \to \rempty$ be a 2-morphism of the 2-category $\Co$.
Note that since the first two entries of the types $t, s$ are zero, we have $\Sigma(\phi)=\Sigma_t \sqcup \Sigma_s^-$. 
The images $\phi(\Sigma_t)$ and $\phi(\Sigma_s^-)$ are denoted by $\bottomb M$ and $\topb M$, respectively as before.
The Reshetikhin-Turaev TQFT is a pair of assignments $(\tau, \T)$ which will be constructed below so that $\tau(M)$ is a $K$-homomorphism from a projective module $\T(\bottomb  M)$ to a projective module $\T(\topb M)$.

\subsubsection{Definition of the projective module $\T(S)$}
For a decorated surface $S$ of type $t=(0,0; a_1, \dots, a_p)$, the projective module $\T(S)$ is defined as follows.
In a non-precise but instructive way, we think that $\T(S)$ is a projective module of all possible colors for the coupon of the ribbon graph $R_t$ in Figure \ref{fig:Rtnew}.
To make it precise, we set up several notations.
First, we define an object $H_i^a$ of the modular category $\V$ as follows.
Here $a$ is either a positive integer or $a=(W, \nu)$ is a signed object of the modular category $\V$.
For a positive integer $a$ and $i=(i_1,\dots, i_a)\in I^a$, we set
\begin{equation}\label{equ:H_i^a}
H^a_i=V_{i_1}\otimes V_{i_2}\otimes \cdots \otimes V_{i_a} \otimes V^*_{i_a}\otimes \cdots \otimes V_{i_2}^* \otimes V^*_{i_1}.
\end{equation}
If $a=(W, \nu)$ is a signed object of $\V$, we set $I^a$ to be a set of only one element and 
set $H^a_i=W^{\nu}$ for $i\in I^a=\{i\}$.
Here we used the letter $i$ for the unique element of $I^a$ to streamline notations.
Note that the tensor product in (\ref{equ:H_i^a}) can be used as a color for rainbow like bands in $R_t$ corresponding to an integer entry $a$.
For a type $t=(m, n; a_1, a_2, \dots, a_p)$, we write 
\begin{equation}\label{equ:I^t}
I^t:=I^{a_1} \times \cdots \times I^{a_p}. 
\end{equation}
For $\zeta=(\zeta_1, \dots, \zeta_p) \in I^t$ with $\zeta_1=(\zeta_1^1, \dots, \zeta_1^{a_1}), \dots, \zeta_p=(\zeta_p^1, \dots, \zeta_p^{a_p})$, we set
\begin{equation}\label{equ:Phi.t.zeta}
\Phi(t; \zeta)= H^{a_1}_{\zeta_1} \otimes H^{a_2}_{\zeta_2} \otimes \cdots \otimes H^{a_p}_{\zeta_p}.
\end{equation}
Note that each choice of $\zeta=(\zeta_1, \dots, \zeta_p) \in I^t$ determines the color of ribbon graph $R_t$ except for the coupon.
The coupon can be colored by a morphism from the monoidal unit $\1$ to the element $\Phi(t; \zeta)$.
Thus all the possible colors of the coupon of $R_t$ varying $\zeta \in I^t$ is
\begin{equation}\label{equ:T(S)}
 \T(S):=\bigoplus_{\zeta \in I^t} \Hom \big(\1, \Phi(t; \zeta) \big)
\end{equation}
and we define it to be $\T(S)$.
Since $\Phi(t;\zeta)$ is an object of the modular category $\V$, $\T(S)$ is a projective module over the grand ring $K=\Hom(\1, \1)$.

\subsubsection{Definition of $K$-homomorphism $\tau(M)$}\label{subsec:tau M}
Let $[(M, \phi)]: \Sigma_t \Rightarrow \Sigma_s: \lempty \to \rempty$ be a 2-morphism of the 2-category $\Co$ and fix a representative $(M, \phi)$.
We explain the construction of the corresponding $K$-homomorphism $\tau(M)$ from $\T(\bottomb M)$ to $\T(\topb M)$.
Glue the standard handlebody $U_t$ and $U_s^-$ to $M$ along the parametrization $\phi$.
The resulting manifold $\tilde{M}$ is a closed 3-manifold with ribbon graph $\tilde{\Omega}$ sitting inside $\tilde{M}$.
The ribbon graph $\tilde{\Omega}$ is obtained by gluing the ribbon graph in $M$ and the ribbon graph $R_t$ and $-R_s$ sitting inside the standard handlebodies $U_t$ and $U_s^-$.
The ribbon graph $\tilde{\Omega}$ is not $v$-colored since the cap-like rainbow bands and the cup-like rainbow bands and the coupons of $R_t$ and $-R_s$ in the newly glued handlebodies are not colored.
By its definition, each element of the module $\T(\bob M)$ determines a color of $R_t$ and each element of $\T(\topb M)^*$ determines a color of $-R_s$.
For such a choice of color $y$ of $\tilde{\Omega}$ we obtain a $v$-coloring of $\tilde{\Omega}$.
Applying the invariant $\tau$ of $v$-colored ribbon graph in a closed 3-manifold defined in (\ref{equ:invariant tau}), we obtain a certain element $\tau(\tilde{M}, \tilde{\Omega}, y)\in K$.
This induces a $K$-homomorphism $\T(\bottomb M) \otimes_K \T(\topb M)^* \to K$.
Taking adjoints, we get a $K$-homomorphism 
\begin{equation}\label{equ:after adjoint}
T(\bottomb M) \to \T(\topb M).
\end{equation}
To finish the construction of $\tau(M):T(\bottomb M) \to \T(\topb M)$, we compose the above $K$-homomorphism with an endomorphism $\eta$ defined as follows.
Let $S$ be a connected parametrized $d$-surface of type $t=(0,0; a_1, \dots, a_p)$.
The endomorphism $\eta(S): \T(S) \to T(S)$ preserves the splitting (\ref{equ:T(S)}) and acts in each summand $\Hom(\1, \Phi(t;\zeta))$ as multiplication by $\D^{1-g}\dim(\zeta)$, where $g$ is the sum of integer entries of the type $t$ and
\[\dim(\zeta):=\prod_{i=1}^p\dim(\zeta_i) \]
with
\[\dim(\zeta_i):=
\begin{cases}
\prod_{l=1}^{a_i} \dim(\zeta_i^l) & \mbox{ if } a_i \in \Z \\
1 & \mbox{ if } a_i \mbox{ is a mark.}
\end{cases}
\]
Recall that $\dim(\zeta_i^l)$ denotes the dimension of the simple object $V_{\zeta_i^l}$.
Now we complete the construction of $\tau(M):T(\bottomb M) \to \T(\topb M)$ by composing the $K$-homomorphism (\ref{equ:after adjoint}) with $\eta(\topb M): \T(\topb M) \to \T(\topb M)$.
The pair $(\tau, \T)$ is the \textit{Reshetikhin-Turaev TQFT}.
In general, this is not a functor because it has \textit{gluing anomaly}.

\subsubsection{Explicit Formula for the homomorphism $\tau(M)$}\label{subsec:explicit formula for tau(M)}
We will develop a technique of presentation of a decorated connected 3-cobordism $M$ by a certain ribbon graph in $\R^3$ and give the explicit formula to calculate the homomorphism $\tau(M)$ using this ribbon graph.

First as we are in the closed case, let $t=(0,0; a_1, \dots, a_p)$ and $s=(0,0; b_1, \dots, b_q)$ be decorated types whose first two entries are zero.
Let $[(M, \phi)]: \Sigma_t \Rightarrow \Sigma_s:\lempty \to \rempty$ be a 2-morphism of the 2-category $\Co$ and fix a representative $(M, \phi)$.
We assume that the cobordism $M$ is connected.
As above, we glue the standard handlebodies $U_t$ and $U_s^-$ using the parametrization $\phi$ to $M$.
We obtain the closed connected 3-manifold $\tilde{M}$ and the partially colored ribbon graph $\tilde{\Omega}$ sitting inside $\tilde{M}$.
Present $\tilde{M}$ as the result of surgery on a framed link $L$ in $S^3$.
Namely, we have a homeomorphism from $M_L$ to $\tilde{M}$, where $M_L$ is the resulting 3-manifold of surgery along $L$.
Let $H=U_t \cup U_s^- \subset S^3$.
We may think that $H$ is a subset of $S^3 \setminus T(L)\subset M_L$, where $T(L)$ is a closed tubular neighborhood of the link $L$.
Restricting this homeomorphism, we see that the pair $(M, \phi)$ is equivalent to $(M_L, \id)$.
Thus we may assume that $\tilde{\Omega}$ is the union of $R_t, -R_s$ and a surgery link and a ribbon graph in $M$ in $\R^2 \times [0, 1] \subset \R^3 \subset S^3$.
Of course, the surgery link might be tangled with $R_t$ and $-R_s$.
By isotopy, we pull $R_t$ down so that the top of the coupon of $R_t$ lies in $\R \times \{0\}\times \{0\}$.
Also we move $-R_s$ up so that the bottom of the coupon $-R_s$ lies in $\R \times \{0\} \times \{1\}$ and move the rest of the ribbon in $\R^2\times (0, 1)$.
See Figure \ref{fig:special ribbon graph}.

\begin{figure}[h]
\center
\includegraphics[width=3.8in]{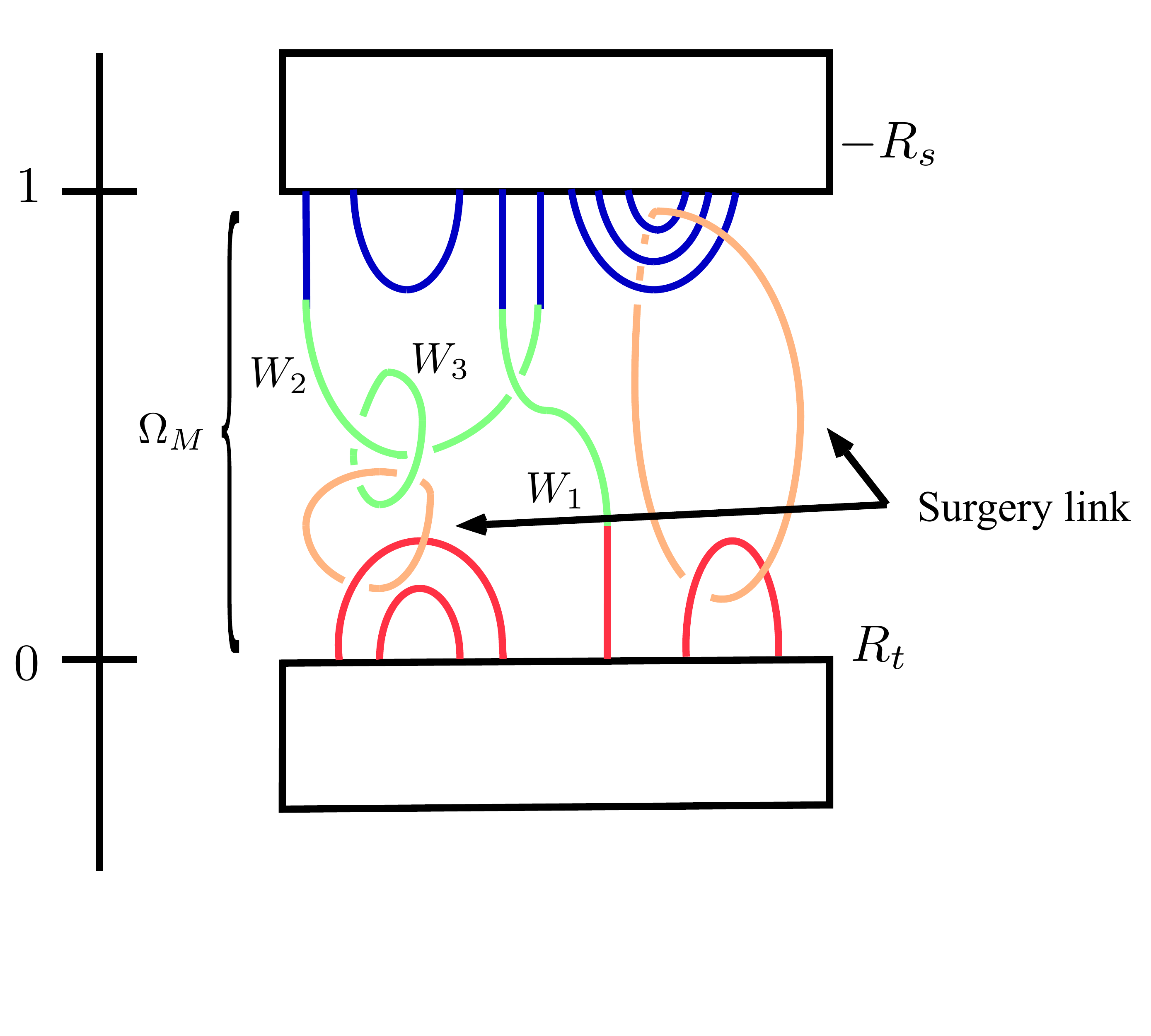}
\caption{Special ribbon graph}
\label{fig:special ribbon graph}
\end{figure}
Let $\Omega_M$ be a ribbon graph obtained by removing the coupons of $R_t$ and $-R_s$ from $\tilde{\Omega}$.
We call $\Omega_M$ \textit{special ribbon graph} for $M$.

We now give an explicit formula for computing the homomorphism $\tau(M): \T(\bottomb M) \to \T(\topb M)$ from the operator invariants of the special ribbon graph $\Omega_M$ (after coloring $\Omega_M$).
With respect to the splittings (\ref{equ:T(S)}) of $\T(\bottomb M)$ and $\T(\topb M)$, the homomorphism $\tau(M)$ may be presented by a block matrix $\tau_{\zeta}^{\eta}$, where 
\[\zeta=(\zeta_1, \dots, \zeta_p) \in I^t \mbox{ with } \zeta_1=(\zeta_1^1, \dots, \zeta_1^{a_1}) \in I^{a_1}, \dots, \zeta_p=(\zeta_p^1, \dots, \zeta_p^{a_p})\in I^{a_p}\]
and
\[\eta=(\eta_1, \dots, \eta_q) \in I^s \mbox{ with } \eta_1=(\eta_1^1, \dots, \eta_1^{b_1}) \in I^{b_1}, \dots, \eta_q=(\eta_q^1, \dots, \eta_q^{b_q})\in I^{b_q}.\]
Each such $\zeta\in I^t$ determines a coloring of the cap-like rainbow bands of $R_t$ in $\Omega_M$.
Similarly each such $\eta \in I^s$ determines a coloring of the cup-like rainbow bands of $-R_s$ in $\Omega_M$.
Therefore a pair $(\zeta, \eta)\in I^t \times I^s$ determines a coloring of uncolored bands of $\Omega_M$.
Note that the surgery link $L$ in $\Omega_M$ is not colored.
(More precisely, the ribbon graph obtained by thickening the surgery link along framing.)
Every element $\lambda \in \col(L)$ determines coloring of the surgery ribbon.
Thus every element $(\zeta, \eta, \lambda)\in I^t \times I^s \times \col(L)$ determines a $v$-coloring of $\Omega_M$.
Denote the resulting $v$-colored ribbon graph in $R^3$ by $(\Omega_M, \zeta, \eta, \lambda)$.
Consider its operator invariant $F(\Omega_M, \zeta, \eta, \lambda): \Phi(t; \zeta)\to \Phi(s; \eta)$ defined in Section \ref{subsec:operator invariant}.
The composition of a morphism $\1 \to \Phi(t; \zeta)$ with $F(\Omega_M, \zeta, \eta, \lambda)$ defines a $K$-linear homomorphism $\Hom(\1, \Phi(t; \zeta))\to \Hom(\1, \Phi(s;\eta))$ denoted by $F_0(\Omega_M, \zeta, \eta, \lambda)$.
It follows from the very definition of $\tau(M)$ given in Section \ref{subsec:tau M} that

\begin{equation}\label{equ:tau zeta eta}
\tau_{\zeta}^{\eta}= \Delta^{\sigma(L)} \D^{-g^+ -\sigma(L) - m} \dim(\eta) \sum_{\lambda \in \col(L)} \dim(\lambda) \F_0(\Omega_M, \zeta, \eta, \lambda),
\end{equation}
where $g^+$ is the sum of the integer entries of $s$ and 
\[\dim(\eta):=\prod_{i=1}^q\dim(\eta_i) \]
with
\[\dim(\eta_i):=
\begin{cases}
\prod_{l=1}^{b_i} \dim(\eta_i^l) & \mbox{ if } b_i \in \Z \\
1 & \mbox{ if } b_i \mbox{ is a mark.}
\end{cases}
\]

\section{An extended TQFT $\X$}\label{sec:An extended TQFT}
Now we proceed to construct a projective pseudo 2-functor $\X$ from the 2-category $\Co$ of decorated cobordisms with corners  to the Kapranov-Voevodsky 2-vector spaces $\KV$ that will be our extension of the Reshetikhin-Turaev TQFT functor.
For our convention of the language of 2-category, see Appendix.
As in the Reshetikhin-Turaev theory, we fix a modular category $\V$ with $\bfk$ simple objects $V_1, \dots, V_{\bfk}$.
Let $I$ be the index set of simple objects, hence its cardinality is $|I|=\bfk$.
Let $K$ denote the ground ring $\Hom(\1,\1)$.

\subsection{$\X$ on objects}
Each object $\nstand{n}$ of $\Co$ for a natural number $n$ is mapped by $\X$ to the object $\{\bfk^n\}$ of $\KV$.
The formal symbol objects $\lempty$ and $\rempty$ are mapped to the object $\{1\}$.

\subsection{$\X$ on 1-morphisms}\label{sec:X on 1-morphisms}
For a 1-morphism $\Sigma_t: \nstand{m} \to \nstand{n}$ with a type 
\[t=(m, n;a_1, a_2, \dots, a_p),\]
 we need to define a $\bfk^m\times \bfk^n$ 2-matrix $\X(\Sigma_t)$.
Using the lexicographic order of the power sets $I^m$ and $I^n$, pick $i$-th element of $I^m$ and $j$-th element of $I^n$.
Abusing notation, we write the $i$-th element of $I^m$ as $i=(i_1, i_2, \dots, i_m)$ and the $j$-th element of $I^n$ as $j=(j_1, \dots, j_n)$.
For each decorated type $t$, we defined the ribbon graph $R_t$ (Figure \ref{fig:Rtnew}).
In the ribbon graph $R_t$, there are $m=L(t)$ uncolored bands on the left and $n=R(t)$ uncolored bands on the right.
Those bands are bent for the convenience of horizontal gluing.
From now on we just draw vertical bands instead of bent ones for the sake of simple graphics as in Figure \ref{fig:non bend Rt}.
\begin{figure}[h]
\center
\includegraphics[width=3.4in]{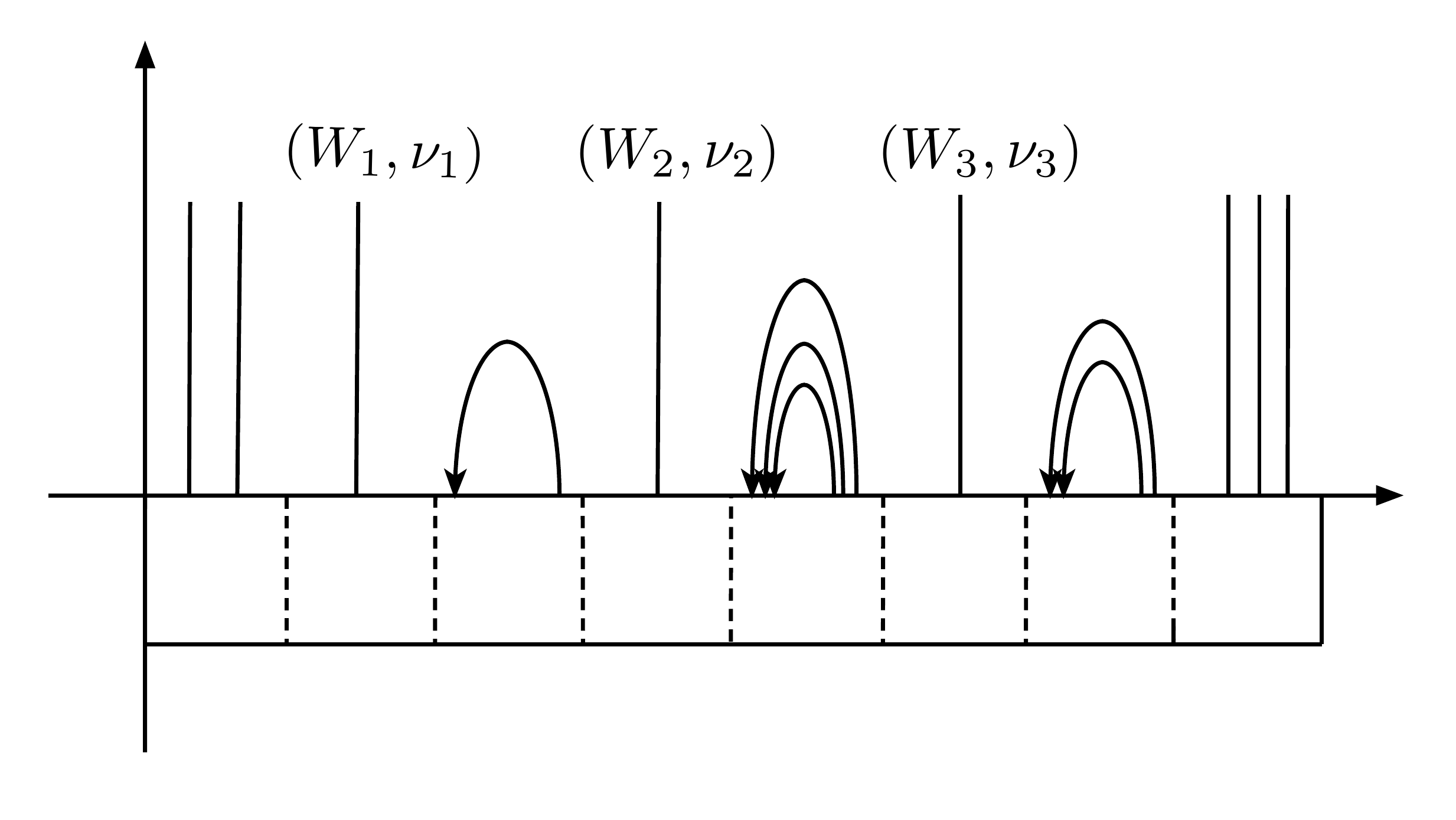}
\caption{The ribbon graph $R_t$}
\label{fig:non bend Rt}
\end{figure}
Recall that those left bands are ordered from the left and those right bands are ordered from the right.
We color the $k$-th left uncolored band with the simple object $V_{i_k}$ for $k=1, \dots, m$.
Also we color the $l$-th right uncolored band with the simple object $V_{j_l}$ for $l=1, \dots, n$.
Let us denote the ribbon graph obtained by this way $R_t(i, j)$.
The only uncolored ribbons in $R_t(i,j)$ are the cap-like bands. 
See Figure \ref{fig:ijcolored}.
\begin{figure}[h]
\center
\includegraphics[width=4in]{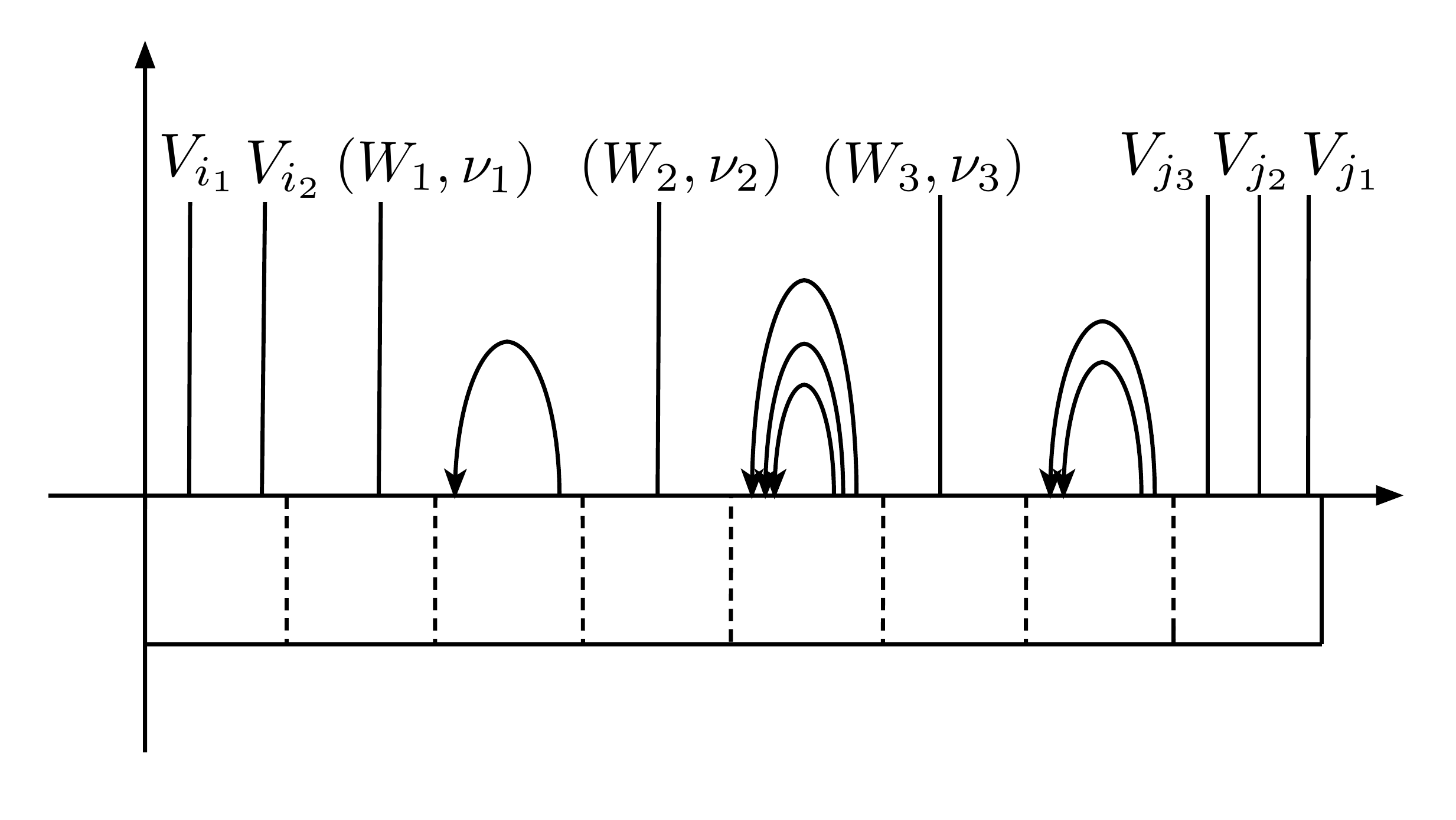}
\caption{The ribbon graph $R_t(i,j)$}
\label{fig:ijcolored}
\end{figure}
Fixing $i\in I^m$ and $j\in I^n$, we consider all the possible colors of the uncolored cap-like bands by simple objects.
The $(i, j)$-component module $\X(\Sigma_t)_{ij}$ of the 2-matrix $\X(\Sigma_t)$ will be the projective module of all the possible colors of the coupon $R_t(i ,j)$.
Recall the definition of the object $H_i^a$ of the modular category $\V$ defined in (\ref{equ:H_i^a}).
For a positive integer $a$ and $i=(i_1,\dots, i_a)\in I^a$, the object was defined to be
\begin{equation}\label{equ:Hai}
H^a_i=V_{i_1}\otimes V_{i_2}\otimes \cdots \otimes V_{i_a} \otimes V^*_{i_a}\otimes \cdots \otimes V_{i_2}^* \otimes V^*_{i_1}.
\end{equation}
If $a=(W, \nu)$ is a signed object of $\V$, we set $I^a$ to be a set of only one element and 
set $H^a_i=W^{\nu}$ for $i\in I^a$.
Note that the tensor product in (\ref{equ:Hai}) can be used as a color for rainbow like bands corresponding to an integer entry $a$ in a decorated type.
For $\zeta=(\zeta_1, \dots, \zeta_p) \in I^t$ with $\zeta_1=(\zeta_1^1, \dots, \zeta_1^{a_1}), \dots, \zeta_p=(\zeta_p^1, \dots, \zeta_p^{a_p})$, recall the notation given in (\ref{equ:Phi.t.zeta}):
\begin{equation*}
\Phi(t; \zeta)= H^{a_1}_{\zeta_1} \otimes H^{a_2}_{\zeta_2} \otimes \cdots \otimes H^{a_p}_{\zeta_p}.
\end{equation*}
The next one is a new notation.
We set
\begin{equation}\label{equ: Phi for e-type}
\Phi(t; \zeta; i, j)=V^*_{i_1} \otimes V^*_{i_2} \otimes \cdots \otimes V^*_{i_m} \otimes \Phi(t, \zeta) \otimes V_{j_n}\otimes V_{j_{n-1}}\otimes \cdots \otimes V_{j_1},
\end{equation}
Note that each choice of $\zeta=(\zeta_1, \dots, \zeta_p) \in I^t$ determines a color of ribbon graph $R_t(i, j)$ via $\Phi(t; \zeta; i, j)$ except for the coupon.
The coupon can be colored by a morphism from the monoidal unit $\1$ to the object $\Phi(t; \zeta; i, j)$.
Thus all the possible colors of the coupon of $R_t(i, j)$ varying $\zeta \in I^t$ is
\begin{equation}\label{equ:tsigma}
 \X(\Sigma_t)_{ij}=\bigoplus_{\zeta \in I^t} \Hom \big(\1, \Phi(t; \zeta;i,j) \big)
\end{equation}
and we define this to be the $(i, j)$-component projective module $\X(\Sigma_t)_{ij}$.

We also need to specify the assignment of $\X$ on each formal identity 1-morphism $\id_n: \nstand{n} \to \nstand{n}$ with an integer  $n$.
The $\bfk^n \times \bfk^n$ 2-matrix $\X(\Sigma_n)$ is defined to be the identity $\bfk^n \times \bfk^n$ 2-matrix.
Namely each diagonal entry of the 2-matrix $\X(\id_n)$ is the ground ring $K$ and each entry off the diagonal is zero.

\subsection{$\X$ on 2-morphisms}
Let $[M]: \Sigma_{t} \Rightarrow \Sigma_{s}: \nstand{m} \to \nstand{n}$ be a 2-morphism of $\Co$.
We need to define a $K$-homomorphism $\X([M])_{ij}$ from $\X(\Sigma_{t})_{ij}$ to $\X(\Sigma_{s})_{ij}$ for each $i\in I^m$ and $j\in I^n$.
This homomorphism will be obtained by applying the Reshetikhin-Turaev TQFT to the decorated cobordism obtained by ``capping'' or ``filling'' the left and the right boundaries of $M$.

\subsubsection{Filling $M$ by the standard handlebodies}
Let $(M, \phi)$ be a representative of the 2-morphism $[M]: \Sigma_{t} \Rightarrow \Sigma_{s}: \nstand{m} \to \nstand{n}$.
Here $\phi$ is a parametrization of the boundary $\partial M$.
Recall that using the parametrization, we can form the surface $\Sigma(\phi)$ by gluing standard surfaces.
Then the parametrization $\phi$ can be regarded as a homeomorphism from $\Sigma(\phi)$ to $\partial M$.
Consider the standard handlebodies $U_t$, $U_s^-$ and solid cylinders $D_m$, $D_n$ (see Figure \ref{fig:Dn}).
Their boundaries are capped standard surfaces.
Then the gluing map induced by the parametrization extends to the disks enclosed by boundary circles of the standard surfaces.
Gluing $U_t$, $U_s^-$, $D_m$, and $D_n$ along this homeomorphism, we obtain a 3-manifold whose boundary is $\Sigma(\phi)$.
We also can assume that the ribbon graphs glues well under this gluing.
Let $\M(t,s)$ be the manifold obtained by the procedure and we call it the \textit{standard handlebody} for the pair $(t,s)$.
By the Alexander trick, the manifold $\M(t, s)$ is defined up to homeomorphism.
The manifold $\M(t, s)$ is equiped with  a ribbon graph obtained from the ribbon graph $R_t$ and $R_s$ in $U_t$ and $U_s^-$ respectively, and the vertical bands in $D_m$, $D_n$ joining uncolored bands along the embedded disks.
We denote the ribbon graph by $R(t,s)$ and call this ribbon graph in $\M(t, s)$ the \textit{standard ribbon graph} for the pair $(t, s)$.
\begin{figure}[h]
\center
\includegraphics[width=3in]{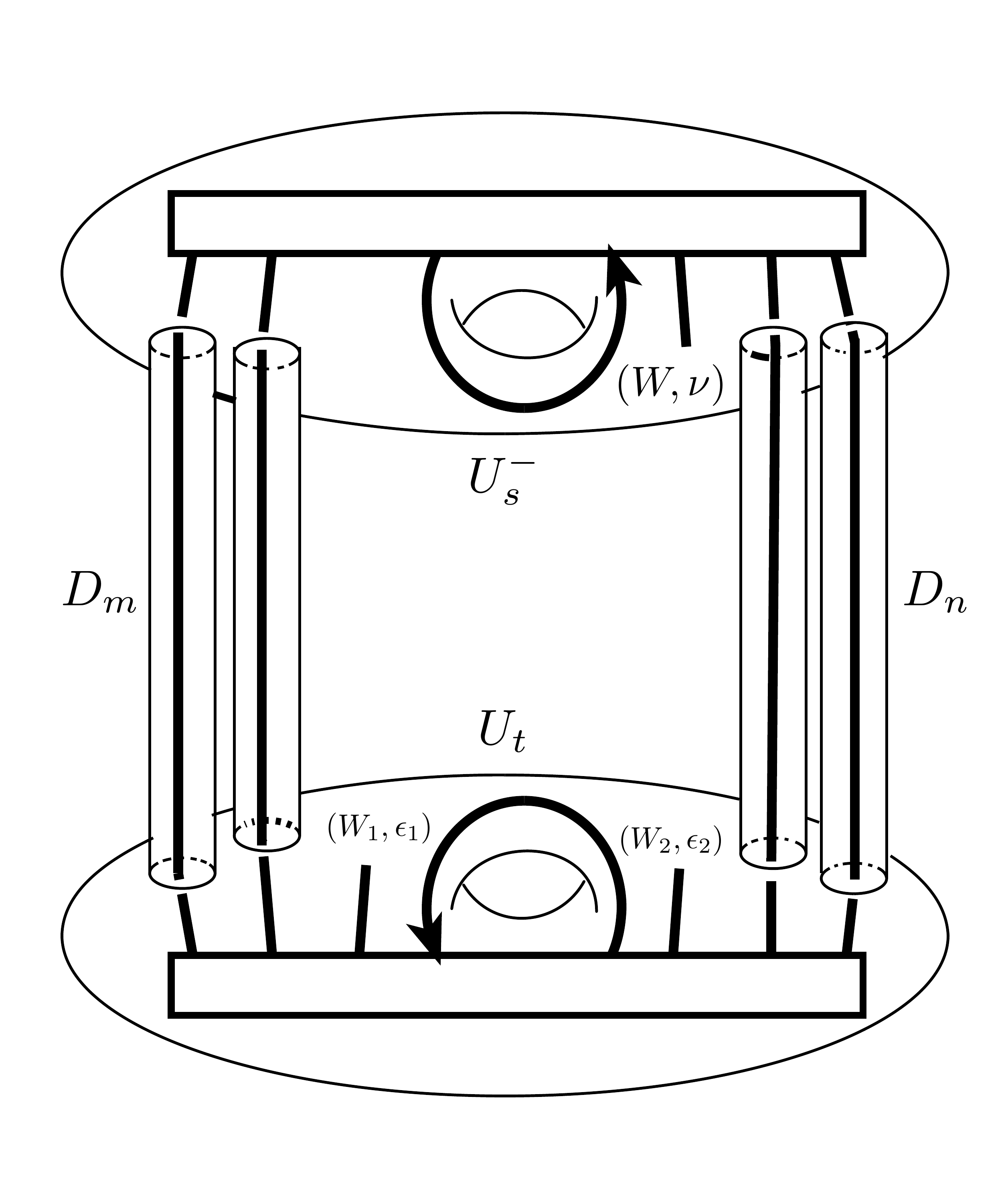}
\caption{The standard handlebody $\M(t, s)$ and the standard ribbon graph $R_t(i,j)$}
\label{fig:standard handlebody with the standard ribbon}
\end{figure}
Now we glue the manifolds $M$ and $\M(t, s)$ along the boundaries by the parametrization $\phi$ and obtain the closed 3-manifold
\begin{equation}\label{equ: tilde M}
\Fill(M):=M\cup_{\phi} \M(t, s).
\end{equation}
In a sense, we ``filled'' the boundary of $M$ by the standard handlebody $\M(t, s)$.
It equips with a ribbon graph coming from a ribbon graph in $M$ and the standard ribbon graph $R(t, s)$ in $\M(t, s)$

The same manifold can be obtained differently as follows.
First, we glue cylinders $D_m$ and $D_n$ to $M$ via $\phi$, which we denote by
\begin{equation}
\cFill(M):=M\cup_{\phi}(D_m\sqcup D_n)
\end{equation}
Namely, we filled only cylindrical parts of $M$.
The subscript $\mathrm{c}$ in $\cFill$ stands for ``cylinder''.
Then the boundary of $M\cup_{\phi}(D_m\sqcup D_n)$ is homeomorphic to the disjoint union of the capped standard boundaries $\hat{\Sigma}_t \sqcup \hat{\Sigma}_s^-$.
The gluing homeomorphism is given by the parametrization $\phi$ extended to embedded disks.
Thus $\cFill(M)$ is a usual parametrized cobordism except that the ribbons in $D_m\sqcup D_n$ are not colored.
If we glue the standard handlebodies $U_t$ and $U_s^-$ via the parametrization we obtain the same manifold $\Fill(M)$ as in (\ref{equ: tilde M}).

\subsubsection{Definition of a $K$-homomorphism $\X([M])_{ij}$}
Fix the $i$-th element $i=(i_1, i_2, \dots, i_m)$ of $I^m$   and the $j$-th element $j=(j_1, \dots, j_n)$ of $I^n$ as before.
We will construct a $K$-homomorphism $\X([M])_{ij}$ from $\X(\Sigma_{t})_{ij}$ to $\X(\Sigma_{s})_{ij}$.
Let us give colors to the uncolored ribbon graphs of $D_m \sqcup D_n$ in $\cFill(M)=M\cup_{\phi}(D_m\sqcup D_n)$ as follows.
Order the uncolored bands in $D_m$ from the left and order the bands in $D_n$ from the right according to the order of the circles in the boundary surface.
We color the $k$-th left uncolored band with the simple object $V_{i_k}$ for $k=1, \dots, m$.
Also we color the $l$-th right uncolored band with the simple object $V_{j_l}$ for $l=1, \dots, n$.
Then $\cFill(M)$ together with this $v$-colored ribbon graph in $\cFill(M)$, which we denote by $\cFill(M)_{ij}$, is the normal cobordism of the Reshetikhin-Turaev type.
Apply the Reshetikhin-Turaev TQFT, we obtain a $K$-homomorphism $\tau(\cFill(M)_{ij})$ from $\X(\Sigma_t)_{ij}$ to $\X(\Sigma_s)_{ij}$.
\begin{lemma}
If $(M, \phi)$ is equivalent to $(N, \psi)$, then $\tau(\cFill(M)_{ij})=\tau(\cFill(N)_{ij})$.
\end{lemma}
\begin{proof}
Since $M$ and $M$ are equivalent, we see that $\cFill(M)_{ij}$ is $d$-homeomorphic to $\cFill(N)_{ij}$.
Since the Reshetikhin-Turaev TQFT $\tau$ is invariant under $d$-homeomorphisms, we have the result.
(To see this invariance, note that $d$-homeomorphic cobordisms have the same special ribbon graph representation.)
\end{proof}

Thus the $K$-homomorphism $\tau(\cFill(M)_{ij})$ is independent of the choice of a representative of $[M]$.
Hence we can define the $(i,j)$-entry of the 2-matrix $\X(M)$ to be
\begin{equation}\label{equ:XM ij}
\X([M])_{ij}:=\tau(\cFill(M)_{ij})
\end{equation}

For each formal identity $\id_{\Sigma_t} \in \Co(\nstand{m}, \nstand{n})$, $\X$ assigns the $\bfk^m$ by $\bfk^n$  2-homomorphism matrix whose $(i,j)$-entry is the identity self-homomorphism of the module $\X(\Sigma_t)_{ij}$.
For the formal horizontal unit 2-morphism $\id_{\id_n}$, $\X$ assigns the $k^n$ by $k^n$ identity 2-homomorphism.
Namely, each of its diagonal entry is the identity self-homomorphism of the base ring $K$ and off-diagonal entries are zero.

\subsubsection{Representation by a ribbon graph}
The rest of the paper will be devoted to prove that the assignment $\X$ is indeed a projective pseudo 2-functor.
The key ingredient of the proof is the explicit formula to calculate the homomorphism $\X(M)_{ij}$ obtained by representing $M$ by a special ribbon graph as in Section \ref{subsec:explicit formula for tau(M)}.
We define a \textit{special ribbon graph} for $(M, i, j)$ to be a special ribbon graph for $\cFill(M)_{ij}$.
In place of $M$, $R_t$, and $-R_s$ in Section \ref{subsec:explicit formula for tau(M)}, we just need to use $\cFill(M)_{ij}$, $R_t(i,j)$, and $-R_s(i,j)$.
As noted above, gluing handlebodies to fill the boundary of $\cFill(M)_{ij}$ produces $\Fill(M)$ with uncolored vertical bands are colored according to $(i,j)$.
The $(i,j)$-colored $\Fill(M)$ is denoted by $\Fill(M)_{ij}$ and the $(i, j)$-colored  standard ribbon graph $R(t,s)$ is denoted by $R(t,s)_{ij}$.
By changing to an equivalent manifold if necessarily, we may assume that a special ribbon graph for $(M, i,j)$, which is denoted by $\Omega_{(M, i,j)}$, is a disjoint union of the standard ribbon graph $R(t,s)_{ij}$ and a surgery link $L=L_1\cup \cdots\cup L_{\mu}$ and a ribbon graph of $M$.
The ribbon graph obtained by replacing $R(t,s)_{ij}$ by $R(t,s)$ is denoted by $\Omega_M$.
This ribbon graph $\Omega_M$ is thus obtained by removing the colors of left and right vertical bands.
Note that if $M$ has no corners then $\Omega_M$ is the same as the definition of special ribbon graph given in Section \ref{subsec:explicit formula for tau(M)}.
In summery we have the following explicit formula for the $(\zeta, \eta)$-block matrix 
\begin{align}\label{equ:tau zeta eta extended}
(\X(M)_{ij})_{\zeta}^{\eta}&=\tau(\cFill(M)_{ij})_{\zeta}^{\eta} \notag \\ 
&= \Delta^{\sigma(L)} \D^{-g^+ -\sigma(L) - \mu} \dim(\eta) \sum_{\lambda \in \col(L)} \dim(\lambda) \F_0(\Omega_{(M, i, j)}, \zeta, \eta, \lambda),
\end{align}
where $g^+$ is the sum of the integer entries of $s$ and 
\[\dim(\eta):=\prod_{i=1}^q\dim(\eta_i) \]
with
\[\dim(\eta_i):=
\begin{cases}
\prod_{l=1}^{b_i} \dim(\eta_i^l) & \mbox{ if } b_i \in \Z \\
1 & \mbox{ if } b_i \mbox{ is a mark.}
\end{cases}
\]

\section{Main Theorem}\label{sec:Main Theorem}
So far we defined the 2-category of decorated cobordisms with corners $\Co$, where cobordisms are decorated by a modular category $\V$.
We constructed the assignment $\X$ from $\Co$ to the Kapranov-Voevodsky 2-vector spaces $\KV$.
The explicit formula was obtained by expressing a cobordism with corners $M$ by the special ribbon graph $\Omega_M$ in $S^3$.
Now we prove the following main theorem of the current paper.
\begin{thm}\label{thm:main theorem}
The assignment $\X$ defined above is a projective pseudo 2-functor from $\Co$ to $\KV$.
\end{thm}
For our convention of the language of 2-category, see Appendix.
This theorem follows from several propositions below.
The idea of the proof is that we reduce the gluings of cobordisms to the gluing of special ribbon graphs and work with the explicit formula.

\subsection{Vertical projective Functor}

The vertical composition is not preserved by $\X$.
This is because of an anomaly in the Reshetikhin-Turaev TQFT.
Thus we instead claim that $\X$ is a projective functor on the hom-category $\Co(\nstand{m}, \nstand{n})$.
For a projective functor to exist, the target category should be an $K$-module for some commutative ring $K$.
In our case, the target category is the hom-category $\KV(\bfk^m, \bfk^n)$.
This hom-category  is a $K=\Hom(\1,\1)$-module by multiplying an element $k\in K$ component-wise:
\[ k\cdot (f_{ij})_{ij}:=(kf_{ij})_{ij},\]
where $(f_{ij})_{ij}$ is a 2-morphism in $\KV$.

First we state results regarding the anomaly of the original RT TQFT.
Let $M_1$ and $M_2$ be a composable decorated cobordisms (without corners).
(As always in this paper, we assume the source and the target boundary surfaces are both connected.)
Let $L_1$, $L_2$ and $L$ be surgery links for special ribbon graphs $\Omega_{M_1}$, $\Omega_{M_2}$ and $\Omega_{M_1\cdot M_2}$ of $M_1$ and $M_2$ and $M_1\cdot M_2$, respectively.

\begin{lemma}
Using notations above, the vertical concatenation $\Omega_{M_1}\cdot \Omega_{M_2}$ of the special ribbon graphs $\Omega_{M_1}$ and $\Omega_{M_2}$ is a special ribbon graph for $M_1 \cdot M_2$.
\end{lemma}

\begin{lemma}\label{lem:vertical anomaly of original RT}
With the same notations as above, we have
\[\tau(M_1\cdot M_2)= k(M_1, M_2)\tau(M_1)\cdot \tau(M_2),\]
where
\[k(M_1, M_2)=(\D \Delta)^{\sigma(L_1)+\sigma(L_2)-\sigma(L)}. \]
\end{lemma}

The proofs of both lemmas can be found in \cite[Lemma I\hspace{-.1em}V 2.1.2]{Turaev10}.

\begin{prop}[Vertical composition]\label{lem:2vertical composition}
Let $[(M_1, \phi_1)]:\Sigma_{t_1}\Rightarrow \Sigma_{t_2}: \nstand{m} \to \nstand{n}$ and $[(M_2, \phi_2)]:\Sigma_{t_2}\Rightarrow \Sigma_{t_3}: \nstand{m} \to \nstand{n}$ be (non-formal) 2-morphisms of $\Co$ so that the target 1-morphism of $[M_1]$ is equal to the source 1-morphism of $[M_2]$.
Then we have 
\[\X(M_1\cdot M_2)=k(M_1, M_2) \X(M_1)\cdot \X(M_2),\]
 where $k(M_1, M_2) \in K$ is a gluing anomaly of the pair $(M_1, M_2)$ given as follows.
Let $L_1$, $L_2$ and $L$ be surgery links of $\Omega_{M_1}$, $\Omega_{M_2}$ and $\Omega_{M_1} \cdot \Omega_{M_2}$, respectively.
Then 
\[k(M_1, M_2)=(\D \Delta)^{\sigma(L_1)+\sigma(L_2)-\sigma(L)}.\]
\end{prop}
\begin{proof}
It suffices to show that the equality
\begin{equation*}\label{equ:XMM}
 \X(M_1 \cdot M_2)_{ij} =k(M_1, M_2) \left( \X(M_1) \cdot \X(M_2) \right)_{ij}
\end{equation*}
holds for $i\in I^m$ and $j\in I^n$.
Note that the surgery links $L_1$, $L_2$ and $L$ are independent of the indices $i$ and $j$, so is $k(M_1, M_2$), hence the result follows.
By the definition of $\X(M)_{ij}$ (see (\ref{equ:XM ij})), the above equality is equivalent to show the following equality:
\begin{equation}\label{equ:tau fill MM}
 \tau\left(\cFill(M_1 \cdot M_2\right)_{ij})=k(M_1, M_2) \tau\left(\cFill(M_1)_{ij}\right) \circ \tau\left(\cFill(M_2)_{ij}\right).
\end{equation}
Since filling corners and vertical gluing commute, we have
\[\cFill(M_1 \cdot M_2)_{ij} = \cFill(M_1)_{ij} \cdot \cFill(M_2)_{ij}.\]
By definition, the special ribbon graphs $(\Omega_{M_1})_{ij}$ and $(\Omega_{M_2})_{ij}$ represent the cobordisms $\cFill(M_1)_{ij}$ and $\cFill(M_2)_{ij}$, respectively.
Note that the surgery links of $(\Omega_{M_1})_{ij}$ and $(\Omega_{M_2})_{ij}$ are the same as the surgery links of $\Omega_{M_1}$ and $\Omega_{M_2}$  since only difference between them is the colors of the left and the right vertical bands, which are not surgery links.
Thus, the equality (\ref{equ:tau fill MM}) follows from Lemma \ref{lem:vertical anomaly of original RT}.
\end{proof}
If one of $M_1$ and $M_2$ is the vertical identity, then we set $k(M_1\cdot M_2)=1 \in K$.

\begin{lemma}\label{lem:vertical anomaly associativity}
Suppose that $M_1$, $M_2$, and $M_3$ are three vertically composable 2-morphisms of $\Co$.
Namely, we can form the 2-morphism $M_1 \cdot M_2 \cdot M_3$.
Then we have
\begin{equation}\label{equ:vertical anomaly}
k(M_1, M_2 \cdot M_3)k(M_2\cdot M_3)=k(M_1 \cdot M_2, M_3)k(M_1, M_2).
\end{equation}
\begin{proof}
 For $i=1,2,3$ present the cobordisms $M_i$ by a special ribbon graph $\Omega_{M_i}$ and let $L_i$ be the surgery link in $\Omega_i$.
The cobordism $M_1\cdot M_2$ is represented by the special ribbon graph $\Omega_{M_1} \cdot \Omega_{M_2}$.
Let $L_{12}$ be a part of the surgery link of $\Omega_{M_1} \cdot \Omega_{M_2}$ that is not in $L_1 \cup L_2$.
Namely, the surgery link $L_{12}$ is a newly emerged ribbons when we concatenate the ribbons $\Omega_{M_1}$ and $\Omega_{M_2}$.
Similarly, let $L_{23}$ be the surgery link that is not in $L_2 \cup L_3$.
Let $L_{123}$ be the surgery link of $\Omega_{M_1}\cdot \Omega_{M_2} \cdot \Omega_{M_3}$, namely $L_{123}$ is the union of all of the above surgery links.

By Lemma \ref{lem:2vertical composition}, anomalies can be computed by signatures of surgery links.
Thus, the equality \ref{equ:vertical anomaly} is equivalent to the equality
\begin{align*}
[\sigma(L)+\sigma(L_3)-\sigma(L_{123})] +[ \sigma(L_1)+\sigma(L_2)-\sigma(L)]\\
=[\sigma(L_1)+\sigma(L')-\sigma(L_{123})]+[\sigma(L_2) +\sigma(L_3)-\sigma(L')].
\end{align*}
Since both sides are equal to $\sigma(L_1)+\sigma(L_2)+\sigma(L_3)-\sigma(L_{123})$, the equality holds.
\end{proof}

\end{lemma}

The results of Proposition \ref{lem:2vertical composition} and Lemma \ref{lem:vertical anomaly associativity} can be summarized into:
\begin{prop}\label{prop:vertical projective functor}
The assignment $\X$ is a projective functor from the hom-category $\Co(\nstand{m}, \nstand{n})$ to the hom-category $\KV(\bfk^m, \bfk^n)$.
\end{prop}

\subsection{Horizontal Axioms}
Now we are going to study how the assignment $\X$ behaves on horizontal gluings.
For each type $t=(m,n; a_1, \cdots, a_p)$, recall the following notations from Section \ref{sec:X on 1-morphisms}:
 \begin{equation*}
\Phi(t; \zeta)= H^{a_1}_{\zeta_1} \otimes H^{a_2}_{\zeta_2} \otimes \cdots \otimes H^{a_p}_{\zeta_p}
\end{equation*}
and
\begin{equation*}
\Phi(t; \zeta; i, j)=V^*_{i_1} \otimes V^*_{i_2} \otimes \cdots \otimes V^*_{i_m} \otimes \Phi(t, \zeta) \otimes V_{j_n}\otimes V_{j_{n-1}}\otimes \cdots \otimes V_{j_1}.
\end{equation*}
Also recall that we defined the module
\begin{equation*}
 \X(\Sigma_t)_{ij}=\bigoplus_{\zeta \in I^t} \Hom \big(\1, \Phi(t; \zeta;i,j) \big)
\end{equation*}

\begin{prop}\label{prop:2-functor on 1-morphisms}
Let $\Sigma_{t_1}: l \to m$ and $\Sigma_{t_2}:m \to n$ be composable 1-morphisms of $\Co$. 
Then  the 2-matrix $\X(\Sigma_{t_1}\circ \Sigma_{t_2})$ is canonically isomorphic to the 2-matrix $\X(\Sigma_{t_1})\X(\Sigma_{t_2})$.
\end{prop}
\begin{proof}
The $(h, j)$-component of the product of the 2-matrices $\X(\Sigma_{t_1})$ and $\X(\Sigma_{t_2})$ is the module
\begin{align}\label{equ:product of 2matrices}
\notag  \left( \X(\Sigma_{t_1}) \circ \X(\Sigma_{t_2}) \right)_{hj}=\bigoplus_{1 \leq i \leq \bfk^m} \X(\Sigma_{t_1})_{h, i}\otimes \X(\Sigma_{t_2})_{i, j}\\ 
=  \bigoplus_{1 \leq i \leq \bfk^m} \left[  \bigoplus_{\zeta \in I^{t_1}} \Hom \big(\1, \Phi(t_1; \zeta;h,i) \big) \otimes \bigoplus_{\eta \in I^{t_2}} \Hom \big(\1, \Phi(t_2; \eta;i,j) \big)  \right]
\end{align}
Using Lemma \ref{lem:sum over simple 2} we sum over $i_1$ and the module (\ref{equ:product of 2matrices}) is isomorphic to 

\begin{equation}\label{equ:product of 2matrices second}
\bigoplus_{i=(i_2, \dots, i_{m-1})\in I^{m-1}} \bigoplus_{\zeta \in I^{t_1}, \eta \in I^{t_2}}\Hom\left(\1, U(i, \zeta, \eta) \right),
\end{equation}
where $U(i, \zeta, \eta)$ is the following module.
\begin{multline*}
V^*_{h_1}\otimes \cdots \otimes V^*_{h_l} \otimes \Phi(t_1, \zeta) \\
\otimes V_{i_m}\otimes V_{i_{m-1}}\otimes \cdots \otimes V_{i_{2}} 
\otimes V_{i_{2}}^*\otimes V_{i_{3}}^*
\otimes \cdots \otimes V_{i_m}^* \\
\otimes \Phi(t_2, \eta) \otimes V_{j_n}\otimes \cdots \otimes V_{1}
\end{multline*}

Note that we have the equality
\[\bigoplus_{i=(i_2, \dots, i_{m})\in I^{m-1}} \bigoplus_{\zeta \in I^{t_1}, \eta \in I^{t_2}}=\bigoplus_{\xi \in I^{t_1 \circ t_2} }\]
and for $\xi=(\zeta, i, \eta)\in I^{t_1\circ t_2}$ with $\zeta \in I^{t_1}, i\in I^{m-1}, \eta \in I^{t_2}$, the object $\Phi(t_1 \circ t_2, \xi)$ is equal to
\[\Phi(t_1, \zeta)\otimes V_{i_m}\otimes V_{i_{m-1}}\otimes \cdots \otimes V_{i_{2}} 
\otimes V_{i_{2}}^*\otimes V_{i_{3}}^*
\otimes \cdots \otimes V_{i_m}^* 
\otimes \Phi(t_2, \eta)\]
Thus the module (\ref{equ:product of 2matrices second}) is equal to the module
\begin{align*}
&\bigoplus_{\xi \in I^{t_1 \circ t_2} }\Hom(\1, V^*_{h_1}\otimes \cdots \otimes V^*_{h_l}\otimes \Phi(t_1 \circ t_2, \xi) \otimes V_{j_n}\otimes \cdots \otimes V_{1}) \\
&=\bigoplus_{\xi \in I^{t_1 \circ t_2} }\Hom(\1, \Phi(t_1\circ t_2; \xi; h,j)) =\X(\Sigma_{t_1\circ t_2})_{hj}
\end{align*}

Note that the isomorphism from $\X(\Sigma_{t_1}) \circ \X(\Sigma_{t_2})$ to $\X(\Sigma_{t_1\circ t_2})$ is given by the isomorphism $u$ of Lemma \ref{lem:sum over simple 2}.
This fact will be used in the proof of Lemma \ref{prop:2horizontal} below.
\end{proof}

We saw that vertical composition of cobordisms corresponds to concatenation of their special ribbon graphs.
This correspondence was the key observation to prove the projective functoriality of $\X$.
Similarly, to investigate how  horizontal composition behave under the map $\X$, we first need to study how  horizontal composition of cobordisms can be expressed as an operation on the special ribbon graph side.
The obvious guess is to juxtapose two special ribbon graphs.
But juxtaposing does not correspond to horizontal composition of cobordisms.
This can be seen, for instance, by noting that the type of bottom surface is not the desired one.

Let $[M]: \Sigma_{t_1} \Rightarrow \Sigma_{t_2}: \nstand{l} \to \nstand{m}$ and $[M']:\Sigma_{s_1} \Rightarrow \Sigma_{s_2}: \nstand{m}\to \nstand{n}$ be 2-morphisms which can be glued horizontally.
Let $\Omega_M$ and $\Omega_{M'}$ be special ribbon graphs representing the cobordisms $M$ and $M'$, respectively.
Recall that the special ribbon graph $\Omega_M$ consists of ribbons from $R_{t_1}$ and $-R_{t_2}$ with uncolored vertical bands connected, and a surgery link.
The surgery link may be tangled with $R_{t_1}$ and $-R_{t_2}$ as in Figure \ref{fig:OmegaM}.

\begin{figure}[h]
\center
\includegraphics[width=3.8in]{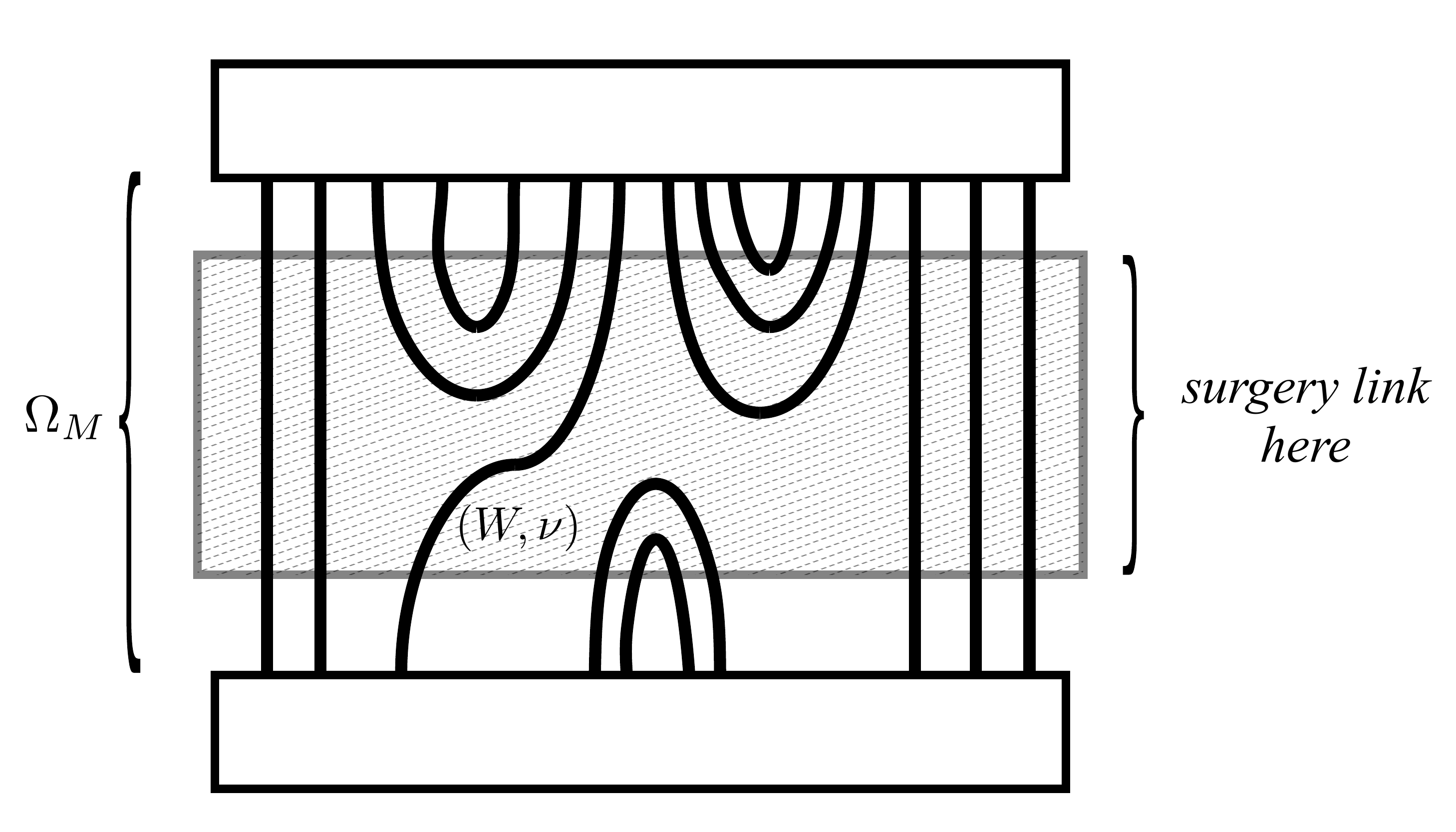}
\caption{The special ribbon graph $\Omega_M$}
\label{fig:OmegaM}
\end{figure}


We may assume that the surgery link in $\Omega_M$ is away from the rightmost vertical band of $\Omega_M$ by pulling a component of the surgery link over the top coupon and bring it to the other side.
Similarly, we may assume that no component of surgery link in $\Omega_{M'}$ is tangled with the leftmost uncolored vertical band of $\Omega_{M'}$ as in Figure \ref{fig:nosurgerylink}.

\begin{figure}[h]
\center
\includegraphics[width=4.8in]{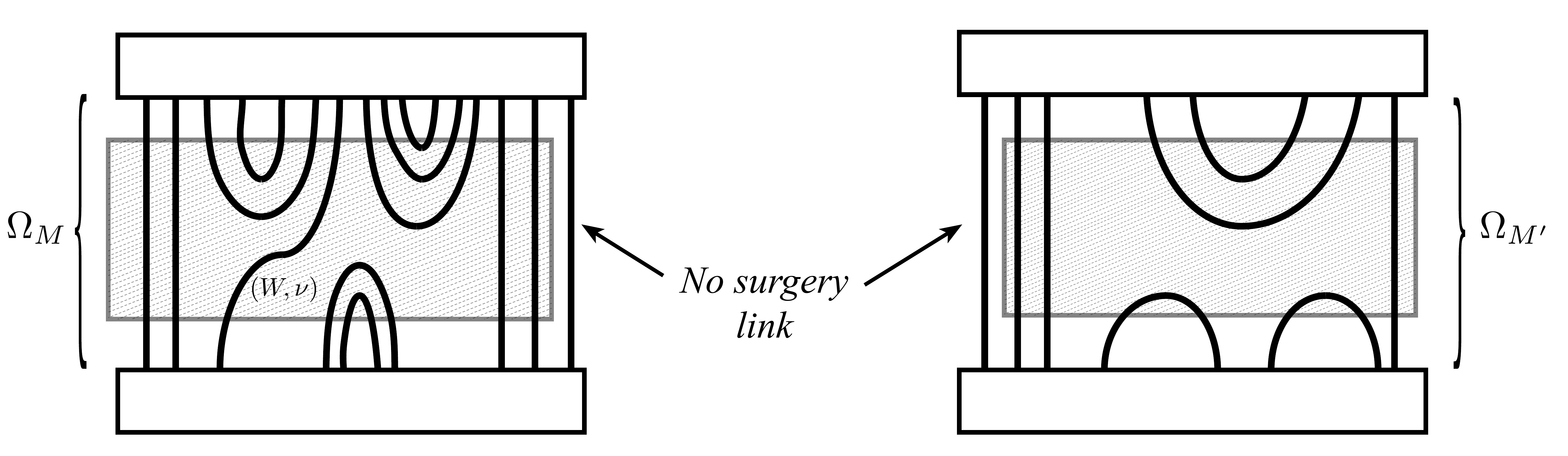}
\caption{No surgery link tangled at the rightmost and the leftmost}
\label{fig:nosurgerylink}
\end{figure}

We construct a new ribbon graph from these ribbon graphs $\Omega_{M}$ and $\Omega_{M'}$ as follows.
From $\Omega_{M}$, remove the rightmost uncolored vertical bands and denote the resulting ribbon graph by $\Omega_{M}^{-}$.
Similarly, remove the leftmost uncolored vertical bands from $\Omega_{M'}$ and denote the resulting ribbon graph by ${^{-}\Omega_{M'}}$.
We juxtapose $\Omega_{M}^{-}$ and $^{-}{\Omega_{M'}}$ so that $\Omega_{M}^{-}$ is on the left of ${^{-}\Omega_{M'}}$, namely $\Omega_{M}^{-} \otimes {^{-}\Omega_{M'}}$ in the category $\Rib$.
In the middle of the ribbon graph $\Omega_{M}^{-} \otimes {^{-}\Omega_{M'}}$, there are $2(m-1)$ uncolored vertical bands coming from the right uncolored vertical bands of $\Omega_{M}^{-}$ and the left uncolored vertical bands of ${^{-}\Omega_{M'}}$.
For each natural number $n$, let $\omega_{n}$ be a ribbon graph in $\R^2\times [0,1]\subset R^3$ defined in Figure \ref{fig:omega}.
The number of annulus ribbons in $\omega_{n}$ is $n$.
\begin{figure}[h]
\center
\includegraphics[width=3in]{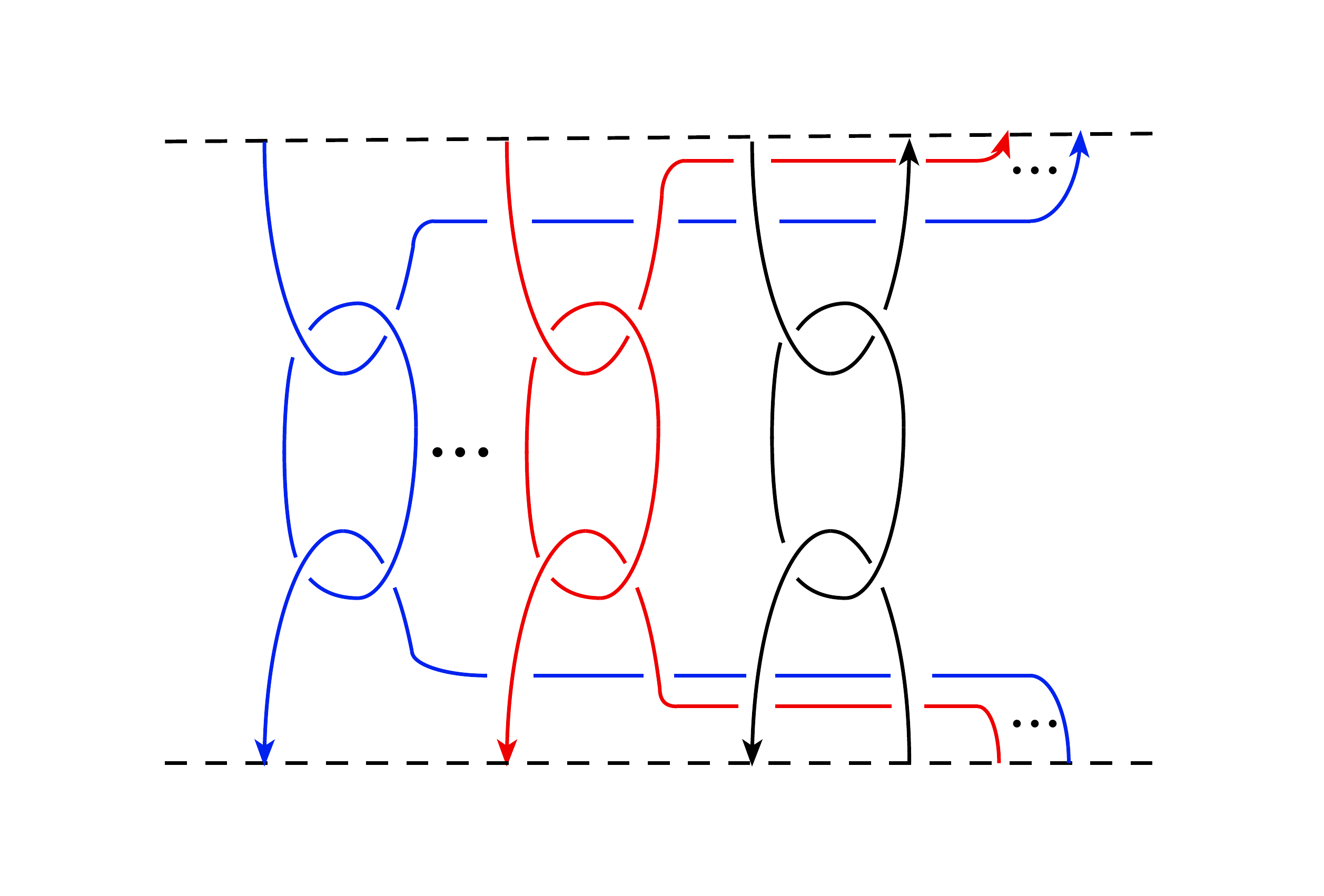}
\caption{Ribbon graph $\omega_{n}$}
\label{fig:omega}
\end{figure}
On the bottom of these $2(m-1)$ bands, we attach the ribbon graph $\omega_{m-1}$ defined in Figure \ref{fig:omega}.
Let $\Omega_{M, M'}$ denote the resulting ribbon graphs fitted in $\R^2 \times [0,1]$.
See Figure \ref{fig:Horizontal Special Ribbon } for an example.
\begin{figure}[h]
\center
\includegraphics[width=4in]{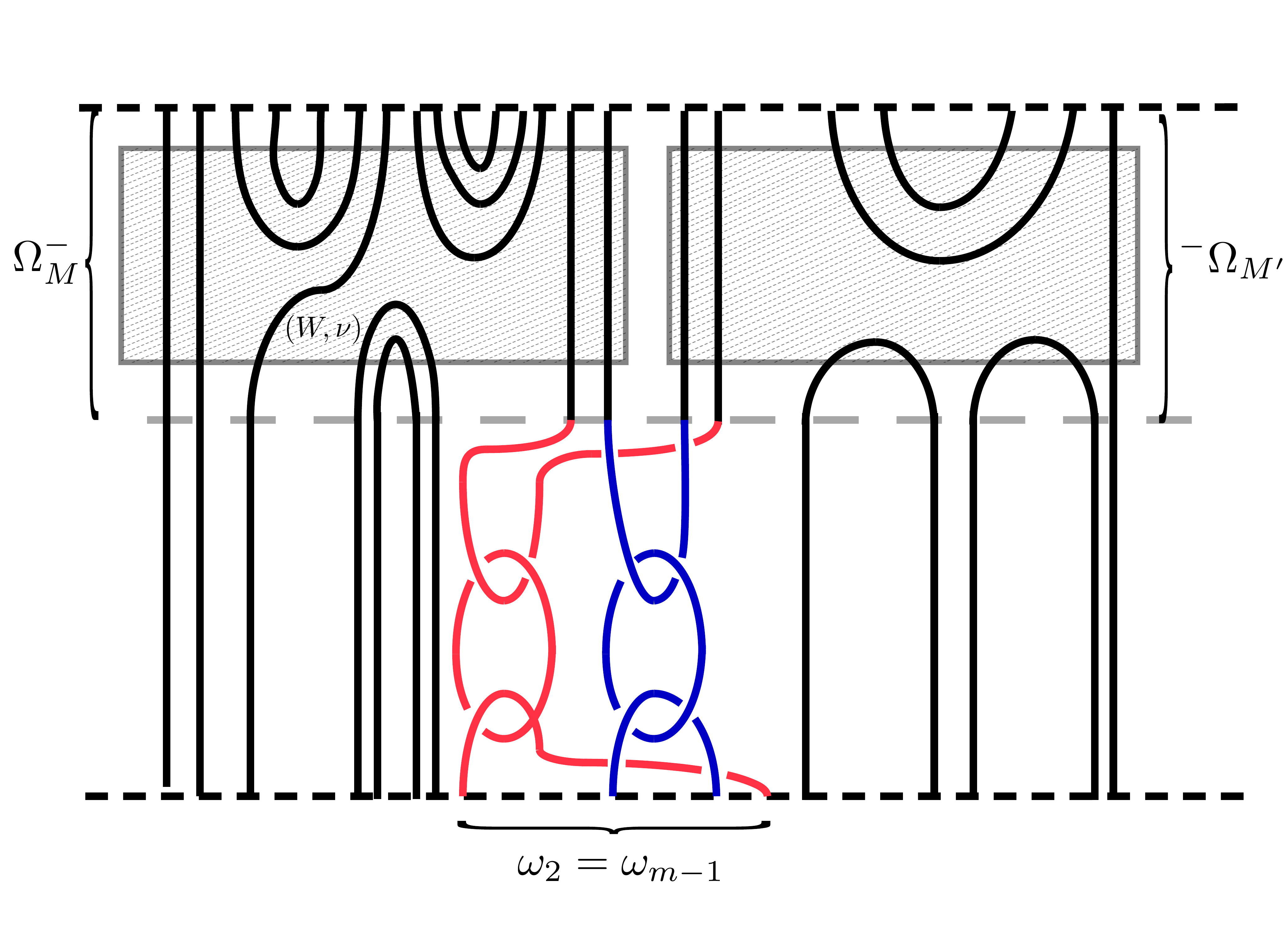}
\caption{Special ribbon graph $\Omega_{M, M'}$ for horizontal gluing}
\label{fig:Horizontal Special Ribbon }
\end{figure}

\begin{lemma}\label{lem:horizontal glue of ribbons}
The ribbon graph $\Omega_{M, M'}$  constructed above represents the horizontally glued cobordism $M\circ M'$.
\end{lemma}
\begin{proof}
First note that since surgery links of $\Omega_M$ and  $\Omega_{M'}$ are away from the neighborhoods of uncolored vertical bands of $\Omega_M$ and $\Omega_{M'}$, the order of the gluing and surgery is interchangeable. 
From $S^3$ with the ribbon graph $\Omega_M$ in it, let us cut out a regular neighborhood  of the bottom coupon and top coupon and rainbow bands attached to them.
Assuming the neighborhood of the top coupon contains the infinity in $S^3=\R^3 \cup \{ \infty \}$, we may assume that the rest of the ribbon graph lies in $S^2 \times [0,1] \subset \R^3$.
Similarly for $\Omega_{M'}$.
The horizontal gluing of $M$ and $M'$ now corresponds to cutting out the regular neighborhoods of the right vertical bands of $\Omega_M\subset S^2 \times [0,1]$ and the left vertical bands of $\Omega_{M'} \subset S^2 \times [0,1]$ and identify their boundaries and do surgery.
We decompose this procedure in several steps.
Instead of cutting out those neighborhoods at the same time, we first cut out only the rightmost vertical band of $\Omega_M$ and the rightmost vertical band of $\Omega_{M'}$.
Then we identify the boundary.
This gluing can be realized in $R^3$ as in Figure \ref{fig:first gluing}.

\begin{figure}[h]
\center
\includegraphics[width=4.6in]{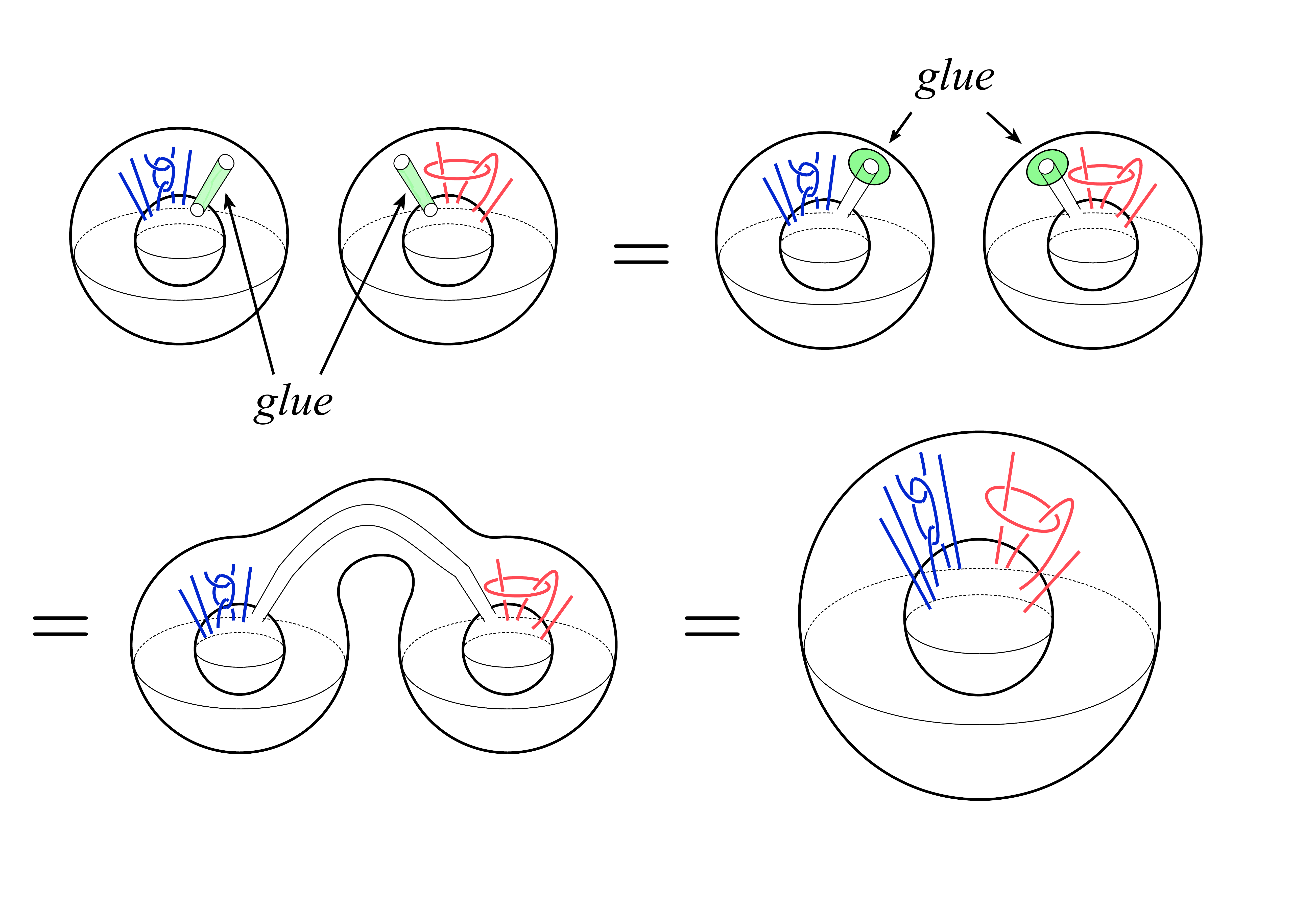}
\caption{Gluing the first corners}
\label{fig:first gluing}
\end{figure}

Thus the horizontally glued cobordism $M \circ M'$ can be obtained from the ribbon graph $\Omega_{M}^{-} \otimes {^{-}\Omega_{M'}}$ sitting in $S^2\times [0,1]$ by removing the neighborhoods of the middle $2(m-1)$ uncolored vertical bands and identify their boundaries.

Now we start from the ribbon graph $\Omega_{M, M'}$.
Attach coupons on the top and the bottom of the graph $\Omega_{M, M'}$.
Cut out the regular neighborhoods $T$ of the top and the bottom coupons and rainbow bands.
The rest of the ribbon lies in $S^2 \times [0,1]\subset R^3$.
We do surgery along the surgery link of $\omega_{m-1}$.
Let us describe this surgery carefully.
We follow the argument given in \cite[Lemma I\hspace{-.1em}V 2.6]{Turaev10}.
The ribbon graph $\omega_{m-1}$ has $m-1$ annuli along which we do surgery.
Let $A_r$ be the $r$-th annulus of $\omega_{m-1}$ for $r=1,\dots, m-1$.
We present this annulus in the form $A_r=D_r \setminus \Int(D_r')$, where $D_r$ and $D_r'$ are concentric 2-disks in $\R^2 \times [0,1]$ such that $D_r' \subset \Int(D_r)$ and $D_r'$ transversally intersects $\omega_{m-1}$ along two short intervals lying on two bands of $\omega_{m-1}$ linked by the annulus $A_r$, see Figure \ref{fig:annulus in omega}.

\begin{figure}[htbp]
 \begin{minipage}{0.4\hsize}
  \begin{center}
   \includegraphics[width=2in]{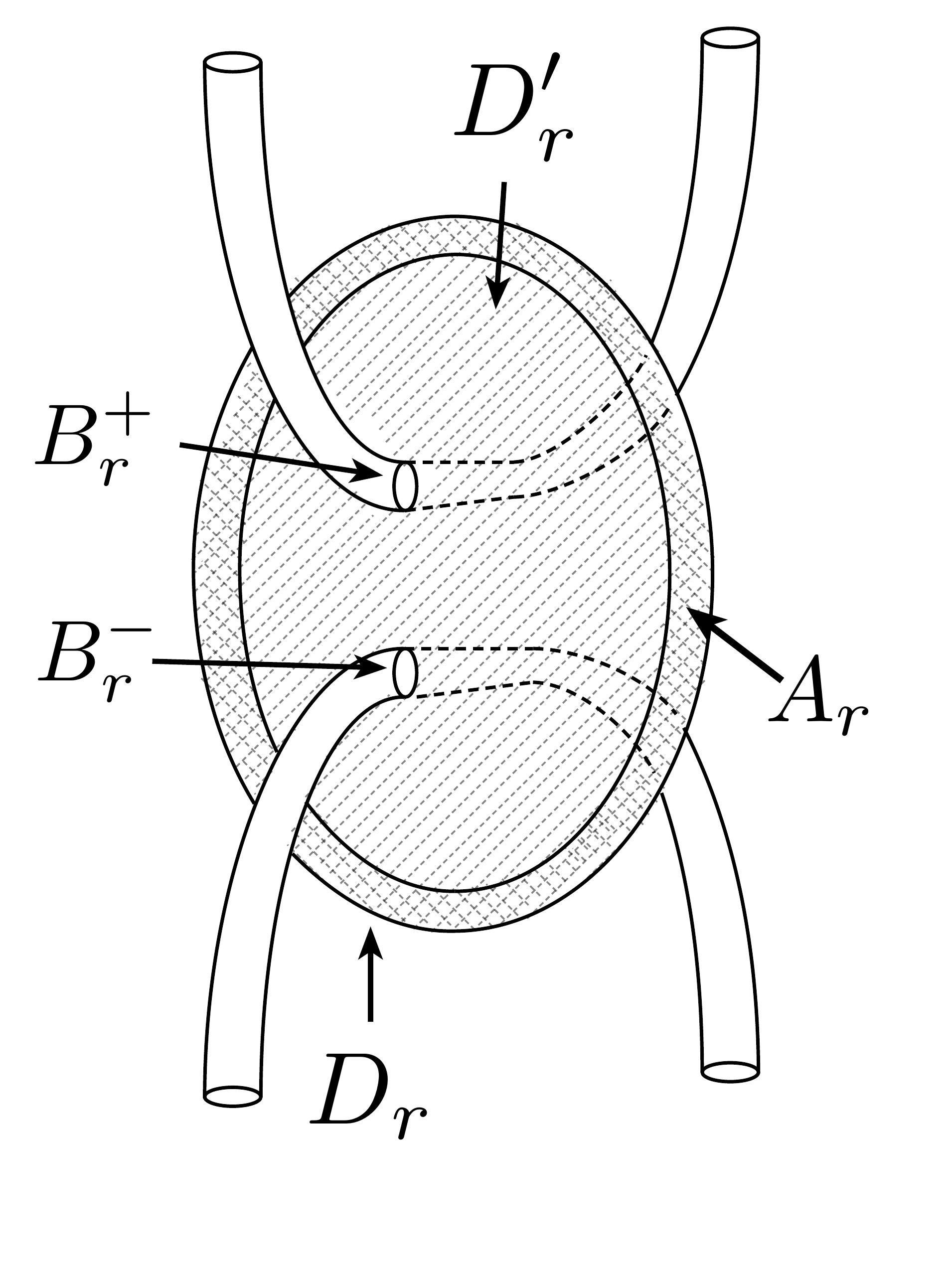}
  \end{center}
  \caption{The annulus $A_r$ in $\omega_{m-1}$}
  \label{fig:annulus in omega}
 \end{minipage}
 \begin{minipage}{0.5\hsize}
  \begin{center}
   \includegraphics[width=3in]{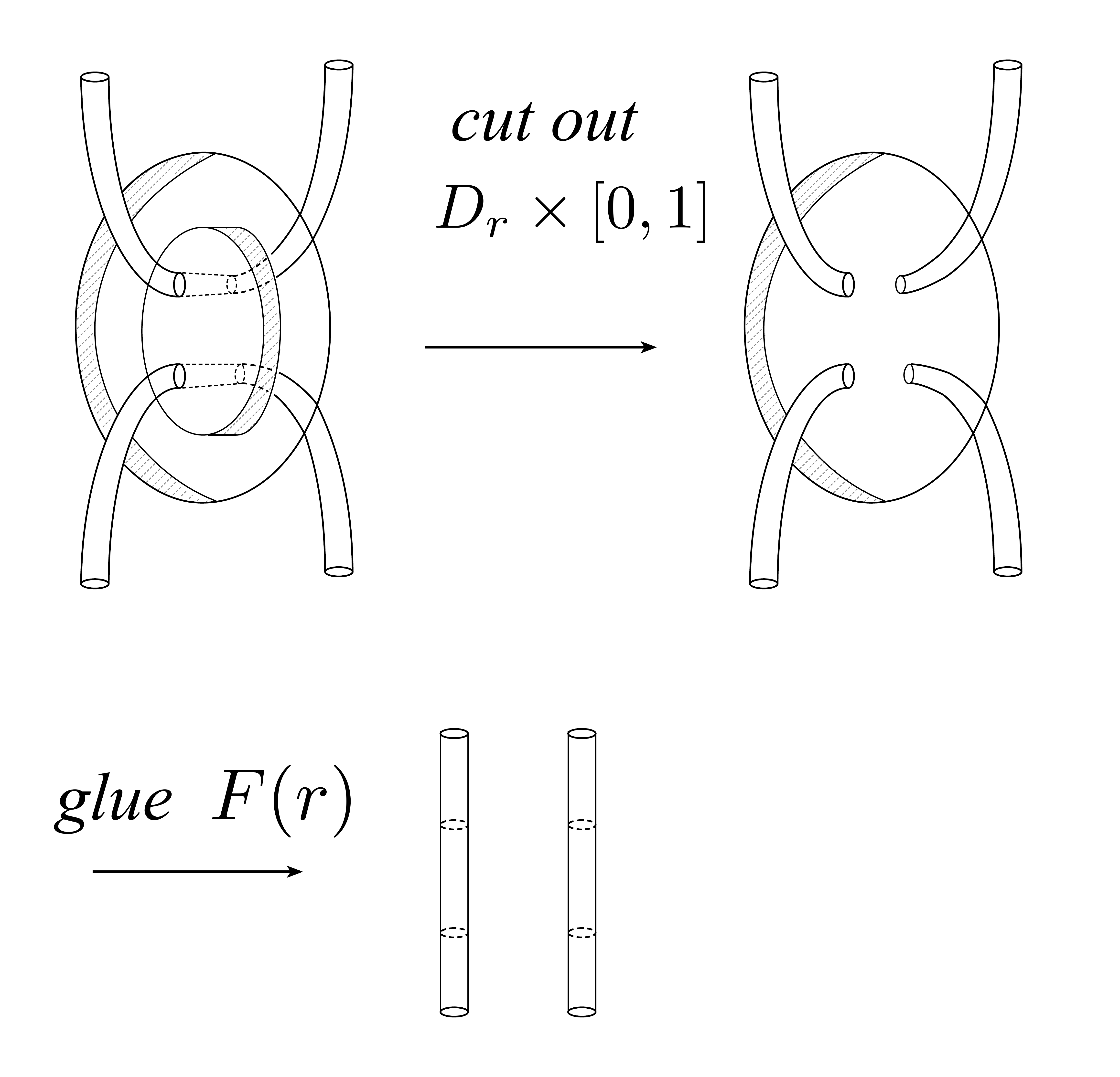}
  \end{center}
  \caption{Gluing $F(r)$}
  \label{fig:gluing F(r)}
 \end{minipage}
\end{figure}

Consider a regular neighborhood $D_r \times [-1, 1]$ in $\R^2 \times (0,1)$ of the larger disk $D_r$.
We think that $D_r$ lies in $D_r\times \{0\}$.
We assume that there are no redundant crossings.
Namely, locally the picture is as in Figure \ref{fig:annulus in omega}.
Let $B_r^{-}$ and $B_r^+$ be small closed disjoint 2-disks in $\Int(D_r')$ and we assume that the intersection of $D_r \times [-1, 1]$ and $T$ is subcylinder $B_r^-\times [-1, 1]$ and $B_r^+ \times [-1,1]$.

The surgery along the framed knot defined by $A_r$ may be described as follows.
Consider the solid torus
\[A_r \times [-1,1]= (D_r \times [-1,1]) \setminus (\Int(D_r') \times [-1,1]) \subset S^3.\]
Its boundary consists of four annuli $A_r \times \{-1\}$, $A_r \times \{1\}$, $\partial D_r \times [-1,1]$, $\partial D_r' \times [-1,1]$.
We remove the interior of $A_r \times [-1,1]$ from $S^3 \setminus \Int(T)$ and glue in its place the standard solid torus $D^2\times S^1$.
The gluing is performed along a homeomorphism $\partial(A_r \times [-1, 1]) \to \partial (D^2 \times S^1)$ carrying each circle $\partial D_r' \times \{t\}$ with $t\in [-1,1]$ onto a circle $\partial D^2 \times \{x\}$ with $x\in S^1$.
Let $E(r)$ denote the solid 3-cylinder formed by the disks $D^2 \times \{x\}$ glued to $\partial D_r' \times \{t\}$ with $t\in [-1,1]$.
Let $F(r)$ denote the complementary solid 3-cylinder $\overline{(D^2 \times S^1) \setminus E(r)}$.

For $r=1, \dots, m-1$ consider the genus 2 handlebody
\[ (D_r' \times [-1,1]) \setminus \Int(T) =(D_r' \setminus (\Int(B_r^- \cup B_r^+))) \times [-1, 1] \]
and glue $E(r)$ to it as specified above.
This gives a 3-cobordism with bases $\partial B_r^- \times [-1, 1]$ and $\partial B_r^+ \times [-1,1]$ lying in the bottom boundary and the top boundary, respectively, of the cobordism represented by $\Omega_{M, M'}$.
This cobordism is a cylinder over $\partial B_r^- \times [-1,1]$.
Indeed, for $t\in [-1,1]$, the disk $D^2 \times \{x\} \subset E(r)$ glued to $\partial D_r' \times \{t\}$ and the disk with two hole $(D_r' \setminus \Int(B_r^- \cup B_r^+)) \times \{t\}$ form an annulus with bases $\partial B_r^- \times \{t\}$ and $\partial B_r^+ \times \{t\}$.
These annuli corresponding to all $t \in [-1,1]$ form the cylinder in question.
When $r$ runs over $1, \dots, m-1$, we get $m-1$ cylinder cobordism.
We may glue each $F(r)$ inside $D_r \times [-1,1]\subset S^3$.
Then locally this is a complement of two cylinders as in Figure \ref{fig:gluing F(r)}.
Note that the union of these spaces corresponds to the identification of the cylindrical boundaries.
In Figure \ref{fig:identification of boundaries}, the space described on the left is the compliment of cylinders. 
The second space is $\partial B_r^- \times [-1,1] \times [0,1]$.
The inner boundary corresponds to $\partial B_r^- \times [-1, 1] \times \{0\}$ and the outer boundary is $\partial B_r^+ \times [-1,1]\times \{1\}$.
For each $s\in[0,1]$, the cylinder $\partial B_r^-\times [-1,1]\times \{s\}$ is glued to the first space. 
(The red circles indicate where to glue and the blue line indicate the interval $[0,1]$.)
This gluing of cylinder corresponds to identifying the time $s$ circles of the cylindrical boundaries.
\begin{figure}[h]
\center
\includegraphics[width=4in]{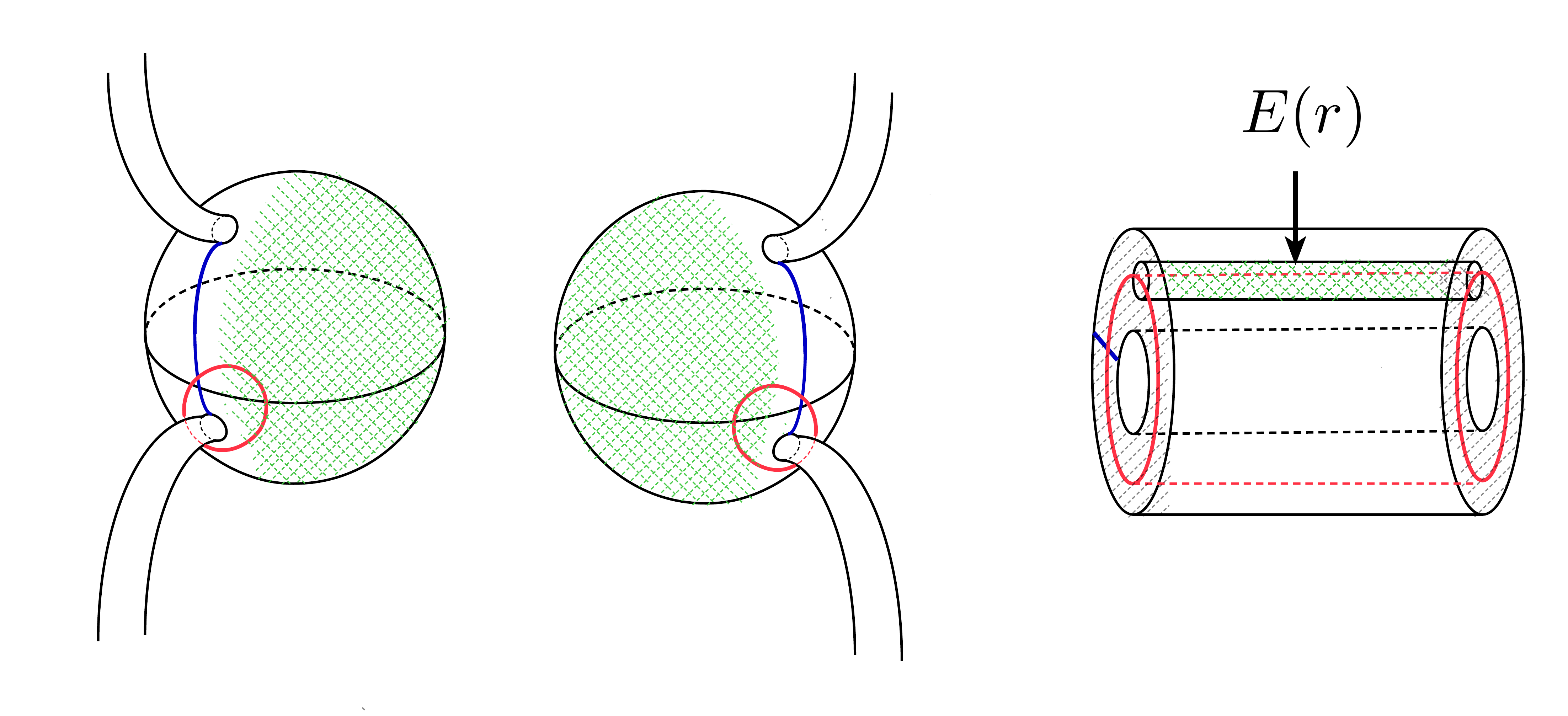}
\caption{Identification of boundaries}
\label{fig:identification of boundaries}
\end{figure}

Thus the surgery along framed link in $\omega_{m-1}$ is the same as cutting out regular neighborhoods of the middle uncolored vertical bands of $\Omega_{M}^{-} \otimes {^{-}\Omega_{M'}}$ and identifying the boundaries (after absorbing the small top and bottom cylindrical parts into the top and bottom boundaries by isotopy respectively.)

\end{proof}

Now that we obtained the ribbon graph operation for horizontal gluing, we use it to study the behavior of the assignment $\X$ under horizontal gluing.
Let $[M_1]: \Sigma_{t_1} \Rightarrow \Sigma_{t_2}: \nstand{l} \to \nstand{m}$ and $[M_2]:\Sigma_{s_1} \Rightarrow \Sigma_{s_2}: \nstand{m} \to \nstand{n}$ be 2-morphisms of $\Co$ that can be glued horizontally.
Recall  the canonical isomorphisms $u_1:\X(\Sigma_{t_1})\circ\X(\Sigma_{s_1}) \to \X(\Sigma_{t_1} \circ \Sigma_{s_1}) $ and  $u_2:\X(\Sigma_{t_2})\circ \X(\Sigma_{s_2}) \to \X(\Sigma_{t_2} \circ \Sigma_{s_2})$ given in the proof of Proposition \ref{prop:2-functor on 1-morphisms}.
The next lemma shows that these isomorphisms commute with 2-homomorphisms $\X(M_1 \circ M_2)$ and $\X(M_1) \circ \X(M_2)$.

\begin{prop}[Horizontal composition]\label{prop:2horizontal}

Let $[M_1]: \Sigma_{t_1} \Rightarrow \Sigma_{t_2}: \nstand{l} \to \nstand{m}$ and $[M_2]:\Sigma_{s_1} \Rightarrow \Sigma_{s_2}: \nstand{m} \to \nstand{n}$ be 2-morphisms of $\Co$.
Then we have $\X(M_1\circ M_2)u_1=u_2(\X(M_1)\circ \X(M_2))$.

\end{prop}
\begin{proof}
Let $\Omega_1=\Omega_{M_1}$ and $\Omega_2=\Omega_{M_2}$ be special ribbon graphs representing $M_1$ and $M_2$, respectively, so that no surgery links are tangled with the rightmost vertical band of $\Omega_{1}$ and the leftmost vertical band of $\Omega_{2}$.
Then the special ribbon graph $\Omega=\Omega_{M_1, M_2}$ represents the horizontally glued cobordism $M_1\circ M_2$ by Lemma \ref{lem:horizontal glue of ribbons}.
From $\Omega_{1}$, remove the rightmost uncolored vertical bands and denote the resulting ribbon graph by $\Omega_{1}^{-}$.
Similarly, remove the leftmost uncolored vertical bands from $\Omega_{2}$ and denote the resulting ribbon graph by ${^{-}\Omega_{2}}$.

For $i=1,2$, we define several notations.
Let $g_i^+$ be the number of cup like bands in $\Omega_{i}$.
Let $g^+$ be the number of cup like band in $\Omega$.
Since there are $m-1$ cup like bands in $\omega_{m-1}$, we have $g^+=g_1^+ + g_2^+ +(m-1)$.
Let $L_i$ be surgery links in $\Omega_{i}$.
The $m-1$ annuli in $\omega_{m-1}$ is denoted by $L_3$.
Denote by $L$ the surgery links of the ribbon graph $\Omega=\Omega_{M_1, M_2}$.
Then $L$ is a disjoint (unlinked) union of $L_1$ and $L_2$ and $L_3$.
Let $\mu, \mu_1, \mu_2$ be the number of components of $L, L_1, L_2$ respectively.
We have $\mu=\mu_1 + \mu_2 + m-1$. 

The $(h,j)$-component homomorphism $\X(M_1 \circ M_2)_{h,j}: \X(\Sigma_{t_1\circ s_1})_{h, j} \to \X(\Sigma_{t_2 \circ s_2})_{h, j}$ can be calculated by Formula (\ref{equ:tau zeta eta extended}).
Let $\zeta$ and $\eta$ be a color for cap like bands and cup like bands of $\Omega$, respectively.
We calculate $(\zeta, \eta)$-block
\[\X_{\zeta}^{\eta}:=\left(\X(M_1 \circ M_2)_{h, j}\right)_{\zeta}^{\eta}.\]
By Formula (\ref{equ:tau zeta eta extended}), we have
\begin{equation*}
\X_{\zeta}^{\eta}=\Delta^{\sigma(L)} \D^{-g^+ -\sigma(L)-\mu} \dim (\eta) \sum_{\lambda \in \col(L)} \dim (\lambda) F_0({_h\Omega_{j}}, \zeta, \eta, \lambda),
\end{equation*}
where ${_h\Omega_{j}}$ is the ribbon graph $\Omega$ with the left vertical bands colored by $h$ and the right vertical bands colored by $j$.
The ribbon graph ${_h\Omega_{j}}$ is the same as $\Omega_{(M\circ M', h, j)}$ in the notation of Formula (\ref{equ:tau zeta eta extended}).
Note that we have $\sigma(L)=\sigma(L_1)+\sigma(L_2)+\sigma(L_3)=\sigma(L_1)+\sigma(L_2)$, since the annuli of $\omega_{m-1}$ are separated, thus $\sigma(L_3)=0$.
We write $\eta=\eta_1+\eta_2+\eta_3$, where $\eta_i$ is a color of the cup like bands of $\Omega_i$ for $i=1, 2$, and $\eta_3$ is a color of the cup like bands of $\omega_{m-1}$.
Then we have $\dim(\eta)=\dim(\eta_1)\dim(\eta_2)\dim(\eta_3)$.
Write analogously $\zeta=\zeta_1 +\zeta_2 +\zeta_3$ for the cap like bands.
Similarly, we decompose a color $\lambda=\lambda_1 + \lambda_2 + \lambda_3$, where $\lambda_i$ is a color of $L_i$ for $i=1,2,3$.
Then we have $\dim(\lambda)=\dim(\lambda_1)\dim(\lambda_2)\dim(\lambda_3)$.

Expressing the ribbon graph $\Omega$ as a morphism of $\Rib$, we have
\[\Omega=(\Omega_{1}^{-} \otimes {^{-}\Omega_{2}})(\id_1\otimes \omega_{m-1} \otimes \id_2).\]
See Figure \ref{fig:Horizontal Special Ribbon }.
Since the operator invariant $F$ is a monoidal functor from $\Rib$, we have $F(_h\Omega_{j}, \zeta, \eta, \lambda)$
\[=\left( F(_h(\Omega_1^-), \zeta_1,\eta_1, \lambda_1) \otimes F(({^-\Omega_2})_j,\zeta_2, \eta_2, \lambda_2) \right) F(\id_1\otimes \omega_{m-1} \otimes \id_2, \zeta_3, \eta_3,\lambda_3).\]
Then $\X_{\zeta}^{\eta}$ is the composition of a morphism $\1 \to \Phi(t_1\circ s_1; \zeta; h,j)$ with
\begin{multline*}
\bigg( \Delta^{\sigma(L_1)} \D^{-g_1^+ -\sigma(L_1)-\mu_1} \dim (\eta_1) \sum_{\lambda_1 \in \col(L_1)} \dim (\lambda) F(_h(\Omega_1^-), \zeta_1, \eta_1, \lambda_1)\\
\otimes
\Delta^{\sigma(L_2)} \D^{-g_2^+ -\sigma(L_2)-\mu_2} \dim (\eta_2) \sum_{\lambda_2 \in \col(L_2)} \dim (\lambda_2) F(({^-\Omega_2})_j, \zeta_2, \eta_2, \lambda_2) \bigg)\\ 
\Delta^{\sigma(L_3)} \D^{-(m-1) -\sigma(L_3)-(m-1)} \dim (\eta_3) \sum_{\lambda_3 \in \col(L_3)} \dim (\lambda_3) F(\id_1\otimes \omega_{m-1} \otimes \id_2, \zeta_3, \eta_3, \lambda_3).
\end{multline*}

We compute the last term.
First, since $\sigma(L_3)=0$, the last term reduces to
\begin{equation}
\D^{-2(m-1)} \dim (\eta_3) \sum_{\lambda_3 \in \col(L_3)} \dim (\lambda_3) F(\id_1\otimes \omega_{m-1} \otimes \id_2, \zeta_3, \eta_3, \lambda_3).
\end{equation}
We claim that the sum is zero unless $\zeta_3=\eta_3$.
If  $\zeta_3=\eta_3$, then  the sum is equal to the operator invariant of $\id_{\zeta}$.
Here $\id_{\zeta}$ is vertical bands whose colors are determined according to $\zeta$.
We postpone the proof of this claim.
We will prove this lemma as a consequence of some graphical calculations.
See Lemma \ref{lem:claim} below.
Assuming the claim, we complete the current proof.

To show $u_2\X(M_1\circ M_2)=(\X(M_1)\circ \X(M_2))u_1$, it suffices to show
\begin{equation}\label{equ:ugg=gu}
\X(M_1\circ M_2)_{h,j}u_1=u_2(\X(M_1)_{h,i}\circ \X(M_2)_{i,j}),
\end{equation}
where $u_i$ is an isomorphism in Lemma \ref{lem:sum over simple 2}.
The left hand side of (\ref{equ:ugg=gu}) is equal to
 \begin{equation}\label{equ:g (Omega Omega)}
\biggr[ \bigoplus_{\zeta, \eta} \Delta^{\sigma(L)}\D^{-g^+ - \sigma(L) - \mu} \dim (\eta) \sum_{\lambda} \dim (\lambda) F_0(_h \Omega_j, \zeta, \eta, \lambda) \biggr]u_1,
\end{equation}
where $\zeta \in I^{t_1\circ s_1}$ and $\eta \in I^{t_2 \circ s_2}$ decompose as $\zeta=\zeta_1+\zeta_2+\zeta_3$ and $\eta=\eta_1+\eta_2+\eta_3$ with notations as above.
By the claim we may assume that $\zeta_3=\eta_3$.
Hence we can write the equation (\ref{equ:g (Omega Omega)}) as follows.

\begin{multline*}
 \bigoplus_{i\in I^{m-1} } \bigoplus_{\zeta_1, \zeta_2} \bigoplus_{\eta_1, \eta_2} \\\biggl[
\Delta^{\sigma(L_1)} \D^{-g_1^+ -\sigma(L_1)-\mu_1} \dim (\eta_1) \sum_{\lambda_1 \in \col(L_1)} \dim (\lambda) F(_h(\Omega_1^-)_i, \zeta_1, \eta_1, \lambda_1)\\
\otimes
\Delta^{\sigma(L_2)} \D^{-g_2^+ -\sigma(L_2)-\mu_2} \dim (\eta_2) \sum_{\lambda_2 \in \col(L_2)} \dim (\lambda_2) F(_i(^-\Omega_2)_j, \zeta_2, \eta_2, \lambda_2)\biggr]u_1.
\end{multline*}
The above equation is further equal to
\begin{multline*}\label{equ:big big bigoplus 2}
u_2\bigoplus_{i\in I^{m} } \bigoplus_{\zeta_1, \zeta_2} \bigoplus_{\eta_1, \eta_2} \\ \biggl[
\Delta^{\sigma(L_1)} \D^{-g_1^+ -\sigma(L_1)-\mu_1} \dim (\eta_1)  \sum_{\lambda_1 \in \col(L_1)} \dim (\lambda) F(_h(\Omega_1)_i, \zeta_1, \eta_1, \lambda_1)\\
\otimes
\Delta^{\sigma(L_2)} \D^{-g_2^+ -\sigma(L_2)-\mu_2} \dim (\eta_2) \sum_{\lambda_2 \in \col(L_2)} \dim (\lambda_2) F(_i(\Omega_2)_j, \zeta_2, \eta_2, \lambda_2)\biggr].
\end{multline*}
To see this, note that as a graphical calculation $u_2$ connects the top of the rightmost band of $\Omega_1$ and the leftmost band of $\Omega_2$ by a cap like band as in Figure \ref{fig:ugg-gu}.
Since no surgery links are tangled with those bands, we can push down the cap like bands.
This explain the equality of Figure \ref{fig:ugg-gu}.
\begin{figure}
\center
\includegraphics[width=3in]{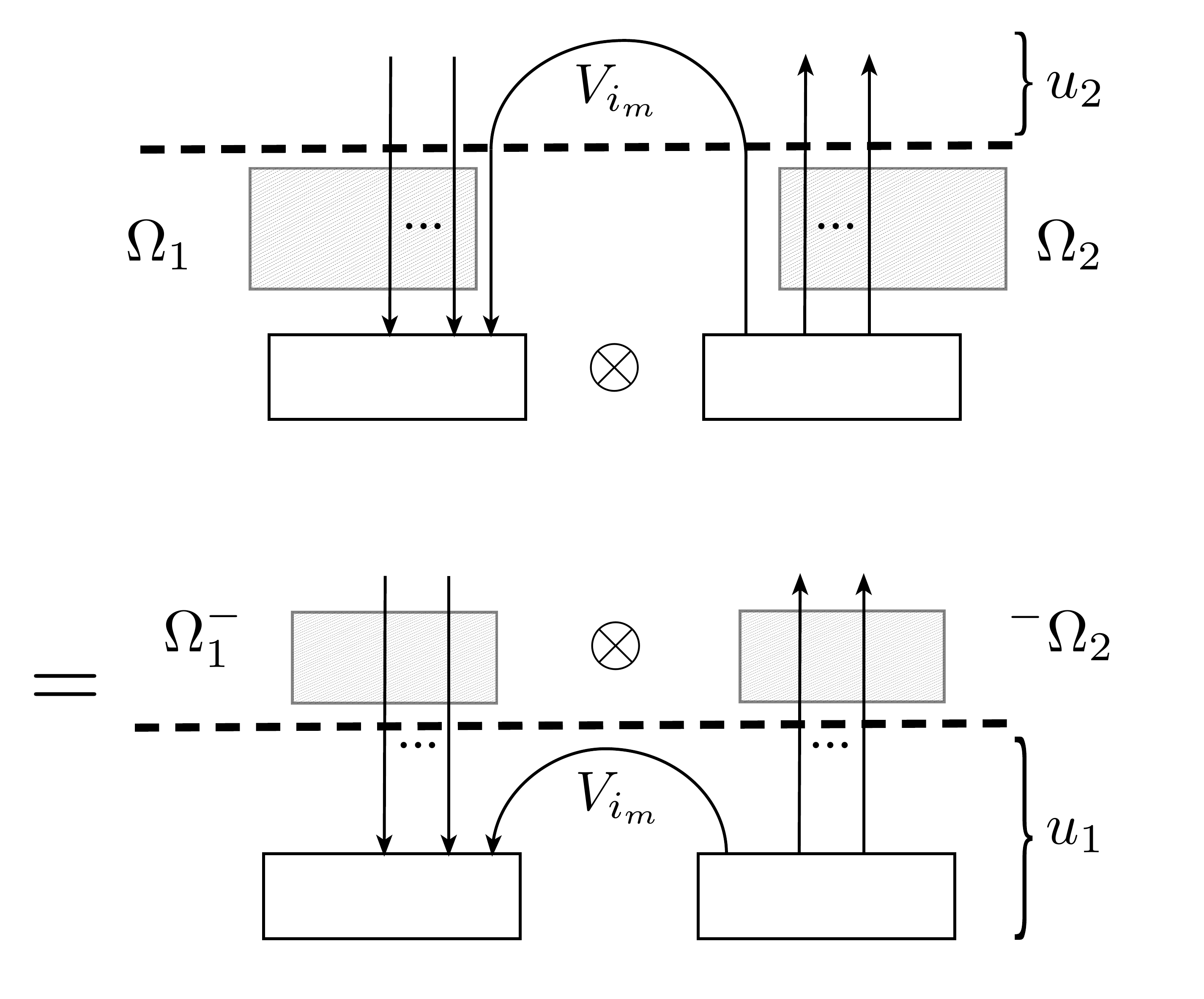}
\caption{The graphical calculation for (\ref{equ:ugg=gu})  }
\label{fig:ugg-gu}
\end{figure}

Finally by Formula (\ref{equ:tau zeta eta extended}),  the above equation is equal to
\[ (\X(M_1)_{h,i} \circ \X(M_2))_{i,j} u_1.\]
Thus the proof is complete assuming the claim, which we prove below.
\end{proof}

In the following graphical calculations, the equality with dot $\stackrel{\bullet}{=}$ means the equality after applying the operator invariant functor $F$ to the ribbon graphs.
Consider the ribbon graphs in Figure \ref{fig:Fig310}.
The label $i$ is an arbitrary element of $I$.
\begin{figure}[h]
\center
\includegraphics[width=4in]{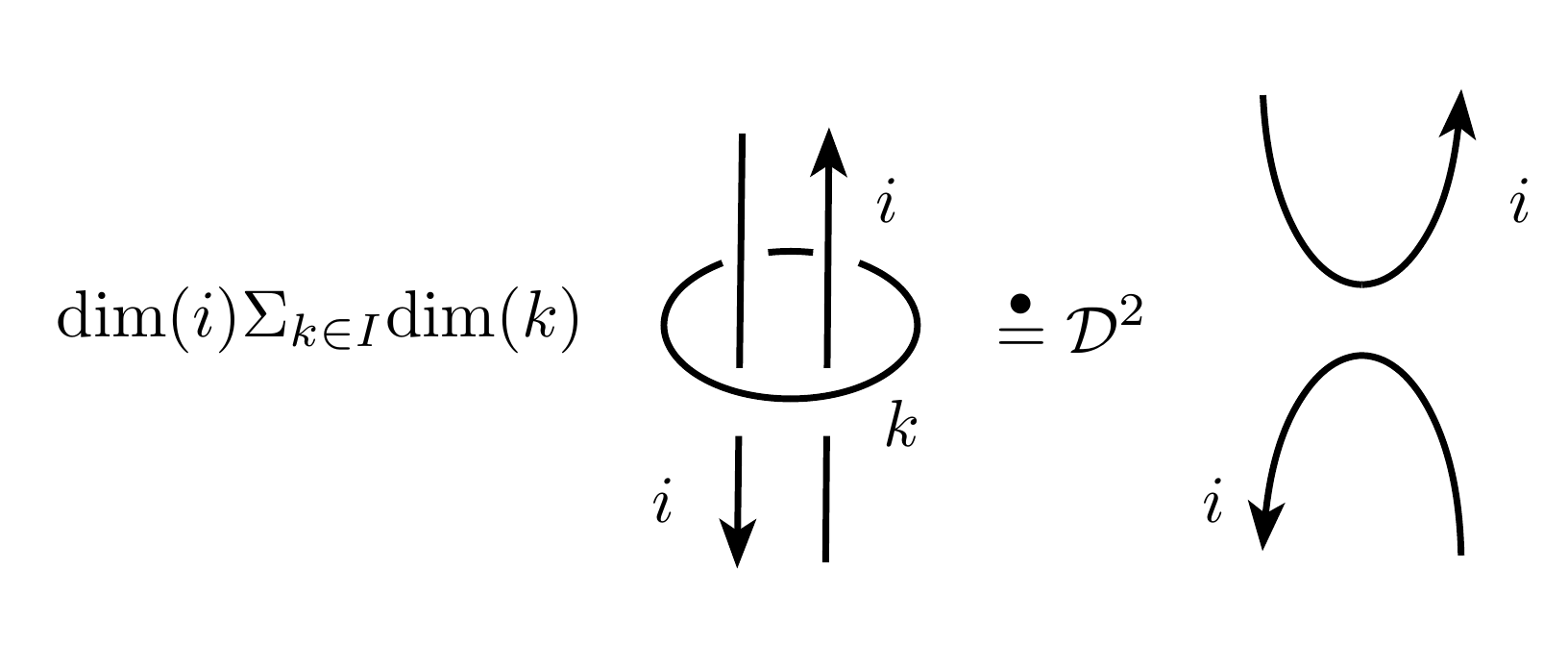}
\caption{}
\label{fig:Fig310}
\end{figure}
\begin{lemma}\label{lem:Fig310}
The equality in Figure \ref{fig:Fig310} holds.
If the color of the left vertical strand is replaced with $j\neq i$, then the sum on the left hand side is equal to $0$.
\end{lemma}
\begin{proof}
See \cite[Section II 3 p.98]{Turaev10}
\end{proof}

\begin{lemma}\label{lem:claim}
For each ribbon graph $\omega_n$, let $\zeta$ and $\eta$ be sequences of colors of the bottom and top rainbow like bands of $\omega_n$.
Then we have
\begin{align}\label{equ:F of omega}
&\D^{-2n} \dim (\eta) \sum_{\lambda \in \col(L)} \dim (\lambda) F(\omega_{n}, \zeta, \eta, \lambda)\\
& \stackrel{\bullet}{=} \begin{cases} \id_{\zeta} &\mbox{if } \zeta=\eta \\ \notag
0 & \mbox{if } \zeta \neq \eta. \end{cases} 
\end{align}

\end{lemma}
\begin{proof}
A part of the ribbon graph $\omega_n$ that consists of an annulus and a pair of a cup like band and a cap like band can be deformed so that it contains the ribbon graphs that appeared on the left hand side of the equation in Figure \ref{fig:Fig310}.
This deformation is depicted in the first equality in Figure \ref{fig:graphical calculation 2}.
It follows from Lemma \ref{lem:Fig310} that the sum on the left hand side of (\ref{equ:F of omega}) is zero unless $\zeta=\eta$.
If $\zeta=\eta$, then the calculation in Figure \ref{fig:graphical calculation 2} shows that each annuli part gives rise to a factor $\D^2$ and the ribbon becomes the $2n$ vertical bands.

\begin{figure}[h]
\center
\includegraphics[width=7in]{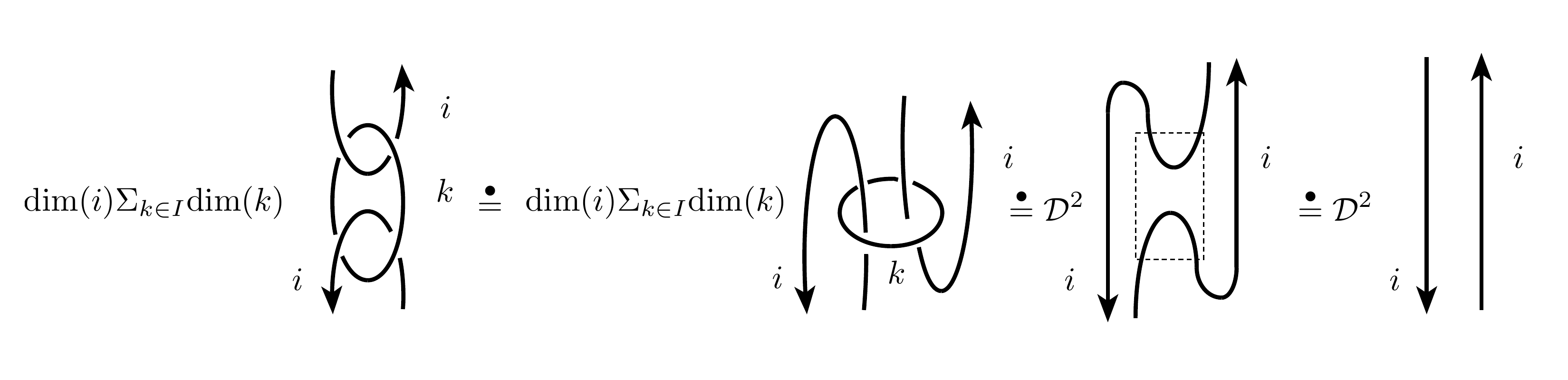}
\caption{}
\label{fig:graphical calculation 2}
\end{figure}
\end{proof}

To wrap up this section, we state here the main theorem (Theorem \ref{thm:main theorem}) again and complete its proof.
\begin{thm}
The assignment $\X$ is a projective pseudo 2-functor from the 2-category $\Co$ of decorated cobordisms with corners to the Kapranov-Voevodsky 2-vector spaces $\KV$.
\end{thm}
\begin{proof}
We check the conditions (1)$\sim$(4) and (M.1), (M.2) of a projective pseudo 2-functor in Section \ref{sec:2-functor}.
(A projectivity is defined in Definition \ref{def:projective functor}.)
In the current case $(F, \phi)$ in the notation of Section \ref{sec:2-functor} is $(\X, u)$, where $u$ is the map used in the proof of Proposition \ref{prop:2-functor on 1-morphisms}.
The condition (1) is just the definition of $\X$.
The condition (2) follows from Proposition \ref{prop:vertical projective functor}.
The condition (3) as well as (M.2) is satisfied since we just use formal identities.
 The condition (4) follows from Proposition \ref{prop:2-functor on 1-morphisms} and \ref{prop:2horizontal}.
 Finally the condition (M.1) follows since the isomorphism $u$ in Lemma \ref{lem:sum over simple 2} satisfies the analogue diagram.

\end{proof}

\subsection{The extended TQFT $\X
$}\label{sec:the extended tqft}
\begin{Definition}
An \textit{extended TQFT} is a projective pseudo 2-functor from $\Co$ to $\KV$.
 An extended TQFT \textit{extends} the Reshetikhin-Turaev TQFT if when it is restricted to the category $\Co(\lempty, \rempty)$ of cobordisms without corners, it is the Reshetikhin Turaev TQFT.
\end{Definition}

Our candidate for an extended TQFT that extends the Reshetikhin-Turaev TQFT is the projective pseudo 2-functor $\X$. 
By definition, the 2-functor $\X$ is an extended TQFT.

\begin{prop}
The extended TQFT $\X$ extends the Reshetikhin-Turaev TQFT.
\end{prop}
\begin{proof}
Suppose $(M, \bottomb M, \topb M)$ is a cobordism without corners.
Then $M$ is represented by a special ribbon graph with a bottom type
$t^-=(0,0; a_1, \dots, a_p)$ and a top type $t^+=(0,0; b_1, \dots, b_q)$.
Since there are no left and right circles on boundary surfaces of $M$, 
$\X(\bob M)$ is a $(1 \times 1)$ 2-matrix and we can canonically identify the 2-matrix with its only entry $\bigoplus_{i \in I^{t^-}}\Hom(\1, \Phi(t^-; i))$.
This module is what the RT TQFT assigns to $\bob M$.
Similarly for $\tob M$.
Then we can identify $\X(M)$ as a homomorphism from the module $\X(\bob M)$ to the module $\X(\tob M)$, which is the same as the RT  TQFT by definition.
\end{proof}

\section{Comments}\label{sec:comments}
\subsection{A 2-category of Special Ribbon Graphs}
As we have seen, the construction of the projective pseudo 2-functor $\X$ from $\Co$ to $\KV$ extensively depends on the technique of representing a cobordism by a special ribbon graphs.
On the level of objects, the 2-functor $X$ extracts just the number of components for each  1-manifold.
On the 1-morphisms, the 2-functor $X$ reads the type of each surface and outputs the projective module constructed only from the information of the type.
We can hence consider the 2-category $\srg$ whose objects are integers and whose 1-morphisms are decorated types and whose 2-morphisms are special ribbon graphs.
we could have defined $\X$ alternatively by the composition of a 2-functors from $\Co$ to $\srg$ and a 2-functor from $\srg$ and $\KV$.
However, to make it meaningful we need to define composition in $\srg$ independently of $\Co$.
This would digress from the main stream of our argument and thus we did not choose to do so.
\begin{center}
    \begin{tabular}{ | l | l | l | }
    \hline
    & $\Co$ & $\srg$ \\ \hline 
            & Geometric realization & Combinatorial data   \\ \hline
   Objects  & Standar circles & Integers
    \\ \hline
   1-morphisms & Standard surfaces & Decorated types   \\ \hline
    2-morphisms & Classes of decorated cobordisms with corners & Special ribbon graphs   \\
    \hline
    \end{tabular}
\end{center}

\subsection{A connection to other work}
We chose that our extended TQFT takes values in $\KV$.
There are 2-functors 2-$\Vect \to \Bimod \to \Cat$.
Here $\Bimod$ is a 2-category whose objects are $K$-algebra, where $K$ is the ground ring, and whose 1-morphisms are  bimodules and 2-morphisms are equivalence classes of $K$-homomorphisms.
The 2-category $\Cat$ consists of categories for objects, functors for 1-morphisms, and natural transformations for 2-morphisms.
Thus our extended TQFT can also take values in $\Bimod$ or $\Cat$ and this make a connection with the work of \cite{DSS}.

The 2-functors can be constructed as follows.
First let us define the 2-functor 2-$\Vect \to \Bimod$.
On object level, to each object $n \in \Z$ of 2-$\Vect$ we assign $K^n$.
On 1-morphism level, to each $m\times n$ 2-matrix $(V_{ij})_{ij}$ we assign a $(K^m, K^n)$-bimodule $\bigoplus_{i,j}V_{ij}$, where the bimodule structure is induced by the multiplication of $1 \times m$ matrix from the left of $(V_{ij})_{ij}$ and the multiplication of $n \times 1$ matrix from the right of $(V_{ij})_{ij}$. On 2-morphism level, if $(T_{ij})$ is a 2-morphism from $(V_{ij})$ to $(W_{ij})$, we assign $\bigoplus_{i,j} T_{i,j}: \bigoplus_{i,j}V_{ij} \to \bigoplus_{i,j}W_{ij}$.
These assignments are easily seen to be a 2-functor.

The 2-functor from $\Bimod$ to $\Cat$ assigns on object level, to each $K$-algebra $A$, the category of right $A$-modules whose objects are right $A$-modules and morphisms are homomorphisms. On 1-morphism level, the assignment is induced by tensoring  a bimodule from the right. On 2-morphism level, natural transformations are induced by  homomorphisms of bimodules.
Composing this 2-functor with the extended TQFT $\X: \Co \to 2\mbox{-}\Vect$, we have a 2-functor from the 2-category of cobordisms with corners to the 2-category $\Cat$.

\section{Appendix: Bicategories}\label{sec:appendix:bicategory}
Here we review the definition of bicategories.
The following definitions of a bicategory and a pseudo 2-functor are excerpted from the paper \cite{MR0220789}.

\subsection{Bicategories}
A \textit{bicategory} $\underbar{S}$ is determined by the following data:
\begin{enumerate}
\item A set $\underbar{S}_0 =\Ob(S)$ called set of objects of $\underbar{S}$.
\item For each pair $(A, B)$ of objects, a category $\underbar{S}(A, B)$.

An object $S$ of $\underbar{S}(A,B)$ is called a \textit{morphism} of $S$, and written $A \xrightarrow{S} B$; the composition sign $\circ$ of maps in $\underbar{S}(A,B)$ will usually be omitted.
A map $s$ from $S$ to $S'$ will be called a \textit{2-morphism} and written $s: S \Rightarrow S'$, or better 
will be represented by:
\begin{center}
\begin{tikzpicture}
\node (A) at (-1,0) {$A$};
\node (B) at (1,0) {$B$};
\node at (0,0) {\rotatebox{270}{$\Rightarrow$}};
\path[->,font=\scriptsize,>=angle 90] node[right]{$s$}
 (A) edge [bend left] node[above] {$S$} (B)
     edge [bend right] node[below] {$S'$} (B);
\end{tikzpicture}
\end{center}

\item For each triple $(A, B, C)$ of objects of $\underbar{S}$, a \textit{composition functor}:
\[c(A, B,C): \underbar{S}(A,B) \times \underbar{S}(B,C) \to \underbar{S}(A,C).\]
We write $S\circ T$ and $s\circ t$ instead of $c(A,B,C)(S,T)$ and $c(A,B,C)(s,t)$ for $(S,T)$ and $(s,t)$ objects and maps of $\underbar{S}(A,B) \times \underbar{S}(B,C)$, and abbreviate $\id_{S} \circ t$ and $s\circ \id_{T}$ into $S\circ t$ and $s\circ T$.
This composition corresponds to the pasting:
\item For each object $A$ of $\underbar{S}$ an object $I_A$ of $\underbar{S}(A,A)$ called an \textit{identity morphism} of $A$.The identity map of $I_A$ is $\underbar{S}(A,A)$ is denoted $i_A:I_A \Rightarrow I_A$ and called an \textit{identity 2-morphism} of $A$.

\item For each quadruple $(A,B,C,D)$ of objects of $\underbar{S}$, a natural isomorphism $a(A, B, C, D)$, called an \textit{associativity isomorphism}, between the two composite functors bounding the diagram:
\begin{center}
\begin{tikzcd}[column sep=3cm]
\underbar{S}(A, B) \times \underbar{S}(B, C) \times \underbar{S}(C, D) \arrow{r}{\id \times c(B, C, D)} \arrow{d}[swap]{c(A,B,C) \times \id} 
               &\underbar{S}(A,B) \times \underbar{S}(B,D) \arrow{d}{c(A,B,D)}\\
\underbar{S}(A,C) \times \underbar{S}(C,D) \arrow[Rightarrow]{ru}{a(A, B, C,D)} \arrow{r}{c(A,C,D)} & \underbar{S}(A,D) 
\end{tikzcd}
\end{center}
Explicitly:
\[a(A,B,C,D): c(A,C,D) \circ ( c(A, B, C)\times \id) \to c(A, B,D) \circ (\it \times c(B,C,D)).\]
If $(S,T,U)$ is an object of $\underbar{S}(A, B) \times \underbar{S}(B, C) \times \underbar{S}(C, D)$ the isomorphism
\[a(A, B,C,D)(S,T,U): (S \circ T) \circ U \xrightarrow{\sim} S \circ (T \circ U)\]
in $\underbar{S}(A, D)$ is called the \textit{component} of $a(A,B,C,D)$ at $(S,T,U)$ and is abbreviated into $a(S,T,U)$ or even $a$, except when confusions are possible.

\item For eahc pair $(A,B)$ of objects of $\underbar{S}$, two natural isomorphisms $l(A,B)$ and $r(A,B)$, called  \textit{left} and \textit{right} identities, between the functors bounding the diagrams:

\begin{center}
\begin{tikzcd}[column sep=3cm]
1\times \underbar{S}(A, B) \arrow{r}{I_A \times \id} \arrow{d}[swap]{\mbox{canonical}} 
               &\underbar{S}(A,A) \times \underbar{S}(A,B) \arrow{d}{c(A,A,B)}\\
\underbar{S}(A,B)  \arrow[Rightarrow]{ru}{l(A,B)} \arrow{r}{=} & \underbar{S}(A,B) 
\end{tikzcd}
\end{center}

\begin{center}
\begin{tikzcd}[column sep=3cm]
 \underbar{S}(A, B)\times 1  \arrow{r}{\id \times I_B} \arrow{d}[swap]{\mbox{canonical}} 
               &\underbar{S}(A,B) \times \underbar{S}(B,B) \arrow{d}{c(A,B,B)}\\
\underbar{S}(A,B)  \arrow[Rightarrow]{ru}{l(A,B)} \arrow{r}{=} & \underbar{S}(A,B) 
\end{tikzcd}
\end{center}

If $S$ is an object of $\underbar{S}(A, B)$, the isomorphism, component at $S$ of $l(A,B)$,
\[l(A,B)(S): I_A \circ S \xrightarrow{\sim} S \]
is abbreviated into $l(S)$ or even $l$, and similarly we write:
\[r=r(S)=r(A,B)(S): S \circ I_B \xrightarrow{\sim} S.\]
\end{enumerate}
The families of natural isomorphisms $a(A,B,C,D)$, $l(A,B)$ and $r(A,B)$ are furthermore required to satisfy the following axioms:

\begin{enumerate}
\item[(A.C.)] Associativity coherence: If $(S, T,U,V)$ is an object of 
\[\underbar{S}(A,B)\times \underbar{S}(B,C) \times \underbar{S}(C,D)\times \underbar{S}(D,E)\]
 the following diagram commutes:
 
\begin{center} 
\begin{tikzpicture}
\matrix (m) [matrix of nodes, column sep=3em, row sep=1.5em]
{ $((S\circ T) \circ U)\circ V$ &  & $(S\circ (T\circ U)) \circ V$ \\
$(S\circ T) \circ (U \circ V)$ &    &  $S \circ ((T \circ U) \circ V)$ \\
  & $S\circ (T \circ (U \circ V))$ & \\};

\path[->, font=\scriptsize]
(m-1-1) edge node[above] {$a(S,T,U)\circ \id$} (m-1-3);

\path[->, font=\scriptsize]
(m-1-1) edge node[left] {$a(A\circ T, U, V)$} (m-2-1);

\path[->, font=\scriptsize]
(m-1-3) edge node[auto] {$a(S, T\circ U, V)$}
 (m-2-3);

\path[->, font=\scriptsize]
(m-2-1) edge node[below left] {$a(S,T, U \circ V)$} (m-3-2);

\path[->, font=\scriptsize]
(m-2-3) edge node[auto]  {$\id \circ a(T, U, V)$} (m-3-2);

\end{tikzpicture}
\end{center}

\item[(I. C.)] Identity coherence: If $(S, T)$ is an object of $\underbar{S}(A,B) \times \underbar{S}(B,C)$ the following diagram commutes:

\begin{center}
\begin{tikzpicture}
\matrix (m) [matrix of nodes, column sep=3em, row sep=1.5em]
{ $(S\circ I_B) \circ T$ &  & $S\circ (I_B \circ T)$ \\
& $S \circ T$ & \\};

\path[->, font=\scriptsize]
(m-1-1) edge node[above] {$a(S,I_B, T)$} (m-1-3);

\path[->, font=\scriptsize]
(m-1-1) edge node[below left] {$r(S) \circ \id$} (m-2-2);

\path[->, font=\scriptsize]
(m-1-3) edge node[auto] {$\id \circ l(T)$}
 (m-2-2);

\end{tikzpicture}
\end{center}

\end{enumerate}

\subsection{Pseudo 2-functors of bicategories}\label{sec:2-functor}
Let $\underbar{S}=( \underbar{S}_0, c, I, a, l, r)$ and $\bar{\underbar{S}}=( \bar{\underbar{S}_0}, \bar{c}, \bar{I}, \bar{a}, \bar{l}, \bar{r})$ be two bicategories.
A \textit{pseudo 2-functor} $\Phi=(F, \phi)$ from $\underbar{S}$ to $\bar{\underbar{S}}$ is determined by the following:
\begin{enumerate}
\item A map $F: \underbar{S}_0 \to \bar{\underbar{S}_0}, A \mapsto FA$.
\item A family of functors
\[ F(A, B): \underbar{S}(A, B) \to \bar{\underbar{S}}(FA, FB), \quad S \mapsto FS, \quad s \mapsto Fs.\]
\item For each object $A$ of $\underbar{S}$, an arrow of $S(FA, FA)$ (i.e., a 2-cell of $\underbar{S}$)
\[ \phi_A: \bar{I}_{FA} \to F(I_A).\]
\begin{center}
\begin{tikzpicture}
\node (A) at (-1,0) {$FA$};
\node (B) at (1,0) {$FA$};
\node at (0,0) {\rotatebox{270}{$\Rightarrow$}};
\path[->,font=\scriptsize,>=angle 90] node[right] {$\phi_A$}
 (A) edge [bend left] node[above] {$\bar{I}_{FA}$} (B)
     edge [bend right] node[below] {$F(I_A)$} (B);
\end{tikzpicture}
\end{center}

\item A family of natural transformations:
\[\phi(A, B, C): \bar{c}(FA, FB, FC) \circ (F(A, B) \times F(B, C)) \to F(A, C) \circ c(A, B, C). \]
\begin{center}
\begin{tikzcd}[column sep=3cm]
\underbar{S}(A, C) \arrow[leftarrow]{r}{c(A, B, C)} \arrow{d}{F(A, C)} \arrow[Leftarrow]{rd}{\phi(A, B, C)}
               &\underbar{S}(A,B) \times \underbar{S}(B,C) \arrow{d}{F(A, B) \times F(B,C)}\\
\bar{\underbar{S}}(FA,FC) \arrow[leftarrow]{r}{\bar{c}(FA,FB,FC)} & \bar{\underbar{S}}(FA,FB) \times \bar{\underbar{S}}(FB,FC)
\end{tikzcd}
\end{center}
If $(S,T)$ is an object of $\bar{\underbar{S}}(A,B) \times \bar{\underbar{S}}(B,C)$ the $(S,T)$-component of $\phi(A,B,C)$
\[FS \circ FT \xrightarrow{\phi(A,B,C)(S,T)} F(S \circ T) \]
shall usually be abbreviated into $\phi(S, T)$ or even $\phi$.
\end{enumerate}
These data are required to satisfy the following coherence axioms:

\begin{enumerate}
\item[(M. 1)] If $(S, T, U)$ is an object of $\underbar{S}(A, B) \times \underbar{S}(B,C) \times \underbar{S}(C,D)$ the following diagram, where indices $A, B, C, D$ have been omitted, is commutative:

\begin{tikzcd}[column sep=large]
(FS \circ FT) \circ FU \arrow{r}{ \phi(S, T) \circ \id} \arrow{d}{ \bar{a}(FS, FT, FU)}
               &F(S\circ T) \circ FU \arrow{r}{ \phi(S \circ T, U)} & F((S\circ T) \circ U) \arrow{d}{F(a(S,T,U))} \\
FS\circ (FT \circ FU) \arrow{r}{\id \circ \phi(T, U)} & FS \circ F(T \circ U) \arrow{r}{\phi(S, T \circ U)} & F(S \circ (T \circ U))
\end{tikzcd}

\item[(M. 2)] If $S$ is an object of $\underbar{S}(A,B)$ the following diagrams commute:

\begin{tikzcd}
FS \arrow[leftarrow]{r}{Fr} 
               & F(S \circ I_B) \\
FS\circ \bar{I}_{FB} \arrow{u}{\bar{r}} \arrow{r}{ \id \circ \phi_B} & FS \circ FI_B \arrow{u}{\phi(S, I_B)}
\end{tikzcd}
\qquad
\begin{tikzcd}
F(I_A \circ S) \arrow{r}{Fl} 
               & FS \\
FI_A \circ FS  \arrow{u}{\phi(I_A, S)}\arrow[leftarrow]{r}{\phi_A \circ \id} & \bar{I}_{FA} \circ FS \arrow{u}{\bar{l}}
\end{tikzcd}

\end{enumerate}

\subsection{Projective Functors and Projective Pseudofunctors}
Let $\catC$ and $\catD$ be categories.
We introduce the notion of a \textit{projective functor} from $\catC$ to $\catD$.
This notion deals with the anomaly of the Reshetikhin-Turaev TQFT.

\begin{Definition}\label{def:projective functor}
Assume that  the set of morphisms of $\catD$ is an $R$-module for some ring $R$.
The following assignment $F$ is called \textit{projective functor}
\begin{enumerate}
\item For each object $X\in \obj(\catC)$, an object $F(X) \in \obj(\catD)$.
\item For a morphism $f: X\to Y$ in $\catC$, a morphism $F(f): F(X) \to F(Y)$ in $\catD$ satisfying the following conditions:
\begin{enumerate}

\item(Unit) For any object $X$ in $\catC$, the identity morphism $\id_X: X \to X$ in $\catC$ is mapped to the identity morphism $\id_{F(X)}: F(X) \to F(X)$ in $\catD$.

\item(Projectivity) For  composable two morphisms $f$ and $g$ in $\catC$, there exist unique element $k(f,g)$, which is called an \textit{anomaly}, of the ring $R$ such that
\begin{equation}
F(f\circ g)=k(f,g) F(f) \circ F(g).
\end{equation}
\item If one of $f$ and $g$ above is the identity morphism, then $k(f, g)$ is the unit element of $R$.
\item (Associativity) For composable three morphisms $f, g, h$ in $\catC$, we have
\begin{equation}
k(f, g\circ h)k(g,h)=k(f \circ g, h)k(f,g)
\end{equation}

\end{enumerate}

\end{enumerate}
\end{Definition}

When every anomaly is the unit, then the notion of a projective functor is the same as the usual notion of a functor.
The notion of a natural transformation between functors can be extended to a natural transformation between projective functors if anomaly factors are the same for both functors.
We call a natural transformation between projective functors \textit{projective natural transformation}
Then, we define a  \textit{projective pseudo 2-functor} by replacing functors and natural transformations in the definition of a pseudo 2-functor with projective functors and projective natural transformations.

\bibliography{reference}{}

\begin{thebibliography}{FHLT10}

\bibitem[Ati88]{MR1001453}
Michael Atiyah.
\newblock Topological quantum field theories.
\newblock {\em Inst. Hautes \'Etudes Sci. Publ. Math.}, (68):175--186 (1989),
  1988.

\bibitem[Bal10]{BalsamII}
Benjamin Balsam.
\newblock {Turaev-viro invariants as an extended tqft II arXiv:1010.1222v1
  [math.QA]}.
\newblock 2010.

\bibitem[Bal11]{BalsamIII}
Benjamin Balsam.
\newblock {Turaev-viro invariants as an extended tqft III arXiv:1012.0560v2
  [math.QA]}.
\newblock 2011.

\bibitem[B{\'e}n67]{MR0220789}
Jean B{\'e}nabou.
\newblock Introduction to bicategories.
\newblock In {\em Reports of the {M}idwest {C}ategory {S}eminar}, pages 1--77.
  Springer, Berlin, 1967.

\bibitem[BK10]{BalsamI}
Benjamin Balsam and Alexander Kirillov.
\newblock {Turaev-viro invariants as an extended tqft arXiv:1004.1533v3
  [math.GT]}.
\newblock pages 1--39, 2010.

\bibitem[CL91]{MR1107480}
John~L. Cardy and David~C. Lewellen.
\newblock Bulk and boundary operators in conformal field theory.
\newblock {\em Phys. Lett. B}, 259(3):274--278, 1991.

\bibitem[CS09]{MR2597733}
Moira Chas and Dennis Sullivan.
\newblock String topology in dimensions two and three.
\newblock In {\em Algebraic topology}, volume~4 of {\em Abel Symp.}, pages
  33--37. Springer, Berlin, 2009.

\bibitem[Dij89]{Dijkgraafthesis}
Robbert Dijkgraaf.
\newblock {\em A geometric approach to two dimensional conformal field theory}.
\newblock PhD thesis, University of Utrecht, 1989.

\bibitem[DSPS]{DSS}
Chris Douglas, Chris Schommer-Pries, and Noah Snyder.
\newblock {3-dimensional topology and finite tensor categories (work in
  progress)}.

\bibitem[DSPVB]{DSPVB}
Chris Douglas, Chris Schommer-Pries, Jamie Vicary, and Bruce Bartlett.
\newblock {work in progress}.

\bibitem[FHLT10]{MR2648901}
Daniel~S. Freed, Michael~J. Hopkins, Jacob Lurie, and Constantin Teleman.
\newblock Topological quantum field theories from compact {L}ie groups.
\newblock In {\em A celebration of the mathematical legacy of {R}aoul {B}ott},
  volume~50 of {\em CRM Proc. Lecture Notes}, pages 367--403. Amer. Math. Soc.,
  Providence, RI, 2010.

\bibitem[FQ93]{MR1240583}
Daniel~S. Freed and Frank Quinn.
\newblock Chern-{S}imons theory with finite gauge group.
\newblock {\em Comm. Math. Phys.}, 156(3):435--472, 1993.

\bibitem[Kas95]{MR1321145}
Christian Kassel.
\newblock {\em Quantum groups}, volume 155 of {\em Graduate Texts in
  Mathematics}.
\newblock Springer-Verlag, New York, 1995.

\bibitem[Kau07]{MR2314100}
Ralph~M. Kaufmann.
\newblock Moduli space actions on the {H}ochschild co-chains of a {F}robenius
  algebra. {I}. {C}ell operads.
\newblock {\em J. Noncommut. Geom.}, 1(3):333--384, 2007.

\bibitem[Kau08]{MR2411420}
Ralph~M. Kaufmann.
\newblock Moduli space actions on the {H}ochschild co-chains of a {F}robenius
  algebra. {II}. {C}orrelators.
\newblock {\em J. Noncommut. Geom.}, 2(3):283--332, 2008.

\bibitem[KL01]{MR1862634}
Thomas Kerler and Volodymyr~V. Lyubashenko.
\newblock {\em Non-semisimple topological quantum field theories for
  3-manifolds with corners}, volume 1765 of {\em Lecture Notes in Mathematics}.
\newblock Springer-Verlag, Berlin, 2001.

\bibitem[KP06]{MR2242677}
Ralph~M. Kaufmann and R.~C. Penner.
\newblock Closed/open string diagrammatics.
\newblock {\em Nuclear Phys. B}, 748(3):335--379, 2006.

\bibitem[KV94]{KV1994}
M.~M. Kapranov and V.~A. Voevodsky.
\newblock {$2$}-categories and {Z}amolodchikov tetrahedra equations.
\newblock In {\em Algebraic groups and their generalizations: quantum and
  infinite-dimensional methods ({U}niversity {P}ark, {PA}, 1991)}, volume~56 of
  {\em Proc. Sympos. Pure Math.}, pages 177--259. Amer. Math. Soc., Providence,
  RI, 1994.

\bibitem[LP08]{MR2395583}
Aaron~D. Lauda and Hendryk Pfeiffer.
\newblock Open-closed strings: two-dimensional extended {TQFT}s and {F}robenius
  algebras.
\newblock {\em Topology Appl.}, 155(7):623--666, 2008.

\bibitem[Lur09]{Lurie2009}
Jacob Lurie.
\newblock {On the Classification of Topological Field Theories (Draft) arXiv :
  0905 . 0465v1 [math . CT] 4 May 2009}.
\newblock pages 1--111, 2009.

\bibitem[RT91]{MR1091619}
N.~Reshetikhin and V.~G. Turaev.
\newblock Invariants of {$3$}-manifolds via link polynomials and quantum
  groups.
\newblock {\em Invent. Math.}, 103(3):547--597, 1991.

\bibitem[Seg88]{MR981378}
G.~B. Segal.
\newblock The definition of conformal field theory.
\newblock In {\em Differential geometrical methods in theoretical physics
  ({C}omo, 1987)}, volume 250 of {\em NATO Adv. Sci. Inst. Ser. C Math. Phys.
  Sci.}, pages 165--171. Kluwer Acad. Publ., Dordrecht, 1988.

\bibitem[SP09]{MR2713992}
Christopher~John Schommer-Pries.
\newblock {\em The classification of two-dimensional extended topological field
  theories}.
\newblock ProQuest LLC, Ann Arbor, MI, 2009.
\newblock Thesis (Ph.D.)--University of California, Berkeley.

\bibitem[Tur10]{Turaev10}
Vladimir~G. Turaev.
\newblock {\em Quantum invariants of knots and 3-manifolds}, volume~18 of {\em
  de Gruyter Studies in Mathematics}.
\newblock Walter de Gruyter \& Co., Berlin, revised edition, 2010.

\bibitem[TV10]{Turaev2010}
Vladimir Turaev and Alexis Virelizier.
\newblock {On two approaches to 3-dimensional TQFTs arXiv:1006.3501v5 [math.GT]
  22 Sep 2010}.
\newblock pages 1--72, 2010.

\end{thebibliography}
\bibliographystyle{alpha.bst}

\end{document}